\documentclass{article}
\usepackage[a4paper, lmargin=2cm, rmargin=2cm, tmargin=2.5cm, bmargin=2.5cm, marginpar=3.0cm]{geometry}

\usepackage{amsmath}
\usepackage{mathrsfs}
\usepackage{mathtools}
\usepackage{amssymb}
\usepackage{bbm}
\usepackage{bm}
\allowdisplaybreaks[1]
\usepackage{tikz}
\usepackage{pgfplots}
\pgfplotsset{compat=1.18} 
\usetikzlibrary{positioning}
\usetikzlibrary{calc}
\usetikzlibrary{arrows.meta}
\usetikzlibrary{decorations.text}
\usetikzlibrary{arrows,positioning}
\usepackage{tikz-cd}
\usepackage{subcaption}

\usepackage{parskip}
\usepackage{physics}
\usepackage{amsthm}

\usepackage[colorlinks, citecolor=hanblue,linkcolor=mordantred19,urlcolor=darkgreen]{hyperref}
\usepackage{color}
\definecolor{hanblue}{rgb}{0.27, 0.42, 0.81}
\definecolor{mordantred19}{rgb}{0.68, 0.05, 0.0}
\definecolor{darkgreen}{rgb}{0.0, 0.38, 0.12}
\definecolor{red}{rgb}{0.8, 0.0, 0.0}
\definecolor{green}{rgb}{0.0, 0.5, 0.0}
\definecolor{lightergreen}{rgb}{0.0, 0.6, 0.0}
\usepackage{xcolor}

\usepackage{array}

\newcolumntype{C}[1]{>{\centering\arraybackslash}m{#1}} 
\newcolumntype{K}[1]{>{\centering\arraybackslash}p{#1}}

\usepackage{thmtools}
\declaretheorem[numberwithin=section]{theorem}
\declaretheorem[numberwithin=section]{lemma}

\declaretheorem[numberwithin=section]{corollary}

\declaretheoremstyle[
  headfont=\normalfont\bfseries,
  bodyfont=\normalfont,
  spaceabove=1em plus 0.75em minus 0.25em,
  qed={$\lozenge$},
]{definitionqed}

\declaretheorem[style=definitionqed,numberwithin=section]{definition}

\declaretheorem[style=definitionqed,numberwithin=section]{remark}

\DeclarePairedDelimiterX{\set}[1]{\lbrace}{\rbrace}{\def\given{\;\delimsize\vert\;\allowbreak}#1}

\DeclareMathOperator{\spn}{span}
\DeclareMathOperator{\twin}{Twin}
\DeclareMathOperator*{\argmin}{arg\,min}
\newcommand{\B}{\mathcal{B}}
\newcommand{\di}{\diamond}

\newcommand{\F}{\mathcal{F}}
\renewcommand{\H}{\mathcal{H}}
\renewcommand{\L}{\mathcal{L}}
\newcommand{\M}{\mathcal{M}}
\newcommand{\N}{\mathbb{N}}
\newcommand{\U}{\mathcal{U}}
\newcommand{\V}{\mathcal{V}}
\renewcommand{\P}{\mathcal{P}}
\newcommand{\R}{\mathbb{R}}

\numberwithin{equation}{section}

\newcommand\blfootnote[1]{%
  \begingroup
  \renewcommand\thefootnote{}\footnote{#1}%
  \addtocounter{footnote}{-1}%
  \endgroup
}

\allowdisplaybreaks

\title{Vector-Valued Reproducing Kernel Banach Spaces for Neural Networks and Operators}
\author{Sven Dummer$^{\ast\dagger}$, Tjeerd Jan Heeringa\footnote{Equal contribution.}$^{\ \dagger}$ and Jos\'e A. Iglesias\thanks{Mathematics of Imaging \& AI, Department of Applied Mathematics, University of Twente, the Netherlands\\ \hspace*{5mm}(\texttt{t.j.heeringa{@}utwente.nl, s.c.dummer{@}utwente.nl, jose.iglesias{@}utwente.nl}).}}
\date{}
\begin{document}

\maketitle

\begin{abstract}
    Recently, there has been growing interest in characterizing the function spaces underlying neural networks. While shallow and deep scalar-valued neural networks have been linked to scalar-valued reproducing kernel Banach spaces (RKBS), $\mathbb{R}^d$-valued neural networks and neural operator models remain less understood in the RKBS setting. To address this gap, we develop a notion of adjoint pairs of vector-valued RKBSs (vv-RKBS), which inherently involves an associated reproducing kernel, and prove that every vv-RKBS belongs to such a pair. Our construction extends existing kernel definitions by avoiding restrictive assumptions such as symmetric kernel domains, finite-dimensional output spaces, reflexivity, or separability, while still recovering familiar properties of vector-valued reproducing kernel Hilbert spaces (vv-RKHS). We then show that shallow $\mathbb{R}^d$-valued neural networks are elements of a specific vv-RKBS, namely an instance of an integral vv-RKBS. To also explore the functional structure of neural operators, we analyze the DeepONet and Hypernetwork architectures and demonstrate that they too belong to an integral vv-RKBS. In all cases, we establish a representer theorem, showing that optimization over these function spaces recovers the corresponding neural architectures.
\end{abstract}
\blfootnote{2020 Mathematics Subject Classification (MSC): 46E15, 68T07, 46G10, 46E22, 46B10, 26B40}

\section{Introduction}
Neural networks are ubiquitous in applications such as computer vision, medical imaging, and scientific computing. Although they perform remarkably well, the fundamental mathematical understanding of neural networks is incomplete.

One approach to advance this understanding is to study neural networks from a function space perspective. This perspective characterizes networks as elements of a function space and rigorously analyzes the analytical properties of the space. For instance, representer theorems are investigated that show that neural network architectures solve supervised optimization problems in these infinite-dimensional function spaces. Examples include spaces inspired by variational splines \cite{parhi2021banach, parhi2022kinds} and the variation (or Barron) space \cite{wojtowytsch2020banach, ma2022barron, bach2017breaking, wojtowytsch2022representation, shenouda2024variation, bartolucci2024neural, heeringa2025deep, wang2024hypothesis, bartolucci2023understanding}. Both belong to the broader class of Reproducing Kernel Banach Spaces (RKBSs) \cite{canu2009functional, zhang2009reproducing, georgiev2013construction, lin2022reproducing, ikeda2022koopman}, which strictly generalizes the more familiar Reproducing Kernel Hilbert Spaces (RKHS) \cite{steinwart2024reproducing}. In particular, variation/Barron spaces are examples of integral RKBSs \cite{spek2025duality}. In this paper, we focus on situations in which the outputs belong to a vector space which is possibly infinite-dimensional, and we refer to the resulting function spaces as vector-valued RKBS (vv-RKBS). Given an output vector space $\U$ and input domain $X$, the integral vv-RKBSs are defined via a feature map $\Phi \colon X \to \L(\F; \U)$ with $\F =\M(\Omega; \U)$ the space of regular countably additive $\U$-valued measures of bounded variation on the weight domain $\Omega$. Specifically, the feature map and resulting functions are defined by integrating a feature function $\phi \colon X \times \Omega \to \R$ against the measure $\mu \in \M(\Omega; \U)$:
\begin{equation}
(A_{\Omega \to X}\mu)(x) = \Phi(x)\mu = \int_\Omega \phi(x,w)  d\mu(w).
\label{eq:integral_rkbs_intro}
\end{equation}
In the scalar case $\U = \mathbb{R}$, Spek et al.\ \cite{spek2025duality} discuss the construction of a reproducing kernel for the space $\mathcal{B}$ of functions of the form $A_{\Omega \to X}\mu$. Their construction utilizes a function space $\B^\di$ that is based on functions defined via a measure $\rho \in \mathcal{M}(X)$ and interchanging the roles of the weight and input domains in \eqref{eq:integral_rkbs_intro}
\begin{equation}
(A_{X \to \Omega}\rho)(w) = \int_X \phi(x,w) d\rho(x).
\end{equation}
The reproducing kernel $K \colon X \times \Omega \to \mathbb{R}$ is given by $\phi(x,w)$ and satisfies
\begin{equation}
f(x) = \braket{\phi(x, \cdot)}{f}_\B, \quad g(w) = \braket{g}{\phi(\cdot, w)}_\B
\end{equation}
with $f \in \B$, $g\in \B^\di$, and $\braket{\cdot}_\B \colon \B^\di \times \B \to \R$ a duality pairing. This generalizes the familiar reproducing property from Reproducing
Kernel Hilbert Spaces to the Banach space setting.

Most existing theory focuses on vector-input, scalar-output networks. In contrast, $\R^d$-valued neural networks, and especially those with outputs in a general Banach space, are far less studied. Besides Bartolucci et al.\ \cite{bartolucci2024neural}, the few exceptions \cite{wang2024hypothesis, shenouda2024variation, parhi2022kinds, korolev2022two} either map to $\R^d$ or restrict activation functions. Except for Wang et al.\ \cite{wang2024hypothesis}, these works do not discuss the existence or role of a reproducing kernel, which is central to the RKBS perspective.

This gap is especially relevant for neural operator methods \cite{boulle2024mathematical}, which map functions to functions and naturally require Banach-valued outputs. Examples include the Fourier Neural Operator (FNO) \cite{li2021fourier}, Graph Neural Operator (GNO) \cite{anandkumar2019neural}, Convolutional Neural Operator (CNO) \cite{raonic2023convolutional}, DeepONet \cite{lu2021learning}, POD-DeepONet \cite{lu2022comprehensive}, NOMAD \cite{seidman2022nomad}, and PCA-NET \cite{bhattacharya2021model}. 

Many neural operators can be interpreted as conditional implicit neural representations (INRs) \cite{sitzmann2020implicit, dummer2024rda}. In this framework, a network takes a spatial or spatiotemporal coordinate \(x \in \mathbb{R}^{d_x}\) and a conditioning vector \(z \in \mathbb{R}^{d_z}\) to represent a family of output functions via
\begin{equation}
    f_\theta \colon \mathbb{R}^{d_x + d_z} \to \mathbb{R}^d.
\end{equation}
The DeepONet, POD-DeepONet, NOMAD, and PCA-NET can immediately be interpreted as conditional INRs. The FNO can be viewed as a DeepONet when evaluated on a fixed discretized grid \cite{kovachki2021universal}, and can hence be interpreted as an INR. The CNO \cite{raonic2023convolutional} projects inputs to a bandlimited space, processes them, and upsamples back to another bandlimited space. These projection steps yield a kernel-method interpretation. Methods such as RONOM \cite{dummer2026ronom} and the approach of Batlle et al.\ \cite{batlle2024kernel} explicitly use RKHS structures and can also be seen as kernel-based INRs. This makes the INR structure a natural framework for investigating neural operators from a function-space perspective. 

In this work, we provide a unified framework for these approaches. This framework supports vector-valued Reproducing Kernel Banach Spaces (vv-RKBS) mapping to arbitrary Banach spaces and inherently includes a reproducing kernel.  Specifically, we define a kernel $K \colon X \times \Omega \to \twin(\U, \U^\di)$, where $(\U, \U^\di)$ form a dual pair and $\twin(\U, \U^\di)$ extends the classical use of bounded linear operators $\L(\U)$ in vv-RKHS to the vv-RKBS setting. 

This kernel structure differs from existing vv-RKBS kernel constructions \cite{zhang2013vector, chen2019vector, lin2021multi, combettes2018regularized}, which require $\Omega = X$ and makes them inapplicable to neural networks. For approaches dealing with finite-dimensional outputs $\U = \R^d$ \cite{chen2019vector, lin2021multi}, the kernel maps to the space of matrices $\L(\U)=\L(\R^d)$. Approaches that allow infinite-dimensional outputs \cite{zhang2013vector, combettes2018regularized} instead map to bounded, but not necessarily linear, operators from $\U$ or $\U^*$ to $\U$. The latter does not capture the linear structure of the space $\L(\U)$ that is used in vv-RKHS. Furthermore, these approaches for infinite-dimensional output spaces rely on reflexivity and, in some cases, separability of the RKBS, its feature space $\F$, or the output space. These properties are used to construct the kernel and prove representer theorems. In contrast, our framework imposes no reflexivity or separability assumptions. 

This allows us to handle, for instance, feature spaces $\F$ that are not reflexive. In particular, to connect the vv-RKBS setting with neural networks and neural operators, we consider the integral vv-RKBS introduced by Bartolucci et al.\ \cite{bartolucci2024neural} and defined in \eqref{eq:integral_rkbs_intro} using a feature space of vector-valued measures. Most existing works on $\R^d$- and vector-valued neural networks \cite{wang2024hypothesis, shenouda2024variation, parhi2022kinds, korolev2022two} either map to $\R^d$ or restrict to homogeneous activation functions. The integral vv-RKBS construction avoids these limitations as it supports arbitrary Banach output spaces, allows non-homogeneous activation functions, and still yields representer theorems. Beyond the work of Bartolucci et al., we also provide the explicit kernel structure for these spaces and establish a representer theorem for infinite-dimensional output spaces. In addition, to explore the function space structure of neural operators, we utilize the previously discussed connection with conditional INRs. In particular, we construct an integral vv-RKBS for hypernetworks and DeepONets and prove a representer theorem.  

\subsection{Contributions}
Our main contributions can be summarized as follows:
\begin{itemize}
    \item We extend the kernel definition of vv-RKBSs to allow asymmetric kernel domains and to avoid structural assumptions such as reflexivity and separability. In addition, we demonstrate that several classical properties of vv-RKHSs still extend to this more general setting.
    \item We relate our vv-RKBS pairs and their reproducing kernels to existing formulations in the scalar RKBS setting. In particular, we show that our definition either recovers the corresponding vector-valued extensions directly or, with additional assumptions, recovers stronger versions.
    \item We further develop the integral and neural vv-RKBS framework introduced in Bartolucci et al.\ \cite{bartolucci2024neural} by analyzing its reproducing kernel structure and establishing a general representer theorem.
    \item We embed $\R^d$-valued neural networks, DeepONets, and hypernetworks into the integral vv-RKBS framework and obtain representer theorems for them. Together, these ensure a clear bridge between vector-valued neural architectures and the kernel structure of RKBSs.
\end{itemize}

\subsection{Paper outline}
Section \ref{sec:vv-RKHS} reviews vector-valued reproducing kernel Hilbert spaces (vv-RKHS) for comparison with our vector-valued reproducing kernel Banach space (vv-RKBS) definitions. Section \ref{sec:vv-RKBS_definition} introduces the general definition of a vv-RKBS through bounded point evaluation functionals and feature maps. Section \ref{sec:adjoint_vv-RKBS_definition} then develops the notion of adjoint vv-RKBS pairs, which naturally incorporate reproducing kernels. Their main properties are examined in Section \ref{sec:adjoint_vv-RKBS_properties}, where we show, among other results, that every space defined via bounded point evaluations or feature maps admits a reproducing kernel within some adjoint pair.

Section \ref{sec:integral_and_neural_vv-RKBS} turns to adjoint vv-RKBS pairs tailored to neural networks, namely the integral and neural vv-RKBS pairs. After recalling the scalar-valued setting, we extend to the vector-valued case, establish key properties, and prove a general representer theorem. Finally, Section \ref{sec:vv-NNs_HypNets_DeepONets} specializes this framework to $\R^d$-valued neural networks, DeepONets, and hypernetworks, and provides representer theorems for them. 

In Appendix \ref{sec:notation} we provide a table of notation to facilitate the reading of the paper. Appendix \ref{sec:additionalresults} contains some additional technical results that are used in the main text, but which would otherwise interrupt the narrative. 

\section{Vector-valued Reproducing Kernel Hilbert Spaces} \label{sec:vv-RKHS}
The classical definition of Reproducing Kernel Hilbert Spaces (RKHS) in terms of bounded evaluation functionals can be extended to the vector-valued case.
\begin{definition}[Vector-valued Reproducing Kernel Hilbert space]
Let $\H$ be a Hilbert space of functions over a set $X$ mapping to a Hilbert space $\U$. $\H$ is a vector-valued Reproducing Kernel Hilbert space (vv-RKHS) if point evaluations are bounded functionals, i.e. for all $x\in X$
\begin{equation}
    \norm{f(x)}_\U \leq C_x\norm{f}_\H
\end{equation}
holds for all $f\in \H$ with the constant $C_x\geq 0$ depending on $x$ but not on $f$.
\label{def:RKHS_bounded_point_eval}
\end{definition}
Just like the scalar RKHS, the vv-RKHS has an equivalent definition in terms of a kernel.

\begin{definition}[Kernel definition vv-RKHS]
A Hilbert space $\H$ of functions over a set $X$ mapping to a Hilbert space $\U$ is a vv-RKHS if and only if there exists a kernel $K \colon X \times X \to \L(\U)$ such that for all $x \in X$ and $u \in \U$:
\begin{subequations}
\begin{align}
    & K(x, \cdot)u \in \H \\
    & \braket{u}{f(x)}_\U = \braket{K(x, \cdot)u }{f}_\H
\end{align}
\end{subequations}
\label{def:RKHS_kernel}
\end{definition}

\begin{theorem}
Definitions~\ref{def:RKHS_bounded_point_eval} and \ref{def:RKHS_kernel} are equivalent.
\end{theorem}
\begin{proof}
Consider the bilinear functional
\begin{equation}
    T_x\colon \U\times \H \to \R,\; (u,f)\mapsto \braket{u}{f(x)}_\U.
\end{equation}
Fixing either argument yields a linear functional in the other, and these functionals are bounded since
\begin{equation}
    \braket{u}{f(x)}_\U \leq \norm{u}_\U \norm{f(x)}_\U \leq C_x \norm{f}_\H \norm{u}_\U.
\end{equation}
Hence, the map $f \mapsto T_x(u, f)$ defines a bounded linear functional on the Hilbert space $\H$. By the Riesz representation theorem, there exists $K(x, \cdot)u \in \H$ such that $T_x(u, f) = \braket{K(x, \cdot)u}{f}_\H$. Moreover, by linearity of $T_x(u ,f)$ in $u$, we get $T_x(u_1+\lambda u_2, f) = \braket{K(x, \cdot)u_1 + \lambda K(x, \cdot)u_2}{f}_\H$. Hence, $K(x, \cdot)(u_1 + \lambda u_2) = K(x, \cdot)u_1 + \lambda K(x, \cdot)u_2$, meaning that $K(x, \cdot) \colon \U \to \H$ is linear. Since $\H$ is a Hilbert space of functions and $K(x, \cdot)u \in \H$, we can evaluate at another value $y \in X$. This gives us $K(x, y) \colon \U \to \U$. Note that
\begin{equation}
    \norm{K(x,\cdot)u}_\H^2 = \braket{K(x, \cdot) u}{K(x, \cdot)u}_\H = \braket{u}{K(x, x)u}_\U \leq \norm{u}_\U \norm{K(x,x)u}_\U \leq \norm{u}_\U C_x \norm{K(x, \cdot) u}_\H.
\end{equation}
Dividing both sides by $\norm{K(x,\cdot)u}_\H$ gives $\norm{K(x,\cdot)u}_\H \leq C_x \norm{u}_\U$. Using this yields
\begin{equation}
    \norm{K(x,y) u}_\U \leq C_y \norm{K(x, \cdot) u}_\H \leq C_y C_x \norm{u}_\U,
\end{equation}
which shows that $K(x,y) \in \L(\U)$. Summarizing, these observations regarding the Riesz representation yield the kernel definition.

For the converse, observe that we can take $C_x=\sqrt{\norm{K(x,x)}_{\L(\U)}} = \sqrt{\sup_{\norm{u}_\U\leq 1}\norm{K(x,x)u}_\U} < \infty$ since 
\begin{equation}
\begin{split}
    \norm{f(x)}_\U & = \sup_{\norm{u}_\U\leq 1}\abs{\braket{u}{f(x)}_\U} \\&= \sup_{\norm{u}_\U\leq 1}\abs{\braket{K(x,\cdot)u}{f}_\H} \\&\leq \norm{f}_\H\sup_{\norm{u}_\U\leq 1}\norm{K(x,\cdot)u}_\H \\
    & = \norm{f}_\H \sup_{\norm{u}_\U\leq 1}\sqrt{\braket{K(x,\cdot)u}_\H} \\
    & =  \norm{f}_\H \sup_{\norm{u}_\U\leq 1}\sqrt{\braket{K(x,x)u}{u}_\U} \\
    & \leq \norm{f}_\H \sup_{\norm{u}_\U\leq 1}C_x \norm{u}_\U \\ &= C_x \norm{f}_\H.\qedhere
\end{split}    
\end{equation}
\end{proof}

The kernel $K$ in Definition \ref{def:RKHS_kernel} satisfies
\begin{equation}
\sum_{i,j} \braket{K(x_i, x_j)u_i}{u_j}_\U = \sum_{i,j} \braket{K(x_i, \cdot)u_i}{K(x_j, \cdot) u_j}_\H = \norm{\sum_i K(x_i, \cdot)u_i}_\H^2   \geq 0
\end{equation}
for all $(x_i, u_i) \in X \times \U$, and
\begin{equation}
\begin{split}
    \braket{K(x,y)^* u}{v}_\U & = \braket{u}{K(x,y)v}_\U = \braket{K(y, \cdot) u}{K(x, \cdot) v}_\H \\
     & = \braket{K(x, \cdot) v}{K(y, \cdot) u}_\H = \braket{v}{K(y, x) u}_\U = \braket{K(y, x)u}{v}_\U.
\end{split}    
\end{equation}
for all $(x, u),(y, v) \in X \times \U$. Hence, each vv-RKHS kernel is of positive type (see Definition 2.2 in Carmeli et al.\ \cite{carmeli2006vector}) and $K(x,y)^* = K(y,x)$. Each such kernel $K \colon X \times X \to \L(\U)$ defines a unique vv-RKHS of $\U$-valued functions (see Proposition 2.3 in Carmeli et al.\ \cite{carmeli2006vector}), which is similar to the classical result by Aronszajn \cite{aronszajn1950theory} for scalar-valued RKHS. Therefore, any vv-RKHS induces a symmetric positive semi-definite kernel, and every symmetric positive semi-definite kernel induces a vv-RKHS. In particular, starting with $\H_s$ a scalar-valued RKHS with kernel $K_s:X \times X \to \R$ and $\U$ any Hilbert space, one can canonically define a kernel 
\begin{equation}\label{eq:vv-RKHS-from-scalar}K(\cdot,\cdot) = K_s(\cdot, \cdot) \mathrm{Id}_\U\end{equation} 
which then induces a vv-RKHS $\H$ of functions $f \colon X \to \U$. 

Another useful vv-RKHS characterization that is often used in machine learning is in terms of feature maps \cite{carmeli2006vector}.

\begin{definition}[Feature map definition vv-RKHS]
    A Hilbert space $\H$ of functions from $X$ to a Hilbert space $\U$ is a vv-RKHS if and only if there exists a Hilbert space $\F$ and a map $\Phi \colon X \to \L(\F, \U)$ such that for $(A\mu)(x) = \Phi(x)\mu$:
    \begin{enumerate}
        \item $\H = \{ A\mu \colon x \mapsto (A \mu)(x) \mid \mu \in \F\} \cong \F / \mathcal{N}(A)$,
        \item $\norm{f}_\H = \inf\left(\norm{\mu}_\F \mid \mu \in \F, f = A\mu \right)$. \qedhere
    \end{enumerate}
    \label{def:RKHS_feature}
\end{definition}

\section{Vector-valued Reproducing Kernel Banach Spaces} \label{sec:vv-RKBS}
While functions in a vv-RKHS map to a Hilbert space, functions in a vector-valued reproducing kernel Banach space (vv-RKBS) map to a Banach space. This section first introduces the standard definitions of vector-valued RKBSs, which do not assume any kernel structure. We subsequently define adjoint pairs of vv-RKBSs to introduce the reproducing kernel structure. Finally, we discuss properties of these adjoint pairs that parallel known results for vv-RKHSs. In particular, we show that every vv-RKBS has a reproducing kernel for some adjoint pair.

\subsection{Vector-valued RKBS definition}  \label{sec:vv-RKBS_definition}
An easy way to change a vv-RKHS into a vv-RKBS is to adjust Definition \ref{def:RKHS_bounded_point_eval} from the Hilbert space setting to the Banach space setting. 
\begin{definition}[Vector-valued Reproducing Kernel Banach space]
Let $\B$ be a Banach space of functions over a set $X$ mapping to a Banach space $\U$. $\B$ is a vector-valued Reproducing Kernel Banach space (vv-RKBS) if point evaluations are bounded functionals, i.e. for all $x\in X$
\begin{equation}
    \norm{f(x)}_\U \leq C_x\norm{f}_\B
\end{equation}
holds for all $f\in \B$ with the constant $C_x\geq 0$ depending on $x$ but not on $f$.
\label{def:RKBS_bounded_point_eval}
\end{definition}
Similarly, Definition \ref{def:RKHS_feature}, where the vv-RKHS is defined via a feature map, can be extended by replacing the Hilbert spaces in the vv-RKHS definition by Banach spaces. This yields a definition of vv-RKBS equivalent to Definition \ref{def:RKBS_bounded_point_eval} \cite{bartolucci2024neural}.
\begin{definition}[Feature map definition vv-RKBS]
    A Banach space $\B$ of functions from $X$ to a Banach space $\U$ is a vv-RKBS if and only if there exists a Banach space $\F$ and a map $\Phi \colon X \to \L(\F, \U)$ such that for $(A\mu)(x) = \Phi(x)\mu$:
    \begin{enumerate}
        \item $\B = \{ A\mu \colon x \mapsto (A \mu)(x) \mid \mu \in \F\} \cong \F / \mathcal{N}(A)$,
        \item $\norm{f}_\B = \inf\left(\norm{\mu}_\F \mid \mu \in \F, f = A\mu \right)$. \qedhere
    \end{enumerate}
    \label{def:RKBS_feature}
\end{definition}
\begin{remark}
    The scalar-valued RKBS definition \cite{bartolucci2023understanding} follows by setting $\norm{\cdot}_\R = |\cdot|$ in Definition~\ref{def:RKBS_bounded_point_eval}. In Definition~\ref{def:RKBS_feature}, we interpret $\L(\F, \R)$ as the dual space $\F^*$ and write $\Phi(x)\mu = \braket{\Phi(x)}{\mu}$, where $\braket{\cdot}{\cdot}$ denotes the canonical duality pairing.
\end{remark}

\subsection{Adjoint vv-RKBS pairs} \label{sec:adjoint_vv-RKBS_definition}
Due to the absence of the Riesz representation theorem, there is no immediate kernel-based definition of an RKBS. To address this, many works impose reproducing kernel assumptions. For scalar RKBSs, some adopt a $\delta$-dual approach, assuming that the closure of the point evaluation functionals is isometrically isomorphic to a Banach space of functions or another RKBS \cite{xu2023sparse, wang2024sparse}. Others assume the existence of a Banach space of functions or RKBS that is isometrically embedded in the dual space $\B^*$ of the RKBS $\B$ \cite{zhang2009reproducing, spek2025duality}. A third group of works uses less strict assumptions and only assumes the existence of another Banach space of functions (e.g., an RKBS) $\B^\di$ and a continuous bilinear form on $\B^\di \times \B$ \cite{georgiev2013construction, lin2022reproducing, ikeda2022koopman, heeringa2025deep}. Based on this form, they define a reproducing kernel property for $\B$, and optionally also for $\B^\di$.

In all these cases, the reproducing property is required for the original space $\B$, but is not always assumed for the $\delta$-dual, the embedded function space in $\B^*$, or the space $\B^\di$ appearing in the bilinear form. In contrast, we focus on the case where both $\B$ and its counterpart form an adjoint RKBS pair, meaning both are RKBSs and possess reproducing properties. We begin with the scalar-valued setting. Here, adjoint RKBS pairs and the corresponding kernel are defined through a continuous bilinear form, specifically a dual pairing between the RKBS $\B$ and another RKBS $\B^\di$. We then extend this framework to vector-valued RKBSs. In the next section, we show that, under additional assumptions, the dual pairing approach recovers the aforementioned isometric isomorphism-based formulations.

\begin{definition}[Dual pair \cite{fabian2011banach}]
A dual pair of Banach spaces $(\B,\B^\diamond)$ is two Banach spaces $\B$ and $\B^\diamond$ with a bilinear map 
\begin{equation}
    \braket{\cdot}{\cdot}_\B\colon \B^\diamond\times \B\to \R,\; (g,f)\mapsto \braket{g}{f}_\B
\end{equation}
that satisfies
\begin{subequations}
\begin{align}
    \braket{g}{f}_\B = 0 \; \forall g \in \B^\di &\implies f = 0 \label{nondegen_cond_pairing_1} \\
    \braket{g}{f}_\B = 0 \; \forall f \in \B\phantom{^\di} & \implies g = 0 \label{nondegen_cond_pairing_2}
\end{align}    
\end{subequations}
This bilinear map is called the pairing corresponding to $(\B,\B^\diamond)$ or, when no confusion arises, simply the pairing. Furthermore, the pairing is called continuous if the bilinear map is continuous. 
\end{definition}

\begin{definition}[Adjoint pair of scalar RKBS]
    Let $\B$ be an RKBS of functions mapping from a set $X$ to $\R$ and let $\B^\di$ be an RKBS of functions mapping from a set $\Omega$ to $\R$. Moreover, let $\braket{\cdot}{\cdot}_\B \colon \B^\di \times \B \to \R$ be a continuous dual pairing and $K \colon X \times \Omega \to \R$. We call $K$ a reproducing kernel for $\B$ when $K(x, \cdot) \in \B^\di$ for all $x \in X$ and
    \begin{equation}
        f(x) = \braket{K(x, \cdot)}{f}_\B
    \end{equation}
    for all $x \in X$ and $f \in \B$. If additionally $K(\cdot, w) \in \B$ for all $w \in \Omega$ and 
    \begin{equation}
        g(w) = \braket{g}{K(\cdot,w)}_\B
    \end{equation}
    for all $w \in \Omega$ and $g\in\B^\di$, then we call $\B^\di$ an adjoint RKBS of $\B$ and $(\B, \B^\di)$ an adjoint pair of RKBSs.
    \label{def:adjoint_pair_scalar_RKBS}
\end{definition}
\begin{remark}
    While some works use a dual pairing \cite{lin2022reproducing, heeringa2025deep}, others disregard the non-degeneracy conditions \eqref{nondegen_cond_pairing_1} and \eqref{nondegen_cond_pairing_2} and consider only a (continuous) bilinear form \cite{georgiev2013construction, ikeda2022koopman}. In principle, one can work with a general bilinear form instead of a duality pairing, since the non-degeneracy conditions are implicitly enforced by the reproducing properties.

    To see this, suppose there exists $f \in \B$ such that $\braket{g}{f}_\B = 0$ for all $g \in \B^\di$. Taking $g = K(x, \cdot)$ yields $f(x) = \braket{K(x, \cdot)}{f}_\B = 0$ for all $x \in X$, and hence $f = 0$. This confirms that condition \eqref{nondegen_cond_pairing_1} must hold. A similar argument shows that condition \eqref{nondegen_cond_pairing_2} is also necessary.
    
    Thus, any bilinear form defining a reproducing property must be a duality pairing. Nevertheless, we retain non-degeneracy as an explicit assumption in our definition to emphasize its necessity and to avoid the misconception that an arbitrary bilinear form is sufficient.
\end{remark}

Definition \ref{def:adjoint_pair_scalar_RKBS} only applies to scalar-valued RKBSs, as vector-valued functions $f$ cannot be represented by $\braket{K(x,\cdot)}{f}_\B \in \R$. To resolve this issue, we can take an approach similar to the vv-RKHS setting in Definition \ref{def:RKHS_kernel}. Instead of considering the function values directly, the vv-RKHS considers the inner products. However, since we are dealing with functions $f \colon X \to \U$ with $\U$ a Banach space, inner products are not available in general. We therefore replace the inner products with a duality pairing.

To that end, we introduce a dual pair $(\U, \U^\di)$ with pairing $\braket{\cdot}{\cdot}_\U$. Then, similar to the RKHS setting, we aim for a representation of the form
\begin{equation}
\braket{u^\di}{f(x)}_\U = \braket{K_{\U^\di}(x, \cdot)\, u^\di}{f}_\B
\end{equation}
for some kernel function $K_{\U^\di}(x, \cdot)\, u^\di \in \B^\di$ with $K_{\U^\di}(x, w) \in \L(\U^\di)$ and $w \in \Omega$. On the adjoint side, we require
\begin{equation}
\braket{g(w)}{u}_\U = \braket{g}{K_\U(\cdot, w)\, u}_\B
\end{equation}
for some $K_\U(\cdot, w)\, u \in \B$  with $K_{\U}(x, w) \in \L(\U)$ and $x \in X$. In the vv-RKHS case, one has the identity
\begin{equation}
\braket{f(x)}{u}_\U = \braket{u}{f(x)}_\U  = \braket{K(x, \cdot)u}{f}_\H = \braket{f}{K(x, \cdot)u}_\H,
\end{equation}
so it is natural to take $K_\U = K_{\U^\di} = K$ in that setting. In general, however, $K_\U \neq K_{\U^\di}$ as they act on different spaces.

Moreover, in the vv-RKHS framework with a real-valued inner product, one has the identity
\begin{equation}
    \braket{K(y,x)u}{v}_\U = \braket{u}{K(x,y)v}_\U = \braket{K(x, \cdot)u}{K(\cdot, y)v}_\H \eqqcolon K_b(x,y)(u, v),
    \label{eq:kernel_expressions_vv-rkhs}
\end{equation}
where the first equality follows from the symmetry property $K(x,y)^* = K(y,x)$. This shows that combining the inner product on $\U$ with a linear operator in either the first or second argument yields the same scalar output, which in turn can also be obtained via a specific bilinear form $K_b(x,y) \colon \U \times \U \to \R$. We aim to mimic this behavior when replacing the inner product with a duality pairing.

To formalize this structure and the reproducing property, we define the kernel as
\begin{equation}
K \colon X \times \Omega \to \twin(\mathcal{U}, \mathcal{U}^\di).
\end{equation}
Here, $(\mathcal{U}, \mathcal{U}^\di)$ is a dual pair, and $\twin(\mathcal{U}, \mathcal{U}^\di)$ is the space of twin operators which we introduce below and is based on the definition introduced in Diekmann et al.\ \cite{diekmann2021twin} for the more restrictive case in which $(\mathcal{U}, \mathcal{U}^\di)$ is assumed to be norming. The space $\twin(\mathcal{U}, \mathcal{U}^\di)$ can be viewed as a generalization of $\L(\U)$ (see Theorem \ref{thm:twin_isom_to_bounded_lin_ops_when_reflexive}), which is used in the vv-RKHS. 

\begin{definition}[Twin operators]
An operator $T:\U^\di\times\U\to\R$ over a dual pair $(\U,\U^\di)$ with pairing $\braket{\cdot}{\cdot}_\U$ is a twin operator when it is a bounded bilinear map that defines both a linear operator $T_\U \colon \U \to \U$ and linear operator $T_{\U^\di} \colon \U^\di \to \U^\di$ via:     
\begin{itemize}
    \item $T_{\U}$: $\braket{u^\di}{T_{\U} u}_\U = T(u^\di, u) \quad \forall u^\di \in \U^\di \quad \forall u \in \U$, and
    \item $T_{\U^\di}$: $\braket{T_{\U^\di}u^\di}{u}_\U = T(u^\di, u) \quad \forall u^\di \in \U^\di \quad \forall u \in \U$.
\end{itemize}
The space of all twin operators $T:\U^\di\times\U\to\R$ is denoted as $\twin(\U,\U^\di)$ and endowed with norm
\begin{equation}
    \norm{T}_{\twin(\U,\U^\di)} = \sup_{\norm{u}_\U\leq 1, \norm{u^\di}_{\U^\di}\leq 1}\abs{T(u^\di,u)}\qedhere
\end{equation}
\label{def:twin_operators}
\end{definition}
\begin{remark}
    In our setting, the twin operator $T$ takes the role of $K_b$ and the associated operators $T_\U$ and $T_{\U^\di}$ correspond to $K_\U$ and $K_{\U^\di}$, respectively. Moreover, $T_{\U^\di}$ can be interpreted as the adjoint of $T_\U$, and $T_\U u$ plays the role of a Riesz representation. Specifically, for fixed $u \in \U$, the map $u^\di \mapsto T(u^\di, u)$ defines a bounded linear functional on $\U^\di$ by the boundedness of $T$. When $\U^\di = \U$ and $\U$ is a Hilbert space, the Riesz representation theorem guarantees the existence of an element $T_\U u \in \U$ such that $\braket{u^\di}{T_\U u}_\U = T(u^\di, u)$. Since $T$ is linear in $u$, we have
    \begin{equation}
        \braket{u^\di}{T_\U(u + \lambda v)}_\U = T(u^\di, u + \lambda v) = T(u^\di, u) + \lambda T(u^\di, v) = \braket{u^\di}{T_\U u + \lambda T_\U v}_\U,
    \end{equation}
    which implies that $T_\U(u + \lambda v) = T_\U(u) + \lambda T_\U(v)$, i.e., $T_\U$ is linear. Thus, in the RKHS setting, we know that a linear operator $T_\U$ exists such that $\braket{u^\di}{T_\U u}_\U = T(u^\di, u)$ holds for all $u^\di \in \U^\di = \U$ and $u \in \U$. However, in the Banach space setting, no such representation theorem holds in general, so we must assume the identity $\braket{u^\di}{T_\U u}_\U = T(u^\di, u)$ as part of the structure. Similar considerations apply for the operator $T_{\U^\di}$.
\end{remark}

\begin{remark}
When the pairing is continuous, an important twin operator is the pairing $\braket{\cdot}{\cdot}_\U \colon \U^\di \times \U\to \R$ itself, with both induced linear operators being the identity on $\U$ and $\U^\di$.
\end{remark}

The norm of a twin operator $T\in \twin(\U,\U^\di)$ is related to the associated bounded linear operators $T_\U$ and $T_{\U^\di}$, but the norm of $T$ does not necessarily coincide with $\norm{T_\U}_{\L(\U)}$ and $\norm{T_{\U^\di}}_{\L(\U^\di)}$. In the special case where $(\U,\U^\di)$ is norming, $\norm{T}_{\twin(\U,\U^\di)} = \norm{T_\U}_{\L(\U)} = \norm{T_{\U^\di}}_{\L(\U^\di)}$, which will be crucial in proving Theorem \ref{thm:twin_isom_to_bounded_lin_ops_when_reflexive}. 
\begin{theorem}[Twin operator norm characterization for norming dual pairs]\label{thm:norming}
A dual pair $(\U,\U^\di)$ is called norming whenever
\begin{subequations}
\begin{align}
    \norm{u}_\U &= \sup_{\norm{u^\di}_{\U^\di} \leq 1}\abs{\braket{u^\di}{u}_\U} \\
    \norm{u^\di}_{\U^\di} &= \sup_{\norm{u}_{\U} \leq 1}\abs{\braket{u^\di}{u}_\U}
\end{align}    
\end{subequations}
Let $(\U,\U^\di)$ be a norming dual pair. If $T\in \twin(\U,\U^\di)$, then
\begin{equation}
    \norm{T}_{\twin(\U,\U^\di)} = \norm{T_\U}_{\L(\U)} = \norm{T_{\U^\di}}_{\L(\U^\di)}
\end{equation}
where $T_\U: \U\to\U$ and $T_{\U^\di}:\U^\di\to\U^\di$ are the bounded linear maps defined by $T$ on $\U$ and $\U^\di$, respectively.
\end{theorem}
\begin{proof}
    We only show $\norm{T}_{\twin(\U,\U^\di)} = \norm{T_\U}_{\L(\U)}$ as $\norm{T}_{\twin(\U,\U^\di)} = \norm{T_{\U^\di}}_{\L(\U^\di)}$ is shown similarly.
    \begin{align}
        \norm{T}_{\twin(\U,\U^\di)} & = \sup_{\norm{u}_\U\leq 1, \norm{u^\di}_{\U^\di}\leq 1}\abs{T(u^\di,u)} = \sup_{\norm{u}_\U\leq 1, \norm{u^\di}_{\U^\di}\leq 1}\abs{\braket{u^\di}{T_\U u}_\U} \\
        & = \sup_{\norm{u}_\U\leq 1} \left(\sup_{\norm{u^\di}_{\U^\di}\leq 1}\abs{\braket{u^\di}{T_\U u}_\U}\right) = \sup_{\norm{u}_\U\leq 1} \norm{T_\U u}_\U = \norm{T_\U}_{\L(\U)}.\qedhere
    \end{align}
\end{proof}
Using the Twin operators in Definition \ref{def:twin_operators}, we can define the adjoint pair of vv-RKBSs. 
\begin{definition}[Adjoint pair of vv-RKBS]
    Given a dual pair of Banach spaces $(\U, \U^\di)$, let $\B$ be a vv-RKBS of functions mapping from a set $X$ to $\U$ and let $\B^\di$ be a vv-RKBS of functions mapping from a set $\Omega$ to $\U^\di$. Moreover, let $\braket{\cdot}{\cdot}_\B \colon \B^\di \times \B \to \R$ be a continuous dual pairing and $K \colon X \times \Omega \to \twin(\U, \U^\di)$. We call $K$ a reproducing kernel for $\B$ when $K_{\U^\di}(x, \cdot) u^\di \in \B^\di$ for all $x, u^\di \in X \times \U^\di$ and
    \begin{equation}
        \braket{u^\di}{f(x)}_\U = \braket{K_{\U^\di}(x,\cdot)u^\di}{f}_\B 
    \end{equation}
    for all $x, u^\di \in X \times \U^\di$ and $f \in \B$. Moreover, if $K_{\U}(\cdot, w) u \in \B$ for all $(w, u) \in \Omega \times \U$ and 
    \begin{equation}
        \braket{g(w)}{u}_\U = \braket{g}{K_\U(\cdot,w)u}_\B
    \end{equation}
    for all $(w, u) \in \Omega \times \U$ and $g\in\B^\di$, then we call $\B^\di$ an adjoint vv-RKBS of $\B$ and $(\B, \B^\di)$ an adjoint pair of vv-RKBSs.
    \label{def:adjoint_pair_vector_RKBS}
\end{definition}

\begin{remark}
    We do not assume $(\U, \U^\di)$ to be a continuous dual pairing. The reason is that for the reproducing property of $\B$ and the continuity of the pairing $\braket{\cdot}{\cdot}_\B$, it suffices that $\braket{\cdot}{\cdot}_\U$ is continuous on the values attained by functions in $\B$. An analogous argument applies to the reproducing property of $\B^\di$. Hence, full continuity of the dual pairing $(\U, \U^\di)$ is not required.
\end{remark}

\begin{figure}[t]
    \centering
    \begin{subfigure}[b]{0.48\textwidth}
        \centering
            \begin{tikzpicture}
        
                \node (RKHSformula) {$K \colon X\times X\to \mathcal{L(\U)} \cong \twin(\U, \U^*)$};
                \node (H1) at ( {$(RKHSformula.west)!.25!(RKHSformula.east)$} |- {$(RKHSformula.north) + (0.0,1.5)$}) {$\H(X)$};    
                \node (H2) at ( {$(RKHSformula.west)!.75!(RKHSformula.east)$} |- {$(RKHSformula.north) + (0.0,1.5)$}) {$\H(X)$};
                \node at ($(RKHSformula.north) + (0.0,2.3)$) {\underline{\large\textbf{vv-RKHS}}};
                        
                % Draw two arrows for the X spaces
                \draw [->, line width=0.1mm, out=200, in=90] (H1.south) to ({$(RKHSformula.west)!.15!(RKHSformula.east)$} |- {$(RKHSformula.north) + (0.0,0.1)$});
                \draw [->, line width=0.1mm, out=200, in=90] (H2.south) to ({$(RKHSformula.west)!.27!(RKHSformula.east)$} |- {$(RKHSformula.north) + (0.0,0.05)$});
        
                % Draw two arrows for the output spaces
                \draw [->, red, line width=0.1mm, out=-20, in=90] (H1.south) to ({$(RKHSformula.west)!.84!(RKHSformula.east)$} |- {$(RKHSformula.north) + (0.0,0.05)$});
                \draw [->, red, line width=0.1mm, out=-20, in=90] (H2.south) to ({$(RKHSformula.west)!.91!(RKHSformula.east)$} |- {$(RKHSformula.north) + (0.0,0.1)$});
            
            \end{tikzpicture}

        \caption{Vector-valued RKHS}
        \label{fig:vv-RKHS}
    \end{subfigure}
    \hfill
    \begin{subfigure}[b]{0.48\textwidth}
        \centering
    \begin{tikzpicture}

        \node (RKBSformula) {${K\colon X\times \Omega \to \twin(\U, \U^\di)}$};
        \node (H1) at ( {$(RKBSformula.west)!.25!(RKBSformula.east)$} |- {$(RKHSformula.north) + (0.0,1.5)$}) {$\mathcal{B}(X)$};    
        \node (H2) at ( {$(RKBSformula.west)!.75!(RKBSformula.east)$} |- {$(RKBSformula.north) + (0.0,1.5)$}) {$\mathcal{B}^\diamond(\Omega)$};
        \node at ($(RKBSformula.north) + (0.0,2.3)$) {\underline{\large\textbf{vv-RKBS}}};
                
        % Draw two arrows for the input spaces
        \draw [->, line width=0.1mm, out=200, in=90] (H1.south) to ({$(RKBSformula.west)!.2!(RKBSformula.east)$} |- {$(RKBSformula.north) + (0.0,0.1)$});
        \draw [->, line width=0.1mm, out=200, in=90] (H2.south) to ({$(RKBSformula.west)!.37!(RKBSformula.east)$} |- {$(RKBSformula.north) + (0.0,0.05)$});
        
        % Draw two arrows for the output spaces
        \draw [->, red, line width=0.1mm, out=-20, in=90] (H1.south) to ({$(RKHSformula.west)!.72!(RKHSformula.east)$} |- {$(RKHSformula.north) + (0.0,0.05)$});
        \draw [->, red, line width=0.1mm, out=-20, in=90] (H2.south) to ({$(RKHSformula.west)!.8!(RKHSformula.east)$} |- {$(RKHSformula.north) + (0.0,0.1)$});
        
    \end{tikzpicture}

        \caption{Adjoint pair of vector-valued RKBS}
    \end{subfigure}
    \caption{\textbf{\textit{From vv-RKHS to adjoint pair of vv-RKBS}}. In a vv-RKHS, functions $f \colon X \to \U$ take values in a Hilbert space $\U$ and admit a reproducing kernel $K \colon X \times X \to \L(\U)$. In the vv-RKBS setting, we instead use Banach spaces $\B$ and $\B^\di$ for functions mapping to a dual pair $(\U, \U^\di)$, and replace inner products with duality pairings. This breaks the symmetry in the domain, replacing $X \times X$ with $X \times \Omega$, and requires twin operators in place of $\L(\U)$ to accommodate the asymmetry.}
    
    \label{fig:vv-RKBS}
\end{figure}

Figure~\ref{fig:vv-RKBS} summarizes the transition from the definition of a vv-RKHS to that of an adjoint pair of vv-RKBS. First, observe that the domain changes from $X \times X$ to $X \times \Omega$, reflecting the fact that Banach spaces are not necessarily isomorphic to their duals.

Moreover, instead of using the space of bounded linear operators $\L(\U)$, we work with twin operators. Twin operators generalize the relationship between $\L(\U)$ and the dual pair $(\U,\U^*)$ to arbitrary dual pairs $(\U,\U^\di)$. In the classical setting, this relationship is realized through adjoint operators in $\L(\U^*)$. The definition of $\twin(\U,\U^\di)$ extends this idea by requiring the existence of an operator $T_{\U^\di} \in \L(\U^\di)$ that acts as an adjoint of $T_{\U} \in \L(\U)$. The following theorem makes the aforementioned relationship precise and shows that twin operators indeed generalize bounded linear operators.
\begin{theorem}[Twin operators generalize $\L(\U)$]
Let $\U$ be any Banach space. Then $\twin(\U,\U^\ast)\cong\L(\U)$.
\label{thm:twin_isom_to_bounded_lin_ops_when_reflexive}
\end{theorem}
\begin{proof}
First, note that any $T_\U \in \L(\U)$ has a bounded adjoint $T_\U^*: \U^* \to \U^*$ defined by $[T_\U^*(u^*)](u) = \braket{u^*}{T_\U(u)}_\U$ (\cite[Definition 2.28]{fabian2011banach}). In addition, $T_\U$ induces a bounded bilinear form via $(u^*, u) \mapsto \braket{u^*}{T_\U u}_\U$, which satisfies the definition of twin operators since $\braket{T_\U^* u^*}{u}_\U = [T_\U^*(u^*)](u) = \braket{u^*}{T_\U u}_\U$ by the definition of the adjoint and of the canonical duality pairing $\braket{\cdot}{\cdot}_\U: \U^* \times \U \to \R$. Consequently, the linear map $T_\U \to \braket{\cdot}{T_\U \cdot}_\U$ maps  $\L(\U)$ into $\twin(\U,\U^\ast)$. 

To establish that $\L(\U)$ and $\twin(\U,\U^\ast)$ are isometrically isomorphic, we show that $T_\U \to \braket{\cdot}{T_\U \cdot}_\U$ is a bijection and an isometry. Injectivity follows by the nondegeneracy of the dual pairing. For surjectivity, note that in the defining relations for $T \in \twin(\U,\U^\ast)$, namely
\[T(u^*, u) = \braket{u^*}{T_{\U} u}_\U = \braket{T_{\U^*}u^*}{u}_\U \quad \forall u^* \in \U^* \quad \forall u \in \U,\]
the second equality is exactly the defining property of the adjoint, that is, we must have $T_{\U^*} = T_\U^*$. This means that the injective map $T_\U \to \braket{\cdot}{T_\U \cdot}_\U$ is surjective onto $\twin(\U,\U^\ast)$ and therefore a bijection. Finally, using Theorem \ref{thm:norming}, this bijection is also an isometric isomorphism, since the dual pair $(\U,\U^\ast)$ is always norming for any Banach space $\U$, as can be seen by the definition of the norm in $\U^\ast$ and an application of the Hahn-Banach theorem.
\end{proof}

% \replace{As a final remark on the twin operators, we show that, under some assumptions, the norm of a twin operator $T\in \twin(\U,\U^\di)$ is linked to the norms of the two bounded linear operators $T_\U$ and $T_{\U^\di}$.

% \begin{theorem}[Twin operator norm characterization for norming dual pairs]%\label{thm:norming}
% A dual pair $(\U,\U^\di)$ is called norming whenever
% \begin{subequations}
% \begin{align}
%     \norm{u}_\U &= \sup_{\norm{u^\di}_{\U^\di} \leq 1}\abs{\braket{u^\di}{u}_\U} \\
%     \norm{u^\di}_{\U^\di} &= \sup_{\norm{u}_{\U} \leq 1}\abs{\braket{u^\di}{u}_\U}
% \end{align}    
% \end{subequations}
% Let $(\U,\U^\di)$ be a norming dual pair. If $T\in \twin(\U,\U^\di)$, then
% \begin{equation}
%     \norm{T}_{\twin(\U,\U^\di)} = \norm{T_\U}_{\L(\U)} = \norm{T_{\U^\di}}_{\L(\U^\di)}
% \end{equation}
% where $T_\U: \U\to\U$ and $T_{\U^\di}:\U^\di\to\U^\di$ are the bounded linear maps defined by $T$ on $\U$ and $\U^\di$, respectively.
% \end{theorem}
% \begin{proof}
%     We only show $\norm{T}_{\twin(\U,\U^\di)} = \norm{T_\U}_{\L(\U)}$ as $\norm{T}_{\twin(\U,\U^\di)} = \norm{T_{\U^\di}}_{\L(\U^\di)}$ is shown similarly.
%     \begin{align}
%         \norm{T}_{\twin(\U,\U^\di)} & = \sup_{\norm{u}_\U\leq 1, \norm{u^\di}_{\U^\di}\leq 1}\abs{T(u^\di,u)} = \sup_{\norm{u}_\U\leq 1, \norm{u^\di}_{\U^\di}\leq 1}\abs{\braket{u^\di}{T_\U u}_\U} \\
%         & = \sup_{\norm{u}_\U\leq 1} \left(\sup_{\norm{u^\di}_{\U^\di}\leq 1}\abs{\braket{u^\di}{T_\U u}_\U}\right) = \sup_{\norm{u}_\U\leq 1} \norm{T_\U u}_\U = \norm{T_\U}_{\L(\U)}.\qedhere
%     \end{align}
% \end{proof}}{}

The kernel construction in Definition \ref{def:adjoint_pair_vector_RKBS} differs from existing works on kernels for vv-RKBSs. Most of these works define a kernel on a symmetric domain $X \times X$. However, within this setting, the output space and the used assumptions vary considerably.

In Zhang et al.~\cite{zhang2013vector}, the vv-RKBS $\B$ and the output space $\U$ are assumed to be uniformly convex and uniformly Fr\'echet differentiable, which in particular implies reflexivity. Under these assumptions, they introduce a compatible semi-inner product structure to construct a reproducing kernel $K \colon X \times X \to B(\U)$. Here $B(\U)$ denotes the set of continuous, not necessarily linear, operators from $\U$ to itself. Using this structure, they establish a general representer theorem for the adjoint element of the solution.

Lin et al.~\cite{lin2021multi} and Chen et al.~\cite{chen2019vector} relax the uniformity assumptions and work in non-reflexive $\ell^1$ and group-lasso settings. However, they focus on $\R^d$-valued functions with kernels $K \colon X \times X \to \R^{d \times d}$. Similar to our approach, the reproducing property is achieved via two function spaces $(\B, \B^\di)$ linked by a bilinear form. Their representer theorem for the actual solution, however, requires a strong Lebesgue constant condition on the kernel as well as additional kernel admissibility conditions.

Combettes et al.~\cite{combettes2018regularized} construct vv-RKBSs using a feature map approach and assume that both the output and feature spaces are separable real Banach spaces. Under reflexivity, strict convexity, and smoothness of the feature space, they show the existence of a reproducing kernel
\begin{equation}
    K \colon X \times X \to B(\U^*, \U),
\end{equation}
where $B(\U^*,\U)$ denotes operators (not necessarily linear) mapping bounded subsets of $\U^*$ to bounded subsets of $\U$, and where the reproducing property additionally requires a duality mapping. They also prove a representer theorem that does not depend on the kernel structure itself, but still requires both the output and feature spaces to be separable and reflexive.

In contrast to the previously mentioned works, Wang et al.~\cite{wang2024hypothesis} go beyond the symmetric setting and, motivated by neural network architectures, define kernels $K \colon X \times \Omega \to \R^d$. Their construction, however, is specialized to finite-dimensional outputs, and the kernel's output does not directly correspond to a $d \times d$ matrix as used in the vv-RKHS setting.

Compared to these existing approaches, our adjoint-pair definition treats both symmetric and asymmetric kernel domains, supports general Banach-valued outputs, and produces kernel outputs that generalize the bounded linear operators used in the vv-RKHS setting. Moreover, it avoids structural assumptions such as reflexivity and separability for the output and feature spaces. This will be crucial in Sections \ref{sec:integral_and_neural_vv-RKBS} and \ref{sec:vv-NNs_HypNets_DeepONets}. There, we must work with the non-separable and non-reflexive space of vector measures as the feature space, as well as the non-separable and non-reflexive output spaces required for DeepONets and hypernetworks.

\subsection{vv-RKBS Properties} \label{sec:adjoint_vv-RKBS_properties}
As mentioned in the previous section, Definitions~\ref{def:adjoint_pair_scalar_RKBS} and~\ref{def:adjoint_pair_vector_RKBS} do not fully align with the definitions based on the $\delta$-dual \cite{xu2023sparse, wang2024sparse} or those involving isometric embeddings into the dual space \cite{zhang2009reproducing, spek2025duality}. The following theorem shows that these definitions become equivalent to Definitions~\ref{def:adjoint_pair_scalar_RKBS} and~\ref{def:adjoint_pair_vector_RKBS} when additional conditions are imposed on the latter two. For clarity, we restrict the statement to the vector-valued case in Definition \ref{def:adjoint_pair_vector_RKBS}.

\begin{theorem}
    Let $X,\Omega$ be sets, $(\U, \U^\di)$ be a dual pair and let $\B$ be a vv-RKBS containing functions mapping $X$ to $\U$. Assume that $\B^\di$ is a Banach space of functions mapping the set $\Omega$ to $\U^\di$. Then the following statements are equivalent:
    \begin{enumerate}
        \item[1.] $\B^\di$ is isometrically embedded in $\B^*$.
        \item[2.] There exists a duality pairing $\braket{\cdot}{\cdot}_\B \colon \B^\di \times \B \to \R$ and the norm $\norm{\cdot}_{\B^\di}$ satisfies $\norm{g}_{\B^\di} = \sup_{\norm{f}_\B \leq 1} | \braket{g}{f}_\B |$.
    \end{enumerate}
    Moreover, let $\delta_{x, u^\di}(f) \coloneqq \braket{u^\di}{f(x)}_\U$ for $f \in \B$, let $\Delta \coloneqq \spn\set{\delta_{x, u^\di} \given (x, u^\di) \in X \times \U^\di}$, and let the $\delta$-dual be defined as the closure $\overline{\Delta}$ of $\Delta$ under the $\B^*$ norm. Then the following are also equivalent:
    \begin{enumerate}
        \item[3.] $\Delta \subseteq \B^*$ and the $\delta$-dual $\overline{\Delta}\subseteq \B^*$ is isometrically isomorphic to $\B^\di$.
        \item[4.] 
        \begin{itemize}
            \item There exists a duality pairing $\braket{\cdot}{\cdot}_\B \colon \B^\di \times \B \to \R$, 
            \item the norm $\norm{\cdot}_{\B^\di}$ satisfies $\norm{g}_{\B^\di} = \sup_{\norm{f}_\B \leq 1} | \braket{g}{f}_\B |$,
            \item there exists a map $K_{\U^\di}$ that assigns to each $(x,w) \in X \times \Omega$ a linear operator $K_{\U^\di}(x,w) \colon \U^\di \to \U^\di$ such that for all $u \in \U^\di$ and $f\in \B$ the function $K_{\U^\di}(x,\cdot)u^\di \in \B^\di$ and $\braket{u^\di}{f(x)}_\U = \braket{K_{\U^\di}(x,\cdot)u^\di}{f}_\B$, and
            \item the set $\spn\{K_{\U^\di}(x, \cdot) u^\di \mid (x, u^\di) \in X \times \U^\di\}$ is dense in $\B^\di$.
        \end{itemize}
    \end{enumerate}
    \label{thm:equivalence_adjoint_banach_space_defs}
\end{theorem}
\begin{proof}
    \mbox{}\\
    \textbf{(1) $\iff$ (2):}
    
    $\B^\di$ is isometrically embedded in $\B^*$ if and only if there exists a mapping $\iota \colon \B^\di \to \B^*$ such that $\norm{g}_{\B^\di} = \norm{\iota(g)}_{\B^*} = \sup_{\norm{f}_\B \leq 1} | \braket{\iota(g)}{f}_{\B^*,\B} |$ with $\braket{\cdot}{\cdot}_{\B^*,\B}$ the canonical duality pairing on $\B$. 

    For $(1)\Rightarrow(2)$, define the pairing between $\B$ and $\B^\di$ as $\braket{g}{f}_\B \coloneqq \braket{\iota(g)}{f}_{\B^*,\B}$. 
    
    For $(2)\Rightarrow(1)$, let $\iota \colon \B^\di \to \B^*$ be defined via $g \mapsto \braket{g}{\cdot}_\B \in \B^*$. Then 
    \begin{equation}
        \norm{\iota(g)}_{\B^*} = \sup_{\norm{f}_\B \leq 1} | \braket{\iota(g)}{f}_{\B^*,\B}| = \sup_{\norm{f}_\B \leq 1} | \braket{g}{f}_\B| = \norm{g}_{\B^\di}, 
    \end{equation}
    which shows that $\B^\di$ is isometrically embedded in $\B^*$.

    \textbf{(3) $\implies$ (4):}
    
    Let $\iota \colon \B^\di \to \overline{\Delta}$ be the isometric isomorphism between the $\delta$-dual and $\B^\di$. Define the pairing $\braket{g}{f}_\B \coloneqq \braket{\iota(g)}{f}_{\B^*,\B}$. Then, as in the proof of (1)~$\Rightarrow$~(2), we get $\norm{g}_{\B^\di} = \sup_{\norm{f}_\B \leq 1} | \braket{g}{f}_\B |$.

    Note that $\iota^{-1}(\delta_{x, u^\di}) \in \B^\di$ for all $(x, u^\di) \in X \times \U^\di$. Since $\delta_{x, u^\di + \lambda v^\di} = \delta_{x, u^\di} + \lambda \delta_{x, v^\di}$, it follows that $u^\di \mapsto \iota^{-1}(\delta_{x, u^\di}) \in \B^\di$ is linear. In particular, as $\B^\di$ contains functions over $\Omega$, we have $u^\di \mapsto \iota^{-1}(\delta_{x, u^\di})(w) \in \U^\di$ is linear for all $w \in \Omega$. Defining the linear operator $K_{\U^\di}(x, w) \coloneqq \iota^{-1}(\delta_{x, \cdot})(w)$ and using the definition of the pairing, we get:
    \begin{equation}
        \braket{K_{\U^\di}(x, \cdot) u^\di}{f}_\B = \braket{\iota^{-1}(\delta_{x, u^\di})}{f}_\B = \braket{\delta_{x, u^\di}}{f}_{\B^*,\B} = \braket{u^\di}{f(x)}_\U.
    \end{equation}
    This proves the reproducing property.
    
    It remains to show that the set $\{K_{\U^\di}(x, \cdot) u^\di \mid (x, u^\di) \in X \times \U^\di\}$ is dense in $\B^\di$. Since $\iota(K_{\U^\di}(x, \cdot) u^\di) = \delta_{x, u^\di}$ and $\iota$ is an isometry, taking closures with respect to the norms on both sides shows that the closure of $\spn\set{K_{\U^\di}(x, \cdot) u^\di \given (x, u^\di) \in X \times \U^\di}$ is isometrically isomorphic to $\overline{\Delta}$, which is isometrically isomorphic to $\B^\di$. This completes the proof.

    \textbf{(4) $\implies$ (3):}
    
     Define the mapping $\iota$ by $\iota(K_{\U^\di}(x, \cdot) u^\di) \coloneqq \delta_{x, u^\di}$. Note that
    \begin{equation}
    \begin{split}
    \norm{\sum_{n=1}^N a_n K_{\U^\di}(x_n, \cdot) u^\di_n}_{\B^\di} &= \sup_{\norm{f}_\B \leq 1} \left| \braket{\sum_{n=1}^N a_n K_{\U^\di}(x_n, \cdot) u^\di_n}{f}_\B\right| = \sup_{\norm{f}_\B \leq 1} \left| \sum_{n=1}^N a_n \braket{K_{\U^\di}(x_n, \cdot) u^\di_n}{f}_\B\right| 
    \\&= \sup_{\norm{f}_\B \leq 1} \left| \sum_{n=1}^N a_n \delta_{x_n, u^\di_n}(f)\right| = \norm{\sum_{n=1}^N a_n \delta_{x_n,u^\di_n}}_{\B^*},
    \end{split}
    \end{equation}
    for all $a_n\in\R$, $x_n\in X$ and $u^\di_n\in \U^\di$ for any $N\in \N$.
    This shows that $\iota \colon \spn\set{K_{\U^\di}(x, \cdot) u^\di \given (x, u^\di) \in X \times \U^\di} \to \Delta$ is an isometric isomorphism. By the continuous extension theorem, $\iota$ can be extended to an isometric isomorphism from $\overline{\spn(K_{\U^\di}(x, \cdot) u^\di \mid (x, u^\di) \in X \times \U^\di)}$ to $\overline{\Delta}$. Since $\{K_\U(x, \cdot) u^\di \mid (x, u^\di) \in X \times \U^\di\}$ is dense in $\B^\di$, it follows that $\B^\di$ is isometrically isomorphic to $\overline{\Delta}$.
\end{proof}

The previous theorem demonstrates that using only a duality pairing is the weakest assumption one can make. This raises the question whether properties of the vv-RKHS extend to the adjoint pair of vv-RKBS under this minimal formulation, or whether additional assumptions, such as those in Theorem~\ref{thm:equivalence_adjoint_banach_space_defs}, are required.

An example where additional assumptions are necessary relates to the fact that, in the vv-RKHS setting, the set $\{K(x,\cdot)u \mid (x,u) \in X \times \U\}$ is dense in the vv-RKHS. As the equivalence between (3) and (4) in Theorem~\ref{thm:equivalence_adjoint_banach_space_defs} already illustrates, further assumptions are needed in the vv-RKBS setting to obtain a similar density property. Moreover, the theorem only establishes density in $\B^\di$, not in $\B$. In Wang et al.\ \cite{wang2024sparse}, this issue is also addressed; they show that, in the scalar case, assuming $\B^\di$ is a pre-dual space leads to a corresponding density result in $\B$.

Another result that holds in the vv-RKHS setting and whose analogue we can ask for in the vv-RKBS setting is the existence of a kernel. In vv-RKHSs a reproducing kernel is always present, but in Definitions \ref{def:RKBS_bounded_point_eval} and \ref{def:RKBS_feature} no kernel arises directly, since the Riesz representation theorem is unavailable. It may therefore seem questionable to refer to these spaces as reproducing kernel spaces. Nevertheless, we show that every vv-RKBS as in these definitions is part of an adjoint pair of vv-RKBSs (Definition \ref{def:adjoint_pair_vector_RKBS}), which ensures the existence of a reproducing kernel. Moreover, the next theorem demonstrates that one can always choose an adjoint vv-RKBS satisfying the strongest assumptions of Theorem \ref{thm:equivalence_adjoint_banach_space_defs}, in particular those in point 4. Thus, while these stronger assumptions are not strictly necessary, they can always be met in a suitable adjoint pair.
\begin{theorem}[Every vv-RKBS corresponds to an adjoint vv-RKBS pair]
    Let $\B$ be a vv-RKBS with output Banach space $\U$. Define $\delta_{x, u^*}(f) \coloneqq \braket{u^*}{f(x)}_\U$ for $\braket{\cdot}{\cdot}_\U$ the canonical duality pairing of the pair $(\U,\U^\ast)$ with $\U^\ast$ the continuous dual. Denote $\Delta \coloneqq \spn\set{ \delta_{x, u^*} \given (x, u^*) \in X \times \U^*}$, and let the $\delta$-dual be defined as the closure $\overline{\Delta}$ of $\Delta$ under the $\B^*$ norm. Define the Banach space 
    \begin{equation}
    \B^\di \coloneqq \left\{g \colon \B \times \U^* \to \U^* \mid g(f, \widetilde{u}^*) = g_*(f) \widetilde{u}^* \text{ for some } g_* \in \overline{\Delta}\right\}
    \end{equation}
    with $\norm{g}_{\B^\di} \coloneqq \norm{g_*}_{\B^*}$. Then $(\B, \B^\di)$ forms an adjoint pair of vv-RKBSs with duality pairing $\braket{g}{f}_\B \coloneqq g_*(f)$ and kernel
    \begin{equation}
    \begin{split}
        K_{\U}(x,(f, \widetilde{u}^*))u \coloneqq f(x) \braket{\widetilde{u}^*}{u}_\U, & \qquad \qquad \qquad \qquad K_{\U^*}(x,(f, \widetilde{u}^*))u^* \coloneqq \braket{u^*}{f(x)}_\U \widetilde{u}^*, \\
        \\
        K(x,(f, \widetilde{u}^*))(u^*, u) &= \braket{u^*}{f(x)}_\U \braket{\widetilde{u}^*}{u}_\U = \delta_{x, u^*}(f) \braket{\widetilde{u}^*}{u}_\U.
    \end{split}        
    \end{equation}
    Moreover, by definition, 
    \begin{equation}
        \norm{g}_{\B^\di} = \norm{g_*}_{\B^*} = \sup_{\norm{f}_\B \leq 1} |g_*(f)| = \sup_{\norm{f}_\B \leq 1} |\braket{g}{f}_\B|,
    \end{equation}
    and the set
    \begin{equation}
    \spn\{K_{\U^*}(x, \cdot) \widetilde{u}^* \mid (x, \widetilde{u}^*) \in X \times \U^*\} = \spn\{g(f, \widetilde{u}^*) \coloneqq \delta_{x, u^*}(f) \widetilde{u}^* \mid (x, u^*) \in X \times \U^*, (f, \widetilde{u}^*) \in \B \times \U^*\}
    \end{equation}
    is dense in $\B^\di$.
    \label{thm:every_vv_RKBS_corresponds_to_adjoint_vv_RKBS_pair}
\end{theorem}
\begin{proof}
    Note that, by definition, 
    \begin{equation}
        \norm{g(f, \widetilde{u}^*)}_{\U^*} = \norm{g_*(f) \widetilde{u}^*}_{\U^*} = \norm{\widetilde{u}^*}_{\U^*} |g_*(f) | \leq \norm{\widetilde{u}^*}_{\U^*} \norm{f}_\B \norm{g_*}_{\B^*} = C_{f, \widetilde{u}^*} \norm{g}_{\B^\di}
    \end{equation}
    where $ C_{f, \widetilde{u}^*} = \norm{\widetilde{u}^*}_{\U^*} \norm{f}_\B$. This shows that $\B^\di$ is a vv-RKBS. 
    
    Furthermore, assume $\braket{g}{f}_\B = 0$ for all $g \in \B^\di$. Then $\braket{\delta_{x, u^*}}{f}_\B = \braket{u^*}{f(x)}_{\U} = 0$ for all $(x, u^*) \in X \times \U^*$. This implies $f(x) = 0$ for all $x \in X$, and hence $f=0$. Conversely, if $\braket{g}{f}_\B = g_*(f) = 0$ for all $f \in \B$, then by definition $g = 0$. Combining these two facts shows that $\braket{\cdot}{\cdot}_\B$ is a duality pairing.
    
    We now check the reproducing properties. For $f \in \B$, the result immediately follows from the definition of the pairing
    \begin{equation}
        \braket{K_{\U^*}(x, (\cdot, \cdot)) u^*}{f}_\B = \delta_{x, u^*}(f) = \braket{u^*}{f(x)}_\U 
    \end{equation}
    Similarly, for $g \in \B^\di$
    \begin{equation}
        \braket{g}{K_{\U}(\cdot, (f, \widetilde{u}^*))u}_\B = g_*(f) \braket{\widetilde{u}^*}{u}_{\U} = \braket{g_*(f) \widetilde{u}^*}{u}_\U = \braket{g(f, \widetilde{u}^*)}{u}_\U 
    \end{equation}
    Finally, the formula for $K(x,(f, \widetilde{u}^*))(u^*, u)$ follows by combining the reproducing properties of elements of $\B^\di$ with the specific form of $K_\U$
    \begin{align}
        K(x,(f, \widetilde{u}^*))(u^*, u) & = \braket{K_{\U^*}(x, (\cdot, \cdot)) u^*}{K_{\U}(\cdot, (f, \widetilde{u}^*)) u}_\B\nonumber \\
        & = \braket{\delta_{x, u^*}(f) \widetilde{u}^*}{u}_\U \\
        & = \delta_{x, u^*}(f) \braket{\widetilde{u}^*}{u}_\U \nonumber\\
        & = \braket{u^*}{f(x)}_\U \braket{\widetilde{u}^*}{u}_\U.\qedhere
    \end{align}
\end{proof}
Two other properties of the vv-RKHS are the uniqueness of the kernel and the fact that every symmetric positive semi-definite kernel defines a vv-RKHS. The next theorem shows that, under the weakest definition of an adjoint vv-RKBS pair, we still obtain (1) a unique kernel that satisfies a relation analogous to \eqref{eq:kernel_expressions_vv-rkhs}, (2) that every kernel satisfies a certain bound, and (3) that every kernel satisfying such a bound defines an adjoint vv-RKBS pair. However, to obtain a vv-RKBS with specific properties, additional assumptions are required. In particular, the next theorem also provides an example of additional conditions that yield an adjoint vv-RKBS pair satisfying $\norm{g}_{\B^\di} = \sup_{\norm{f}_\B\leq 1}| \braket{g}{f}_\B|$. 

\begin{theorem}[vv-RKBS kernel properties and kernels defining vv-RKBSs]
    Let $(\U,\U^\di)$ be a dual pair. If $(\B, \B^\di)$ is an adjoint vv-RKBS pair with reproducing kernel $K \colon X \times \Omega \to \twin(\U, \U^\di)$, then the kernel $K$ is unique, satisfies the equalities
    \begin{equation}
    K(x,w)(u^\di,u) = \braket{K_{\U^\di}(x,\cdot) u^\di }{K_\U(\cdot, w) u}_\B = \begin{cases}
        \braket{K_{\U^\di}(x,w)u^\di}{u}_\U \\
        \braket{u^\di}{K_\U(x,w)u}_\U ,
        \end{cases}\label{eq:charachterization_twin_kernel}
    \end{equation}
    and the bounds    \begin{subequations}
    \begin{align}
    \norm{K_\U(x,w)u}_\U &\leq C_X(x) \norm{K_\U(\cdot, w)u}_\B \eqqcolon C_X(x) C_{\Omega \times \U}(w, u), \text{ and} \label{eq:kernel_bound_1a}\\\norm{K_{\U^ \di}(x,w)u^\di}_\U &\leq C_\Omega(w) \norm{K_{\U^\di}(x, \cdot)u^\di}_{\B^\di} \eqqcolon C_\Omega(w) C_{X\times \U^\di}(x, u^\di) \label{eq:kernel_bound_1b}
    \end{align}        
    \end{subequations}
    for some functions $C_X\colon X\to \R^+$ and $C_\Omega\colon\Omega\to \R^+$.  

    Conversely, if we are given a continuous dual pair $(\U, \U^\di)$, sets $X, \Omega$, and a map $K \colon X \times \Omega \to \twin(\U, \U^\di)$ that satisfies 
    \begin{subequations}
    \begin{align}
        \norm{K_\U(x,w)u}_\U &\leq C_X(x) C_{\Omega \times \U}(w, u) \quad \text{and} \label{eq:kernel_bound_2a}\\
        \norm{K_{\U^ \di}(x,w)u^\di}_\U &\leq C_\Omega(w) C_{X\times \U^\di}(x, u^\di)\label{eq:kernel_bound_2b}
    \end{align}
    \end{subequations} 
    for some functions $C_X\colon X\to \R^+$, $C_\Omega\colon\Omega\to \R^+$, $C_{\Omega\times \U}\colon \Omega\times \U\to \R^+$, and $C_{X\times \U^\di}\colon X\times \U^\di\to \R^+$, then there exists a vv-RKBS pair $(\B, \B^\di)$ with domains $X$ and $\Omega$, respectively, and $K$ as reproducing kernel. Moreover, when additionally assuming 
    \begin{subequations}
    \begin{align}
        \norm{u^\di}_{\U^\di} &= \sup_{\norm{u}_{\U} \leq 1}|\braket{u^\di}{u}_\U| \quad \text{and} \label{eq:u-di_norming}\\
        \norm{K_\U(x, w)u}_\U &\leq C_X(x) C_\Omega(w) \norm{u}_\U,\label{eq:kernel_bound_c}
    \end{align}
    \end{subequations}
    one can define a vv-RKBS pair $(\B, \B^\di)$ satisfying $\norm{g}_{\B^\di} = \sup_{\norm{f}_\B \leq 1} | \braket{g}{f}_\B|$.
\end{theorem}
\begin{proof}
    Let $(\B, \B^\di)$ be an adjoint vv-RKBS pair. By first using the definition of a twin operator and subsequently the reproducing property, we obtain \eqref{eq:charachterization_twin_kernel}.
    Now, assume there is another kernel $\widetilde{K}$. Again by the reproducing property:
    \begin{equation}
        0 = \braket{g(w)}{u} - \braket{g(w)}{u} = \braket{g}{\left(K_\U(x,w) - \widetilde{K}_\U(x,w)\right)u}_\B
    \end{equation}
    holds for all $g\in \B^\di$. When $g = K_{\U^\di}(x, \cdot)u^\di \in \B^\di$ and we use the reproducing property $\braket{u^\di}{f(x)}_\U = \braket{K_{\U^\di}(x, \cdot)u^\di}{f}_\B$, we get
    \begin{equation}
    \begin{split}
        0 &= \braket{K_{\U^\di}(x, \cdot)u^\di}{\left(K_\U(\cdot,w) - \widetilde{K}_\U(\cdot,w)\right)u}_\B = \braket{u^\di}{\left(K_\U(x,w) - \widetilde{K}_\U(x,w)\right)u}_\U
    \end{split}    
    \end{equation}
    for all $(x,w,u, u^\di) \in X \times \Omega \times \U \times\U^\di$. By the non-degeneracy condition \eqref{nondegen_cond_pairing_1} of the dual pair $(\U, \U^\di)$, we conclude $K_\U(x,w) = \widetilde{K}_\U(x,w)$. Employing the same reasoning for $K_{\U^\di}$ shows that $K_{\U^\di}=\widetilde{K}_{\U^\di}$. Hence, both $K_\U$ and $K_{\U^\di}$ are unique. Thus, by \eqref{eq:charachterization_twin_kernel}, the full kernel is unique.
    
    To show the converse, we follow the approach in Heeringa et al.\ \cite{heeringa2025deep}. For a kernel $K \colon X \times \Omega \to \twin(\U, \U^\di)$ satisfying \eqref{eq:kernel_bound_2a} and \eqref{eq:kernel_bound_2b} define the vector spaces:
    \begin{subequations}
        \begin{align}
            \V & \coloneqq \spn\set{K_\U(\cdot, w) u \colon X \to \U \given (w,u) \in \Omega \times \U}, \\
            \V^\di & \coloneqq \spn\set{K_{\U^\di}(x, \cdot) u^\di \colon \Omega \to \U^\di \given (x,u^\di) \in X \times \U^\di}
        \end{align}
    \end{subequations}
    and the bilinear mapping
    \begin{equation}
        \braket{\sum_{j} K_{\U^\di}(x_j, \cdot)u_j^\di}{\sum_i K_\U(\cdot, w_i) u_i}_\V \coloneqq \sum_{i,j} K(x_j, w_i)(u_j^\di, u_i)
    \end{equation}
    between them. 
    Equip $\V$ with the norm
    \begin{equation}
        \norm{f}_\V \coloneqq \sup_{x\in X} \frac{\norm{f(x)}_\U}{C_X(x)}.
    \end{equation}
    This norm is well-defined since
    \begin{equation}
        \norm{f(x)}_\U = \norm{\sum_{i} K_\U(x, w_i)u_i}_\U \leq \sum_i \norm{K_\U(x, w_i)u_i}_\U \leq C_X(x) \sum_i C_{\Omega \times \U}(w_i, u_i).
    \end{equation}
    By definition, we also have
    \begin{equation}
        \norm{f(x)}_\U \leq C_X(x) \norm{f}_\V
    \end{equation}
    for all $f \in \V$. Let $\B$ be the completion of $\V$ with respect to this norm. This makes $\B$ a Banach space of functions, and hence, by the above inequality, a vv-RKBS.
    
    For $\V^\di$, define the norm as
    \begin{equation}
        \norm{g}_{\V^\di} \coloneqq \max\left(\sup_{w\in \Omega} \frac{\norm{g(w)}_{\U^\di}}{C_\Omega(w)}, \sup_{\norm{f}_\B\leq 1} |\braket{g}{f}_{\V^\di,\B}|\right),
    \end{equation}
    where the pairing $\braket{\cdot}{\cdot}_{\V^\di,\B}$ is the extension of the pairing $\braket{\cdot}{\cdot}_{\V}$ from $\V^\di \times \V$ to $\V^\di \times \B$. This ensures
    \begin{equation}
        \norm{g(w)}_{\U^\di} \leq C_\Omega(w) \norm{g}_{\V^\di}, \quad |\braket{g}{f}_\B| \leq \norm{g}_{\V^\di} \norm{f}_\B.
    \end{equation}
    Let $\B^\di$ be the completion of $\V^\di$ with respect to $\norm{\cdot}_{\V^\di}$. Extending the pairing $\braket{\cdot}{\cdot}_{\V}$ continuously via the density of $\V,\V^\di$ in the corresponding completions, yields a pairing $\braket{\cdot}{\cdot}_{\B}$ satisfying
    \begin{equation}
    |\braket{g}{f}_\B| \leq \norm{g}_{\B^\di} \norm{f}_\B.
    \end{equation}
    
    To show that $(\B, \B^\di)$ is a dual pair, we verify non-degeneracy. First, note that the reproducing property on $\V^\di$ still holds on $\B^\di$ due to continuity of the duality pairing on $(\U, \U^\di)$. Then, observe that
    \begin{equation}
        \forall f \in \B \setminus \{0\} \colon \braket{g}{f}_\B = 0 \implies \forall (w, u) \in \Omega \times \U \colon \braket{g(w)}{u}_\U = \braket{g}{K_\U(\cdot, w) u}_\B = 0 \implies g = 0,
    \end{equation}
    where the final implication follows from the non-degeneracy of the dual pair $(\U, \U^\di)$, i.e., $\braket{g(w)}{u}_\U = 0$ for all $u \in \U$ implies $g(w) = 0$. A similar argument establishes non-degeneracy condition \eqref{nondegen_cond_pairing_1}. Hence, we obtain a continuous dual pairing $(\B, \B^\di)$, with $\B$ and $\B^\di$ both vv-RKBSs. Thus, they form an adjoint pair of vv-RKBSs.
    
    Now, suppose that \eqref{eq:u-di_norming} and \eqref{eq:kernel_bound_c} hold. Equip $\V^\di$ with the induced dual-norm
    \begin{equation}
        \norm{g}_{\V^\di} \coloneqq \sup_{\norm{f}_\B\leq 1} |\braket{g}{f}_\B|.
    \end{equation}
    Using $\norm{g(w)}_{\U^\di} = \sup_{\norm{u}_{\U} \leq 1}|\braket{g(w)}{u}_\U|$ and the definition of the pairing $\braket{\cdot}{\cdot}_\B$, we get
    \begin{equation}
        \norm{g(w)}_{\U^\di} = \sup_{\norm{u}_\U \leq 1} \left| \braket{g(w)}{u}_\U \right| = \sup_{\norm{u}_\U \leq 1} \left| \braket{g}{K_\U(\cdot, w)u}_\B \right| \leq C_\Omega(w) \sup_{\norm{f}_\B \leq 1} \left| \braket{g}{f}_\B \right| = C_\Omega(w) \norm{g}_{\V^\di},
    \end{equation}
    where the last inequality follows from
    \begin{equation}
        \norm{K_\U(\cdot, w)u}_\B = \sup_{x \in X} \frac{\norm{K_\U(x, w)u}_\U}{C_X(x)} \leq \sup_{x \in X} \frac{C_X(x) C_\Omega(w)\norm{u}_\U}{C_X(x)} = C_\Omega(w)\norm{u}_\U \leq C_\Omega(w).
    \end{equation}
    Following the same reasoning as above, taking completions yields a vv-RKBS, a continuous duality pairing, and thus, a vv-RKBS pair with the required property.
\end{proof}

Finally, for a vv-RKHS, it is known that it is isomorphic \cite[Proposition 2.7]{carmeli2006vector} to a scalar RKHS. Below, we show that this remains true even without any additional assumptions.

\begin{theorem}[Equivalence scalar RKBS (pair) and vv-RKBS (pair)]
    Let $(\B, \B^\di)$ be a vv-RKBS dual pair over $X$, $\Omega$ to a dual pair $(\U, \U^\di)$ and with kernel $K \colon X \times \Omega \to \twin(\U, \U^\di)$. Define the spaces $\B_s$ and $\B_s^\di$ as
    \begin{subequations}
        \begin{align}
            \B_s \coloneqq \{\tilde{f} \colon X \times \U^\di \rightarrow \R \mid \tilde{f}(x, u^\di)=\braket{u^\di}{f(x)}_\U,\, f\in \B  \}, \quad \norm{\tilde{f}}_{\B_s}\coloneqq \norm{f}_\B, \\
            \B_s^\di \coloneqq \{\tilde{g} \colon \Omega \times \U \rightarrow \R \mid \tilde{g}(w, u)=\braket{g(w)}{u}_\U,\, g\in \B^\di  \}, \quad \norm{\tilde{g}}_{\B_s^\di}\coloneqq \norm{g}_{\B^\di}
            \label{eq:kernel_expressions}
        \end{align}
    \end{subequations}
    the pairing between them as
    \begin{equation}
        \braket{\tilde{g}}{\tilde{f}}_{\B_s} \coloneqq \braket{g}{f}_\B
    \end{equation}
    and the scalar kernel as
    \begin{equation}
        K_s \colon (X \times \U^\di) \times (\Omega \times \U) \to \R,\;((x,u^\di), (w, u)) \mapsto K(x,w)(u^\di,u)\label{eq:kernel_Ks_definition}
    \end{equation}
    Then $(\B_s, \B_s^\di)$ is an adjoint pair of RKBSs with kernel $K_s$, $\B_s \cong \B$, and $\B_s^\di \cong \B^\di$.
\end{theorem}
\begin{proof}
    First, we show $\B_s$ is isometrically isomorphic to $\B$, where the isometric isomorphism between $\B_s^\di$ and $\B^\di$ follows from a similar argument and is hence omitted. By definition, $f \to \tilde{f} \coloneqq \braket{\cdot}{f(\cdot)}_\U$ is a linear surjective map from $\B$ to $\B_s$. To show injectivity (and hence bijectivity), assume we have $f_1$ and $f_2$ such that $\braket{u^\di}{f_1(x)}_\U = \braket{u^\di}{f_2(x)}_\U$ for all $(x, u^\di) \in X\times \U^\di$. Hence, for a given $x \in X$, $\braket{u^\di}{f_1(x) - f_2(x)}_\U = 0$ for all $u^\di \in \U^\di$. By the nondegeneracy condition \eqref{nondegen_cond_pairing_1} of the pairing $\braket{\cdot}{\cdot}_\U$, we get $f_1(x) = f_2(x)$ for every $x$. As we are dealing with a Banach space of functions, this shows that $f_1 = f_2$ and therefore shows injectivity. Finally, the isometric property follows from $\norm{\tilde{f}}_{\B_s} \coloneqq \norm{f}_\B$.

    What remains is to show that $K_s$ is the kernel of the dual pair $(\B_s, \B_s^\di)$. The combination of \eqref{eq:charachterization_twin_kernel} and \eqref{eq:kernel_Ks_definition} implies that
    \begin{subequations}
    \begin{align}
        K_s((\cdot, \bullet), (w, u)) &= \braket{\bullet}{K_{\U}(\cdot, w) u}_\U \in \B_s \\
        K_s((x,u^\di), (\cdot, \bullet)) &= \braket{K_{\U^\di}(x, \cdot) u^\di}{\bullet}_\U \in \B_s^\di
    \end{align}  \label{eq:K_s_as_function_from_scalar_RKBSs}
    \end{subequations}
    and thus
    \begin{subequations}
    \begin{align}
    \tilde{f}(x, u^\di) &= \braket{u^\di}{f(x)}_\U = \braket{K_{\U^\di}(x,\cdot)u^\di}{f}_\B = \braket{K_s((x,u^\di), (\cdot, \bullet))}{\tilde{f}}_{\B_s} \\
     \tilde{g}(w, u) &= \braket{g(w)}{u}_\U = \braket{g}{K_\U(\cdot,w)u}_\B = \braket{\tilde{g}}{K_s((\cdot, \bullet), (w, u))}_{\B_s}
    \end{align}
    \end{subequations}
    where the second equality corresponds to the reproducing kernel property of the vv-RKBS pair $(\B, \B^\di)$ and the final equality follows from $\braket{\tilde{g}}{\tilde{f}}_{\B_s} \coloneqq \braket{g}{f}_\B$ and \eqref{eq:K_s_as_function_from_scalar_RKBSs}.
\end{proof}

\section{Integral and neural RKBS} \label{sec:integral_and_neural_vv-RKBS}

While the previous sections discussed vector-valued RKBSs in general, they did not yet address their connection to neural networks. In the scalar-valued case, the relevant RKBS is the scalar integral RKBS, particularly the subclass known as neural RKBSs. We begin by reviewing scalar integral and neural RKBS pairs, then extend the discussion to the vector-valued setting, and finally examine key properties of the vector-valued integral and neural RKBS pairs.

\subsection{Scalar-valued integral and neural RKBS} \label{sec:integral_and_neural_scalar-RKBS}
The integral RKBS is defined via the feature map construction in Definition \ref{def:RKBS_feature}. In this setting, functions are constructed by integrating a feature function $\phi \in C_0(X \times \Omega)$ against a Radon measure $\mu \in \M(\Omega)$ or $\rho \in \M(X)$.

\begin{definition}[Scalar-valued integral RKBS pair]
Given locally-compact Hausdorff $X,\Omega$ and $\phi \in C_0(X \times \Omega)$, for $\mu \in \M(\Omega)$ and $\rho \in \M(X)$ define 
\begin{align}
    (A_{\Omega\to X}\mu)(x)\coloneqq \int_\Omega \phi(x,w)d\mu(w), \label{eq:scalar_integral_RKBS_f} \\
    (A_{X\to \Omega}\rho)(w) \coloneqq\int_X\phi(x,w)d\rho(x). \label{eq:scalar_integral_RKBS_g} 
\end{align}
The adjoint pair of scalar-valued RKBS $(\B, \B^\di)$ defined by
\begin{subequations}
\begin{align}
    \B \coloneqq \{ f =A_{\Omega \to X}\mu \mid \mu \in \M(\Omega) \}, & \quad \norm{f}_\B \coloneqq \inf_{f = A_{\Omega \to X}\mu} |\mu|(\Omega),  \\
    \B^\di \coloneqq \overline{\B^\di_{\text{pre}}} \coloneqq \overline{\{ g =A_{X \to \Omega}\rho \mid \rho \in \M(X) \}}, &  \quad\norm{g}_{\B^\di} \coloneqq \sup_{w \in \Omega} |g(w)|,
\end{align}
\end{subequations}
with the pairing 
\begin{equation}
    \braket{g}{f}_{\B} = \braket{\mu}{g}_{C_0(\Omega)} \stackrel{g\in\B^\di_{\text{pre}}}= \int_{X\times \Omega}\phi(x,w)d(\rho\otimes\mu)(x,w),
\end{equation}
where $\mu\in \M(\Omega)$ is any representative of $f\in \B$, $\rho\in \M(X)$ is any representative of $g\in \B_{\text{pre}}^\di$, $\braket{\mu}{g}_{C_0(\Omega)}=\int_\Omega g(w)d\mu(w)$ and $\rho\otimes\mu$ denotes the product measure, is called an integral RKBS pair, and we refer to $\B$ as an integral RKBS.
Its kernel $K \colon X \times \Omega \to \R$ is
\begin{equation}
    K(x,w) = \phi(x,w) = \braket{A_{X \to \Omega} \delta_x}{A_{\Omega \to X} \delta_w}_{\B}\qedhere
\end{equation}
\label{def:scalar_integral_RKBS}
\end{definition}
\begin{remark}
Effectively, for a given integral RKBS pair $(\B,\B^\di)$, the space $\B$ consists of functions $f\colon X\to\R$ for which there exists a $\mu\in\M(\Omega)$ such that 
\begin{equation}
    f(x) = \int_\Omega \phi(x,w)d\mu(w),
\end{equation}
and $\B^\di$ of functions $g\colon\Omega\to\R$ which are limits of functions $g_{\text{pre}}\colon\Omega\to\R$ for which there exists a $\rho\in \M(X)$ such that
\begin{equation}
    g_{\text{pre}}(w) = \int_X \phi(x,w)d\rho(x).
\end{equation}%
Since $\phi\in C_0(X\times \Omega)$, $f\in C_0(X)$ and $g\in C_0(\Omega)$.
\end{remark}

In Spek et al.\ \cite{spek2025duality}, it is proven that the pairing is well-defined, that $(\B, \B^\di)$ indeed forms an adjoint pair of RKBSs with kernel $K$, and that $(\B^\di)^*$ is isometrically isomorphic to $\B$; that is, $\B^\di$ is a predual of $\B$. In particular, note
\begin{equation}
    \left|\braket{g}{f}_\B\right| = \left|\braket{\mu}{g}_{C_0(\Omega)}\right| = \left|\int_\Omega g(w) d \mu(w)\right| \leq \left(\sup_{w\in \Omega} |g(w)| \right)|\mu|(\Omega) \implies \sup_{\norm{f}_\B \leq 1} | \braket{g}{f}_\B| \leq \sup_{w \in \Omega} |g(w)|
\end{equation}
where the last inequality follows by taking the infimum over $\mu$ such that $f = A_{\Omega \to X} \mu$. Using this, and noting that $g \in C_0(\Omega)$ attains its maximum at some $w = w_\text{max}$, we find
\begin{equation}
    \norm{g}_{\B^\di} = \sup_{w \in \Omega} |g(w)| = |\braket{\delta_{w_{\max}}}{g}_{C_0(\Omega)}| \leq \sup_{\norm{f}_\B \leq 1} | \braket{g}{f}_\B| \leq \sup_{w \in \Omega} |g(w)| = \norm{g}_{\B^\di}.
\end{equation} 
Hence, $\norm{g}_{\B^\di} = \sup_{\norm{f}_\B \leq 1} | \braket{g}{f}_\B|$. This shows that the integral RKBS has additional structure beyond that described in Definition~\ref{def:adjoint_pair_scalar_RKBS}. In particular, it corresponds to point~(4) in Theorem~\ref{thm:equivalence_adjoint_banach_space_defs}, since it can also be shown that the $ K(x, \cdot) = \phi(x, \cdot)$ functions are dense in $\B^\di$; see Theorem~\ref{thm:B_diamond_in_integral_RKBS_is_delta_dual} for a proof in the vector-valued setting.

Although the integral RKBS formulation in Definition \ref{def:scalar_integral_RKBS} is closely related to neural networks, it does not inherently exhibit the characteristic structure of an affine transformation followed by a nonlinear activation. This structure can be recovered by selecting a specific form of $\phi$ within the integral RKBS framework. 
\begin{definition}[Scalar-valued neural RKBS pairs]
    Let $(\V, \V^\di)$ be a dual pair of normed vector spaces, and assume $X \subseteq \V$ is compact and $\widetilde{\Omega} \subseteq \V^\di$ locally-compact Hausdorff. Let $\Omega \coloneqq \widetilde{\Omega} \times \R$, let $\sigma \colon \R \to \R$ be a measurable activation function, and let $\beta \colon \Omega \to \R$ be a measurable positive function such that $\phi \in C_0(X \times \Omega)$, where
    \begin{equation}
        \phi(x,w) = \phi(x, (\omega, b)) \coloneqq \sigma(\braket{\omega}{x}_\V + b) \beta(w)
        \end{equation}
    with $w = (\omega, b) \in \Omega$.
    
    If $(\B, \B^\di)$ is an integral RKBS pair with kernel $\phi$, then we refer to it as a neural RKBS pair.    
\end{definition}
\begin{remark}
The function $\beta$ in the above definition ensures that the kernel integrals in Equations \eqref{eq:scalar_integral_RKBS_f} and \eqref{eq:scalar_integral_RKBS_g} are well defined for commonly used activation functions. One could avoid introducing the weighting function $\beta$ by modifying the definition of the integral RKBS. For example, we can restrict to measures for which the integrals are already well-defined. Another approach is taken in Barron spaces, a well-known example of a neural RKBS that also treats neural networks but incorporates $\beta$ into the total variation norm of the measures instead \cite{ma2019priori, ma2020towards, spek2025duality}. Yet another option is to consider subspaces of measures satisfying moment conditions, which are dual to spaces of continuous functions with controlled growth \cite{bartolucci2024lipschitz}.

Although such alternatives exist, we retain the $C_0(X \times\Omega)$ assumption, as it simplifies the analysis without introducing significant limitations. In practice, exceedingly large weights are rarely desirable, so applying a continuous cut-off to zero for such values is both reasonable and effective.
\end{remark}

\subsection{Vector-valued integral and neural RKBS}
To extend the previous section from scalar outputs to outputs in a Banach space $\U$, two natural approaches would be to either make $\phi$ a $\U$-valued function, or to replace the scalar-valued measures with $\U$-valued ones. We follow the second approach, as also done in Bartolucci et al.\ \cite{bartolucci2024neural}. With the first option, it is challenging to ensure that the represented functions are expressive enough to take any value in a high- or even infinite-dimensional space $\U$ at each domain point. It would require us to work explicitly with parameter spaces $\Omega$ of comparable dimension and with measures defined on them. In contrast, by working with vector-valued measures, we can directly use any scalar-valued $\phi$ while maintaining expressivity in terms of the attained values in $\U$. This approach generalizes the standard vv-RKHS construction starting from a scalar kernel of the form \eqref{eq:vv-RKHS-from-scalar}. Moreover, it allows us to work with spaces of measures on locally-compact parameter domains, which behave significantly better in terms of duality.

For a Banach space $\U$ and $\Omega$ locally-compact Hausdorff, we define $\M(\Omega; \U)$ to be the set of regular countably additive $\U$-valued vector measures over $\Omega$ equipped with the Borel $\sigma$-algebra and with finite total variation. In this setting, a vector-valued generalization \cite[Theorem 1]{meziani2009dual} of the Riesz representation theorem provides an isometric isomorphism
\begin{equation}
    C_0(\Omega; \U)^* \cong \M(\Omega; \U^*).
    \label{eq:duality_C0_and_RadonMeasureSpace}
\end{equation}
In the special case when $\Omega$ is compact Hausdorff, this duality result is known as Singer's representation theorem \cite{singer1957linear}. In the context of Hilbert space-valued measures on second countable locally-compact Hausdorff domains, it also appears as Theorem 9 of Carmeli et al.\ \cite{carmeli2010vector}.

For the scalar-valued case $\U=\U^*=\R$, the pairing between $C_0(\Omega; \U)$ and $\M(\Omega; \U^*)\cong C_0(\Omega; \U)^*$  appeared in Definition~\ref{def:adjoint_pair_scalar_RKBS}. There, it serves as the pairing between $\B$ and $\B^\di$ in the scalar-valued integral RKBS pair and is defined by
\begin{equation*}
    \braket{\mu}{g}_{C_0(\Omega)} = \int_\Omega g(w)\, d\mu(w).
\end{equation*}
In the vector-valued analogue of Definition~\ref{def:adjoint_pair_scalar_RKBS}, we similarly want to introduce a pairing analogous to the $C_0(\Omega; \U)$--$\M(\Omega; \U^*)$ pairing.

To this end, let $(\U, \U^\di)$ be a dual pair equipped with a continuous duality pairing and write a simple function $g\colon \Omega \to \U^\di$ as
\begin{equation}
    g = \sum_{j=1}^m u_j^\di \mathbbm{1}_{B_j}
\end{equation}
where the sets $B_j \subseteq \Omega$ are disjoint. For $\mu \in \M(\Omega; \U)$, we then define
\begin{equation}
    \braket{\mu}{g}_{C_0(\Omega;\U^\di)} \coloneqq \int_\Omega \braket{g(w)}{d\mu(w)}_\U \coloneqq \sum_{j=1}^m \braket{u_j^\di}{\mu(B_j)}_\U.\label{eq:braket_integral_defining_equation}
\end{equation}
Since $\braket{\cdot}{\cdot}_\U : \U^\di \times \U \to \R$ is continuous, there exists $C > 0$ such that $\braket{u^\di}{u}_\U \leq C \norm{u^\di}_{\U^\di} \norm{u}_\U$ for all $u^\di \in \U^\di$ and $u \in \U$. Consequently, the linear operator $g \mapsto \braket{\mu}{g}_{C_0(\Omega;\U^\di)}$ is bounded on the space of simple functions
\begin{equation}
\begin{split}
    \left|\braket{\mu}{g}_{C_0(\Omega;\U^\di)} \right| &= \left| \sum_{j=1}^m \braket{u_j^\di}{\mu(B_j)}_\U \right| \leq C \sum_{i=1}^m \norm{u_j^\di}_{\U^\di} \norm{\mu(B_j)}_\U \leq C\left(\max_{i=1, \ldots, m} \norm{u_j^\di}\right) \sum_{j=1}^m \norm{\mu(B_j)}_\U \\
    & = C\left( \sup_{w \in \Omega} \norm{g(w)}_{\U^\di} \right) \sum_{j=1}^m \norm{\mu(B_j)}_\U \leq C \left( \sup_{w \in \Omega} \norm{g(w)}_{\U^\di} \right) |\mu|_\U(\Omega).
\end{split}
\label{eq:bound_g_braket_mu}
\end{equation}
where the last inequality follows from the definition of the total variation $|\mu|_\U(\Omega)$. 

To extend this definition beyond simple functions, we consider the space of functions that are uniform limits of simple functions and equip this space with the supremum norm. The upper bound $C \left( \sup_{w \in \Omega} \norm{g(w)}_{\U^\di} \right) |\mu|_\U(\Omega)$ remains computable for all functions in this space in which the simple functions are dense by definition. Hence, we can extend $g \mapsto \braket{\mu}{g}_{C_0(\Omega;\U^\di)}$ continuously. For $f \colon X \to \U$ and $\rho \in \M(X; \U^\di)$, we define the pairing $\braket{\rho}{f}_{C_0(X;\U)}$ in the same way.
\begin{remark}
    The above construction coincides with the integral in Appendix~A of Carmeli et~al.\ \cite{carmeli2010vector}. They extend using the Hahn-Banach theorem to the space $L^1(\Omega,|\mu|_\U;\U^\di)$ of functions from $\Omega$ to $\U^\di$ that are Bochner integrable with respect to the measure $|\mu|_\U$. While such an extension would also be possible in our setting, our simpler extension suffices as it immediately preserves the bound in \eqref{eq:bound_g_braket_mu}, which is crucial for establishing the continuity of the pairing $\braket{\cdot}{\cdot}_\B$.
\end{remark}
With the above definition of $\braket{\cdot}{\cdot}_{C_0(\Omega;\U^\di)}$, we can now define integral and neural vv-RKBS pairs.

\begin{definition}[Vector-valued integral and neural RKBS]
Given locally-compact Hausdorff $X,\Omega$, $\phi \in C_0(X \times \Omega)$, and a continuous dual pair $(\U, \U^\di)$ of Banach spaces, for $\U$-valued measure $\mu \in \M(\Omega; \U)$ and $\U^\di$-valued measure $\rho \in \M(X;\U^\di)$ define 
\begin{align}
    (A_{\Omega\to X}\mu)(x)\coloneqq \int_\Omega \phi(x,w)d\mu(w), \label{eq:vv_integral_RKBS_f} \\
    (A_{X\to \Omega}\rho)(w) \coloneqq\int_X\phi(x,w)d\rho(x). \label{eq:vv_integral_RKBS_g} 
\end{align}
The adjoint pair of vv-RKBS $(\B, \B^\di)$ defined by
\begin{subequations}
\begin{align}
    \B \coloneqq \{ f =A_{\Omega \to X}\mu \mid \mu \in \M(\Omega; \U) \} , & \quad \norm{f}_\B \coloneqq \inf_{f = A_{\Omega \to X}\mu} |\mu|_\U(\Omega), \\
    \B^\di \coloneqq \overline{\B^\di_{\text{pre}}} \coloneqq \overline{\{ g =A_{X \to \Omega}\rho \mid \rho \in \M(X; \U^\di) \}}, &  \quad\norm{g}_{\B^\di} \coloneqq \sup_{w \in \Omega} \norm{g(w)}_{\U^\di},
\end{align}
\end{subequations}
with the pairing 
\begin{equation}
    \braket{g}{f}_{\B} = \braket{\mu}{g}_{C_0(\Omega; \U^\di)} \stackrel{g\in\B^\di_{\text{pre}}}= \int_{X \times \Omega} \phi(x,w) d \braket{\rho}{\mu}_\U(x,w),
\end{equation}
where $\mu\in \M(\Omega;\U)$ is any representative of $f\in \B$, $\rho\in \M(X;\U^\di)$ is any representative of $g\in \B_{\text{pre}}^\di$ and $\braket{\rho}{\mu}_\U$ denotes the Hahn-Kolmogorov extension of the scalar-valued measure satisfying $\braket{\rho}{\mu}_\U(E\times F) = \braket{\rho(E)}{\mu(F)}_\U$ for any Borel subsets $E \subset X, F \subset \Omega$, is called an integral vv-RKBS pair, and we refer to $\B$ as an integral RKBS.
Its kernel $K \colon X \times \Omega \to \twin(\U, \U^\di)$ satisfies
\begin{equation}
\begin{split}
    K_\U(x,w)u \coloneqq \phi(x,w) u, & \qquad K_{\U^\di}(x,w)u^\di \coloneqq \phi(x,w) u^\di \\
    K(x,w)(u^\di, u) &= \phi(x,w) \braket{u^\di}{u}_\U.
\end{split}    
\end{equation}
If, in addition, $(\V, \V^\di)$ is a dual pair of normed vector spaces with $X \subseteq \mathcal V$ compact, $\widetilde{\Omega} \subseteq \mathcal V^\di$ locally-compact Hausdorff, $\Omega \coloneqq \widetilde{\Omega} \times \mathbb R$, and 
\begin{equation}
    \phi \in C_0(X \times \Omega),\quad \phi(x,(\omega,b)) \coloneqq \sigma(\braket{\omega}{x}_\V + b) \beta((\omega,b)),
\end{equation}
for a measurable activation $\sigma \colon \mathbb R \to \mathbb R$ and a measurable positive $\beta \colon \Omega \to \mathbb R$, then $(\mathcal B, \mathcal B^\di)$ is called a neural vv-RKBS pair.
\label{def:vector_integral_RKBS}
\end{definition}

\begin{remark}[Motivating the $\phi \in C_0(X \times \Omega)$ assumption]
    Although all the operations are defined for bounded measurable $\phi$, the definition uses the standard $\phi \in C_0(X \times \Omega)$ assumption, which already covers many cases. In particular, for $\R^d$-valued neural networks, $X \subseteq \R^{d_x}$ can be taken compact, since real-world inputs are always bounded. With $\Omega \subseteq \R^{d_\Omega}$ and bounded weights (e.g. through $\beta$), this ensures $\phi \in C_0(X \times \Omega)$.
    
    In neural operators, weights also lie in $\Omega \subseteq \R^{d_\Omega}$, but inputs are theoretically infinite-dimensional functions, making compactness a strong assumption. The finite-dimensional manifold hypothesis, however, suggests that data effectively lies on a finite-dimensional manifold. Moreover, in practice, inputs are given as meshes or point clouds. Many methods first project them to finite-dimensional (function) spaces \cite{raonic2023convolutional, dummer2026ronom, batlle2024kernel, bhattacharya2021model} or latent spaces \cite{seidman2022nomad, kontolati2023learning}. In all presented cases, the inputs or their projections lie in a finite-dimensional space, where compactness assumptions are more natural. Finally, in supervised learning, we consider only a finite number of training data points $x_n \in X$, which form a compact set, and therefore also make $\phi \in C_0(X \times \Omega)$ natural.

    Given this and the simplification it provides for analysis, we retain the $\phi \in C_0(X\times \Omega)$ assumption.
\end{remark}

\begin{remark}
Since Definition \ref{def:vector_integral_RKBS} is based on the real-valued function $\phi$ and leads to the kernel $\phi(x,w) \braket{u^\di}{u}_\U$, it generalizes the Hilbert space construction in \eqref{eq:vv-RKHS-from-scalar}.
\end{remark}

\begin{remark}
    The Hahn–Kolmogorov extension theorem (see Theorem 1.7.8 in Tao \cite{tao2011introduction}) is usually stated for nonnegative pre-measures, whereas in Definition \ref{def:vector_integral_RKBS} we are dealing with a signed set function. Moreover, for the pairing $\braket{\cdot}_\B$ to be well-defined, the extension $\braket{\rho}{\mu}_\U$ must be of bounded variation, which is not addressed by the Hahn–Kolmogorov theorem. To ensure that the extension is both well-defined and of bounded total variation, we establish this result in Theorem \ref{thm:existence_braket_measure}.
\end{remark}

Similar to the treatment of the scalar-valued case in \cite{spek2025duality}, it remains to show that the vector-valued pairing is independent of the choice of representatives $\mu$ and $\rho$ for $f \in \B$ and $g \in \B_{\mathrm{pre}}^\di$, respectively. We now prove that the pairing $\braket{\cdot}{\cdot}_\B$ in Definition~\ref{def:adjoint_pair_vector_RKBS} is well-defined.
\begin{theorem}[Pairing of integral vv-RKBS is well-defined]
    Let $(\U, \U^\di)$ be a dual pair of Banach spaces with a continuous pairing. Let $(\B, \B^\di)$ be a vector-valued integral RKBS pair, where functions in $\B$ map from a set $X$ to $\U$ and functions in $\B^\di$ map from a set $\Omega$ to $\U^\di$. For any $\mu \in \M(\Omega; \U)$ and $\rho \in \M(X; \U^\di)$, if $f = A_{\Omega \to X}\mu$ and $g = A_{X \to \Omega} \rho$, then
    \begin{equation}
        \braket{g}{f}_\B = \braket{\mu}{g}_{C_0(\Omega;\U^\di)} = \braket{\rho}{f}_{C_0(X;\U)}
    \end{equation}
    In particular, the pairing $\braket{\cdot}{\cdot}_\B$ is independent of the particular representations $\mu$ and $\rho$ of $f$ and $g$, respectively. Moreover, by continuity of the pairing on $C_0(\Omega;\U^\di)$, this independence extends to all $g\in\B^\di$: if $g_k=A_{X\to\Omega}\rho_k\to g$, then
    \begin{equation}
        \braket{g}{f}_\B
        = \braket{\mu}{g}_{C_0(\Omega;\U^\di)}
        = \lim_{k\to\infty}\braket{\mu}{g_k}_{C_0(\Omega;\U^\di)}
        = \lim_{k\to\infty}\braket{\rho_k}{f}_{C_0(X;\U)},
    \end{equation}
    which is independent of the choice of $\mu$.
    \label{thm:vv-RKBS_pairing_is_well-defined}
\end{theorem}

\begin{proof}
    As $\phi \in C_0(X \times \Omega)$, Theorem~\ref{thm:simple_approximation_of_C0_functions} guarantees that for any $\varepsilon > 0$, there exist:
    \begin{itemize}
        \item pairwise disjoint Borel sets $A_1, \dots, A_n \subseteq X$ with $\bigcup_{i=1}^n A_i = X$,
        \item pairwise disjoint Borel sets $B_1, \dots, B_m \subseteq \Omega$ with $\bigcup_{j=1}^m B_j = \Omega$,
    \end{itemize}
    and a simple function of the form
    \begin{equation}
    \phi_\varepsilon(x, w) = \sum_{i=1}^n \sum_{j=1}^m a_{ij} \mathbbm{1}_{A_i \times B_j}(x, w)
    \end{equation}
    such that
    \begin{equation}
    \sup_{(x,w) \in X \times \Omega} \left| \phi_\varepsilon(x, w) - \phi(x, w) \right| \leq \varepsilon.
    \end{equation}
    
    By uniform convergence, we have:
    \begin{equation}
    \braket{g}{f}_\B = \int_{X \times \Omega} \phi(x, w) \, d\braket{\rho}{\mu}_\U(x,w) = \lim_{\varepsilon \to 0} \int_{X \times \Omega} \phi_\varepsilon(x, w) \, d\braket{\rho}{\mu}_\U(x,w).
    \end{equation}
    The integral in the limit can be written equivalently as
    \begin{equation}
    \begin{split}
    \int_{X \times \Omega} \phi_\varepsilon(x, w) \, d\braket{\rho}{\mu}_\U(x,w) 
    &= \sum_{i=1}^n \sum_{j=1}^m a_{ij} \braket{\rho(A_i)}{\mu(B_j)}_\U \\
    &= \sum_{j=1}^m \braket{\sum_{i=1}^n a_{ij} \rho(A_i)}{\mu(B_j)}_\U = \braket{\mu}{g_\varepsilon}_{C_0(\Omega;\U^\di)}, 
    \end{split}\label{eq:pairing_simple_function}
    \end{equation}
    by using the simple function
    \begin{equation}
    g_\varepsilon(w) \coloneqq \sum_{j=1}^m \left( \sum_{i=1}^n a_{ij} \rho(A_i) \right) \mathbbm{1}_{B_j}(w).
    \end{equation}
    Hence,
    \begin{equation}
    \braket{g}{f}_\B = \lim_{\varepsilon \to 0} \braket{\mu}{g_\varepsilon}_{C_0(\Omega;\U^\di)}.
    \end{equation}

    We will now show that $\lim_{\varepsilon \to 0} \braket{\mu}{g_\varepsilon}_{C_0(\Omega; \U^\di)} = \braket{\mu}{g}_{C_0(\Omega; \U^\di)}$.  
    For fixed $\mu$, the map $g \mapsto \braket{\mu}{g}_{C_0(\Omega; \U^\di)}$ is bounded on the space of functions that can be expressed as uniform limits of simple functions, where this space is equipped with the supremum norm. Hence, it suffices to prove that $g_\varepsilon \to g$ uniformly.

    To see this, fix $w \in B_j$. Because the $A_i$ are pairwise disjoint and cover $X$, we have:
    \begin{equation}
    \begin{split}
    \norm{g_\varepsilon(w) - g(w)}_{\U^\di} 
    &= \left\| \sum_{i=1}^n a_{ij} \rho(A_i) - \int_X \phi(x, w) \, d\rho(x) \right\|_{\U^\di} \\
    &= \left\| \sum_{i=1}^n \int_{A_i} (a_{ij} - \phi(x, w)) \, d\rho(x) \right\|_{\U^\di} \\
    &\leq \sum_{i=1}^n \left\| \int_{A_i} (a_{ij} - \phi(x, w)) \, d\rho(x) \right\|_{\U^\di} \\
    &\leq \sum_{i=1}^n \sup_{x \in A_i} |a_{ij} - \phi(x, w)| \cdot |\rho|_{\U^\di}(A_i) \\
    &= \sum_{i=1}^n \sup_{x \in A_i} |\phi_\varepsilon(x, w) - \phi(x, w)| \cdot |\rho|_{\U^\di}(A_i) \\
    & \leq \varepsilon \sum_{i=1}^n |\rho|_{\U^\di}(A_i) = \varepsilon |\rho|_{\U^\di}(X).
    \end{split}        
    \end{equation}
    Since $|\rho|_{\U^\di}$ has finite total variation and is independent of $\varepsilon$, this shows:
    \begin{equation}
        \sup_{w \in \Omega} \norm{g_\varepsilon(w) - g(w)}_{\U^\di} \to 0 \quad \text{as } \varepsilon \to 0.
    \end{equation}
    Hence,
    \begin{equation}
        \lim_{\varepsilon \to 0} \braket{\mu}{g_\varepsilon}_{C_0(\Omega;\U^\di)} = \braket{\mu}{g}_{C_0(\Omega;\U^\di)}.
    \end{equation}
    Combining this with \eqref{eq:pairing_simple_function}, we conclude:
    \begin{equation}
    \braket{g}{f}_\B = \braket{\mu}{g}_{C_0(\Omega;\U^\di)}.
    \end{equation}    
    A completely analogous argument, interchanging the roles of $\rho$ and $\mu$, shows that
    \begin{equation}
    \braket{g}{f}_\B = \braket{\rho}{f}_{C_0(X;\U)}.\qedhere
    \end{equation}
\end{proof}

\subsection{Properties of integral vv-RKBS}
Having established that all components in the definition of the integral vector-valued RKBS pair are well-defined, the question remains whether all the requirements for an adjoint pair of vector-valued RKBS are satisfied. The following theorem confirms this, and in particular shows that the pairing $\braket{\cdot}{\cdot}_{\B}$ inherits the continuity bound from the pairing $\braket{\cdot}{\cdot}_{\U}$.

\begin{theorem}[Integral vv-RKBS pair is adjoint pair of vv-RKBS]\label{thm:integral_RKBS_is_vv-RKBS}
    The integral vv-RKBS pair $(\B, \B^\di)$ is an adjoint pair of vv-RKBS. In particular, if $C > 0$ is the smallest constant such that the continuous pairing $\braket{\cdot}{\cdot}_\U$ satisfies  $\left|\braket{u^\di}{u}_\U\right| \leq C \norm{u^\di}_{\U^\di} \norm{u}_\U$ for all $(u, u^\di) \in \U \times \U^\di$, then the pairing $\braket{\cdot}{\cdot}_\B$ is continuous with the bound
    \begin{equation}
        \left|\braket{g}{f}_\B \right|\leq C \norm{g}_{\B^\di}\norm{f}_\B
    \end{equation}
    for all $(f,g) \in \B \times \B^\di$.
\end{theorem}
\begin{proof}
    To prove that the integral vector-valued RKBS pair $(\B, \B^\di)$ forms an adjoint pair of vv-RKBSs, we verify the following points:

    \underline{\textbf{$\B$ and $\B^\di$ are vv-RKBS:}} \\
    The space $\B^\di$ is a vv-RKBS by construction, as it uses the supremum norm. For $\B$, consider an arbitrary $x\in X$. For any $f = A_{\Omega \to X} \mu$ with $\mu \in \M(\Omega; \U)$, 
    \begin{equation}
        \norm{f(x)}_\U = \norm{\int_\Omega \phi(x,w) d\mu(w)}_\U \leq \sup_{w \in \Omega}\left|  \phi(x,w) \right| |\mu|_\U(\Omega),
    \end{equation}
    where $\sup_{w\in \Omega}\left| \phi(x,w) \right| < \infty$ due to $\phi \in C_0(X \times \Omega)$. Taking the infimum over all $\mu$ such that $f = A_{\Omega \to X} \mu$, we obtain
    \begin{equation}
        \norm{f(x)}_\U \leq \sup_{w\in \Omega}\left| \phi(x,w) \right| \norm{f}_\B
    \end{equation}
    Since $x$ was arbitrary, all the point evaluations are continuous. Hence, $\B$ is a vv-RKBS.

    \underline{\textbf{\textit{The pairing is continuous}}}: \\
    This follows immediately from \eqref{eq:bound_g_braket_mu} when taking the infimum over all  $\mu$ such that $A_{\Omega \to X}\mu = f$.
    
    \underline{\textbf{$K$ satisfies the reproducing properties:}} \\
    Let $f = A_{\Omega \to X} \mu \in \B$. For any $x \in X$ and $u^\di \in \U^\di$, we have
    \begin{equation}
        \braket{u^\di}{f(x)}_\U = \int_\Omega \phi(x, w) d\braket{u^\di}{\mu}_\U (w) = \int_{X\times \Omega} \phi(y,w) d \braket{u^\di \delta_x}{\mu}_\U (y, w) = \braket{\phi(x,\cdot)u^\di}{f}_\B,
    \end{equation}
    which verifies the reproducing property for $K_{\U^\di}$. The reproducing property for $K_{\U}$ follows by a similar argument and taking a limit of $g_k \in \B_{\text{pre}}^\di \to g \in \B^\di$:
    \begin{equation}
        \braket{g(w)}{u}_\U = \lim_{k \rightarrow \infty}\braket{g_k(w)}{u}_\U = \lim_{k \rightarrow \infty}\braket{g_k}{\phi(\cdot,w)u}_\B = \braket{g}{\phi(\cdot,w)u}_\B,
    \end{equation}
    where the first equality follows by supremum norm convergence of $g_k$ to $g$ and continuity of the $(\U, \U^\di)$ pairing and the last equality follows by continuity of the $(\B, \B^\di)$ pairing. Furthermore, by the definition of a twin operator:
    \begin{equation}
        K(x,w)(u^\di, u) = \braket{u^\di}{K_\U(x,w) u}_\U = \braket{u^\di}{\phi(x,w) u}_\U=\phi(x,w)\braket{u^\di}{u}_\U 
    \end{equation}

    \underline{\textbf{\textit{The pairing is non-degenerate}}}: \\
    Assume $g \in \B^\di$ is not the zero function. Then there exists a $w \in \Omega$ such that $g(w) \neq 0$. Since \((\U, \U^\di)\) is a dual pair, the non-degeneracy of \(\braket{\cdot}{\cdot}_\U\) implies that there exists a \(u \in \U\) such that \(\braket{g(w)}{u}_\U \neq 0\). Using the reproducing property,
    \begin{equation}
        0 \neq \braket{g(w)}{u}_\U = \braket{g}{K_\U(\cdot, w) u}_\B.
    \end{equation}
    Hence, there exists an $f \in \B$ with $\braket{g}{f}_\B \neq 0$, proving non-degeneracy in the second argument. A similar argument with the roles of $f$ and $g$ reversed shows the non-degeneracy in the first argument. 
\end{proof}

The previous theorem demonstrates that the integral vv-RKBS pair forms an adjoint pair of vv-RKBS when using Definition \ref{def:adjoint_pair_vector_RKBS}. As highlighted in Theorem \ref{thm:equivalence_adjoint_banach_space_defs}, this definition is the weakest one possible. The following theorem establishes that, under limited assumptions on the output dual pair $(\U, \U^\di)$, the integral vector-valued RKBS pair satisfies additional properties that correspond to the $\delta$-dual definition.

\begin{theorem}[$\B^\di$ in the integral RKBS is a $\delta$-dual]\label{thm:B_diamond_in_integral_RKBS_is_delta_dual}
    Let $(\U, \U^\di)$ be a dual pair with continuous pairing and let $(\mathcal{B}, \mathcal{B}^\diamond)$ be a vector-valued integral RKBS pair, where functions in $\mathcal{B}$ map from a set $X$ to $\mathcal{U}$ and functions in $\mathcal{B}^\diamond$ map from a set $\Omega$ to $\mathcal{U}^\diamond$.  
    
    Then the set 
    \begin{equation}
    \spn\left\{ K_{\mathcal{U}^\diamond}(x, \cdot) u^\diamond\mid (x, u^\diamond) \in X \times \mathcal{U}^\diamond \right\} = \spn\left\{ \phi(x, \cdot) u^\diamond \mid (x, u^\diamond) \in X \times \mathcal{U}^\diamond \right\}
    \end{equation}
    is dense in $\mathcal{B}^\diamond$.  
    
    Furthermore, if 
    \begin{equation}
        \norm{u^\diamond}_{\mathcal{U}^\diamond} = \sup_{\norm{u}_{\U} \leq 1} \left| \braket{u^\diamond}{u}_{\U} \right|,
        \label{eq:dual_sup_norm_assumption_U}
    \end{equation}
    then 
    \begin{equation}
        \norm{g}_{\mathcal{B}^\diamond} = \sup_{\norm{f}_{\B} \leq 1} \left| \braket{g}{f}_{\B} \right|.
        \label{eq:dual_sup_norm_consequence_B}
    \end{equation}    
\end{theorem}
\begin{proof}
    To prove the density, we apply Lemma \ref{lemma:C0_domain_decomposition} to obtain pairwise disjoint sets $A_1, \dots, A_n \subseteq X$ with $\bigcup_{i=1}^n A_i = X$, pairwise disjoint sets $B_1, \dots, B_m \subseteq \Omega$ with $\bigcup_{j=1}^m B_j = \Omega$, and a Borel set $D_{\frac{\varepsilon}{2}} \subseteq X \times \Omega$, such that for each $(i, j)$ either:
    \begin{itemize}
        \item $A_i \times B_j \subseteq D_{\frac{\varepsilon}{2}}$ and $|\phi(x,w) - \phi(\widetilde{x}, \widetilde{w})| \leq \epsilon$ for all $(x,w), (\widetilde{x}, \widetilde{w}) \in A_i \times B_j$,
        \item $A_i \times B_j \subseteq \left( X \times \Omega \right) \setminus D_{\frac{\varepsilon}{2}}$ and $|\phi(x,w)| \leq \frac{\varepsilon}{2}$ for all $(x,w) \in A_i \times B_j$, which implies by the triangle inequality that $|\phi(x,w) - \phi(\widetilde{x}, \widetilde{w})| \leq \epsilon$ for all $(x,w), (\widetilde{x}, \widetilde{w}) \in A_i \times B_j$. 
    \end{itemize}
    In particular, for an arbitrary $x_i \in A_i$, we have
    \begin{equation}
        \sup_{x \in A_i} |\phi(x_i,w) - \phi(x, w)| \leq \sup_{x \in A_i, w \in \Omega} |\phi(x_i,w) - \phi(x, w)| = \max_j\left(\sup_{x \in A_i, w \in B_j} |\phi(x_i,w) - \phi(x, w)|\right) \leq \varepsilon
    \end{equation}
    where the equality follows from $\bigcup_{j=1}^m B_j = \Omega$. 
    
    Let $g$ be an arbitrary element of $\B^\di_{\text{pre}}$. For arbitrary $x_i \in A_i$, define $g_\varepsilon$ as:
    \begin{equation}
        g_\varepsilon = \sum_{i=1}^n \phi(x_i, \cdot) \rho(A_i) \in \spn\set{\phi(x, \cdot) u^\di \given (x, u^\di) \in X \times \U^\di}
    \end{equation}
    Then, for every $w \in \Omega$,
    \begin{equation}
    \begin{split}
        \norm{g_\varepsilon(w) - g(w)}_{\U^\di} & = \norm{\sum_{i=1}^n \int_{A_i} (\phi(x_i, w) - \phi(x,w)) d\rho(x)}_{\U^\di} \\
        & \leq \sum_{i=1}^n  \sup_{x \in A_i} |\phi(x_i, w) - \phi(x,w)| |\rho|_{\U^\di}(A_i) \\
        & \leq \varepsilon \sum_{i=1}^n |\rho|_{\U^\di}(A_i) = \varepsilon |\rho|_{\U^\di}(X)
    \end{split}
    \end{equation}
    Hence,
    \begin{equation}
        \norm{g_\varepsilon - g}_{\B^\di} = \sup_{w \in \Omega} \norm{g_\varepsilon(w) -g(w)}_{\U^\di} \leq \varepsilon |\rho|_{\U^\di}(X) \xrightarrow{\varepsilon \to 0} 0
    \end{equation}
    which shows that $\spn\set{\phi(x, \cdot) u^\di \given (x, u^\di) \in X \times \U^\di} = \spn
    \set{K_{\U^\di}(x, \cdot)u^\di \given (x, u^\di) \in X \times \U^\di}$ is dense in $\B^\di_{\text{pre}}$ and thus also in $\B^\di$.

    To prove \eqref{eq:dual_sup_norm_consequence_B} given \eqref{eq:dual_sup_norm_assumption_U}, observe that every $g\in \B^\di\subset C_0(\Omega;\U^\di)$ attains the supremum in its norm at some point $w_{\text{max}}$. Hence,
    \begin{equation}
        \norm{g}_{\B^\di} = \norm{g(w_{\text{max}})}_{\U^\di} = \sup_{\norm{u}_\U \leq 1} \left| \braket{g(w_{\text{max}})}{u}_\U \right| = \sup_{\norm{u}_\U \leq 1} \left| \braket{g}{K_\U(\cdot, w_{\text{max}}) u}_\B \right| \leq \sup_{\norm{f}_\B \leq 1} \left| \braket{g}{f}_\B \right|
    \end{equation}
    where the last inequality holds since
    \begin{equation}
        \norm{K_\U(\cdot, w_{\text{max}}) u}_\B = \norm{\phi(\cdot, w_{\text{max}})u}_\B \leq |\delta_{w_{\text{max}}}|(\Omega)\norm{u}_\U \leq 1,
    \end{equation}
    with the final inequality following from the assumption $\norm{u}_\U \leq 1$. 
    
    Moreover, if \eqref{eq:dual_sup_norm_assumption_U} holds, then $C=1$ in Theorem \ref{thm:integral_RKBS_is_vv-RKBS}. Combining this with the above yields
    \begin{equation}
        \norm{g}_{\B^\di} \leq \sup_{\norm{f}_\B \leq 1} \left| \braket{g}{f}_\B \right| \leq \sup_{\norm{f}_\B \leq 1} \norm{g}_{\B^\di} \norm{f}_\B \leq \norm{g}_{\B^\di}. 
    \end{equation}
    Hence, \eqref{eq:dual_sup_norm_consequence_B} holds.
\end{proof}

\subsection{Representer theorem}
So far, we have introduced the integral vv-RKBS and the neural RKBS as infinite-dimensional formulations of neural networks. In practice, we aim to optimize the weights of a neural network using data. Thus, if the integral vv-RKBS is to serve as the function space for neural networks, optimizing over this space should return neural networks as solutions. Results of this type are known as representer theorems. This section establishes a general representer theorem for the integral vv-RKBS introduced in Definition \ref{def:vector_integral_RKBS}. In the next section, we demonstrate that $\R^d$-valued neural networks, hypernetworks, and DeepONets are elements of specific integral and neural vv-RKBS spaces, and utilize the general representer theorem to derive representer theorems for these models.

To set up the optimization problem required for learning from data, assume we want to learn a function from a domain $X$ to a Banach space $\U$. For $\R^d$-valued neural networks, $X \subseteq \R^{d_x}$ and $\U \subseteq \R^{d_u}$. In operator learning, $X$ may be $\R^{d_x}$ or even a function space, with the output also lying in a function space. In practice, however, we rarely have access to the full output function and instead observe only its sampled values. Concretely, every output function $u \colon D \to \R$ is sampled at $d$ points $p_i \in D$ through the measurement operator
\begin{equation}
    Mu = [u(p_1), \ldots, u(p_d)] \in \R^d.
\end{equation}
When the output function space is taken to be an RKBS, the measurement operator $M$ is a bounded linear operator from the RKBS to $\R^d$.  

In general, we assume a measurement operator $M \colon \U \to \R^d$, such as the sampling operator above or, when $\U = \R^d$, the identity map. Given data $\{(x_n, y_n)\}_{n=1}^N \subseteq X \times \R^d$, we then formulate a supervised optimization problem, for which we derive a representer theorem.

\begin{theorem}
    Assume we are given data $\{(x_n, y_n)\}_{n=1}^N \subseteq X \times \R^d$ .  
    Let $L \colon \R^d \times \R^d \to \R$ be such that for every $y \in \R^d$, the map $L(\cdot, y)$ is convex, coercive, and lower semicontinuous.  
    Let $\B$ be an integral vv-RKBS of functions from $X$ to a Banach space $\U$ with predual $\U_*$ and with weight space $\Omega$.   
    Suppose there exists a bounded linear operator $M \colon \U \to \R^d$ of the form
    \begin{equation}
    (Mu)_j = \braket{u}{v_j}_{\U_*},
    \end{equation}
    where $v_j\in \U_\ast$, $\braket{\cdot}{\cdot}_{\U_*}$ is the canonical duality pairing between ${\U_*}$ and $(\U_*)^*=\U$.  
    
    Then there exists a solution to the supervised learning problem
    \begin{equation}
        \min_{f \in \B} \; \frac{1}{N} \sum_{n=1}^N L(M(f(x_n)), y_n) + \lambda \norm{f}_\B, 
        \qquad \lambda \geq 0,
        \label{eq:sup_opt_problem}
    \end{equation}
    that admits a representation of the form
    \begin{equation}
    \mu^\dagger = \sum_{m=1}^{Nd} a_m \, \delta_{w_m} u_m,
    \qquad
    f^\dagger = A_{\Omega \to X} \mu^\dagger 
      = \sum_{m=1}^{Nd} a_m \, \phi(\cdot, w_m) u_m,
    \end{equation}
    with $u_m \in \operatorname{Ext}(\{u \in \U : \norm{u}_\U \leq 1\})$ and $a_m \in \R$.
    \label{thm:representer_theorem_integral_RKBS}
\end{theorem}

\begin{proof}
    Assume that there exists $\mu^\dagger$ with
    \begin{equation}
        \mu^\dagger \in \argmin_{\mu \in \M(\Omega; \U)} \; \frac{1}{N}\sum_{n=1}^N L\big(M((A_{\Omega \to X}\mu)(x_n)), y_n\big) + \lambda |\mu|_\U(\Omega).
        \label{eq:sup_opt_problem_mu}
    \end{equation}
    Let $f^\dagger = A_{\Omega \to X}\mu^\dagger$. For any $f = A_{\Omega \to X}\mu$ with $|\mu|_\U \leq \norm{f}_\B + \delta$, where $\delta > 0$ is arbitrary, we have
    \begin{equation}
    \begin{split}
        \frac{1}{N}\sum_{n=1}^N L(M(f^\dagger(x_n)), y_n) + \lambda \norm{f^\dagger}_\B
        &\leq \frac{1}{N}\sum_{n=1}^N L(M((A_{\Omega \to X}\mu^\dagger)(x_n)), y_n) + \lambda |\mu^\dagger|_\U(\Omega) \\
        &\leq \frac{1}{N}\sum_{n=1}^N L(M((A_{\Omega \to X}\mu)(x_n)), y_n) + \lambda |\mu|_\U(\Omega)  \\
        &\leq \frac{1}{N}\sum_{n=1}^N L(M(f(x_n)), y_n) + \lambda \norm{f}_\B + \delta.
    \end{split}
    \end{equation}
    Letting $\delta \to 0$ yields, for any $f\in \B$,
    \begin{equation}
        \frac{1}{N}\sum_{n=1}^N L(M(f^\dagger(x_n)), y_n) + \lambda \norm{f^\dagger}_\B
        \;\leq\; \frac{1}{N}\sum_{n=1}^N L(M(f(x_n)), y_n) + \lambda \norm{f}_\B.
    \end{equation}
    Hence, $f^\dagger$ is a minimizer of \eqref{eq:sup_opt_problem}. Since $f^\dagger = A_{\Omega \to X}\mu^\dagger$ with $\mu^\dagger$ a solution to \eqref{eq:sup_opt_problem_mu}, the theorem will follow once we show the existence of a sparse solution to \eqref{eq:sup_opt_problem_mu}. To this end, we apply Theorem 3.3 of Bredies and Carioni \cite{bredies2020sparsity}.
    
    By the dual space characterization \eqref{eq:duality_C0_and_RadonMeasureSpace},  
    \begin{equation}
        C_0(\Omega; {\U_*})^* \cong \M(\Omega; (\U_*)^*) \cong \M(\Omega; \U),
        \label{eq:duality_C0_predual}
    \end{equation}
    since $\U$ has ${\U_*}$ as a predual and $\Omega$ is locally compact and Hausdorff. Consequently, we can equip $\M(\Omega; \U)$ with the weak-* topology, under which it becomes a real locally-convex topological vector space.
    
    Define
    \begin{equation}
        \mathcal{A} \colon \M(\Omega; \U) \to \R^{Nd}, \quad
        \mu \mapsto \big[M((A_{\Omega \to X}\mu)(x_1)), \ldots, M((A_{\Omega \to X}\mu)(x_N))\big].
    \end{equation}
    A component of $\mathcal{A}(\mu)$ can be written as
    \begin{equation}
    \begin{split}
        M_j((A_{\Omega \to X}\mu)(x_n))
        &= \braket{(A_{\Omega \to X}\mu)(x_n)}{v_j}_{\U_*} \\
        &= \braket{\int_\Omega \phi(x_n, w)d\mu(w)}{v_j}_{\U_*} \\
        &= \int_\Omega \phi(x_n, w) \, d\braket{\mu}{v_j}_{\U_*}(w) \\
        &= \int_\Omega \phi(x_n, w) \, d\braket{v_j}{\mu}_{\U_*,\U}(w) \\
        &= \int_\Omega \braket{\phi(x_n, w) v_j}{d\mu(w)}_{\U_*,\U} \\
        &= \braket{\mu}{\phi(x_n, \cdot) v_j}_{C_0(\Omega; {\U_*})}.
    \end{split}
    \end{equation}
    Here, the last equality follows from the fact \cite{meziani2009dual} that the duality pairing that realizes the characterization in the first isomorphism of \eqref{eq:duality_C0_predual} is the continuous extension of \eqref{eq:braket_integral_defining_equation} by uniform approximation with simple functions. By definition of the weak-* topology, this expression is weak-* continuous in $\mu$. Hence $\mathcal{A}$ is weak-* continuous.  
    
    Let $H = \operatorname{range}(\mathcal{A})$ and define
    \begin{equation}
        F \colon H \to \R,\; [\widetilde{y}_1, \ldots, \widetilde{y}_N] \mapsto \frac{1}{N}\sum_{n=1}^N L(\widetilde{y}_n, y_n).
    \end{equation}
    Since $L(\cdot, y)$ is convex, coercive, and lower semicontinuous, $F$ has the same properties on $H$.  
    
    The regularizer is the total variation norm $|\mu|_\U(\Omega)$. Using again the duality \eqref{eq:duality_C0_predual}, for any $v \in {\U_*}$ with $\norm{v}_{\U_*} \leq 1$,
    \begin{equation}
        \braket{\mu_n}{v}_{C_0(\Omega; {\U_*})}
        \;\leq\; \sup_{\norm{\widetilde{v}}_{C_0(\Omega; {\U_*})} \leq 1} \braket{\mu_n}{\widetilde{v}}_{C_0(\Omega; {\U_*})} = |\mu_n|_{\U}(\Omega).
    \end{equation}
    Taking the liminf, the supremum over $v$, and using \eqref{eq:duality_C0_predual} shows
    \begin{equation}
        |\mu|_\U(\Omega)  = \sup_{\norm{v}_{C_0(\Omega; {\U_*})}\leq 1}\abs{\braket{\mu}{v}_{C_0(\Omega; {\U_*})}} \leq \liminf_n |\mu_n|_\U(\Omega).
    \end{equation}
    Therefore, $|\cdot|_{\U}(\Omega)$ is weak-* lower semicontinuous. Moreover, $\{\mu \in \M(\Omega; \U) : |\mu|_\U(\Omega) = 0\} = \{0\}$. Finally, coercivity follows since the sets $\{\mu \in \M(\Omega; \U) : |\mu|_\U(\Omega) \leq \alpha\}$ are weak-* compact by Banach–Alaoglu.  
    
    Thus, all assumptions of Theorem 3.3 in Bredies and Carioni \cite{bredies2020sparsity} are satisfied. Since the extreme points of the unit ball in $\M(\Omega; \U)$ are (see Theorem 2 of Dirk \cite{werner1984extreme}, or Lemma 3.2 of Bredies et al.\ \cite{bredies2024extremal} whose proof applies without modifications in the general setting considered here) characterized as
    \begin{equation}
        \operatorname{Ext}\big(\{\mu \in \M(\Omega; \U) : |\mu|_\U(\Omega) \leq 1\}\big)
        = \{\delta_w u : w \in \Omega, \, u \in \operatorname{Ext}(\{u \in \U : \norm{u}_\U \leq 1\})\},
    \end{equation}
    there exists a solution that has the form
    \begin{equation}
        \mu^\dagger = \sum_{m=1}^{Nd} a_m \, \delta_{w_m} u_m,
    \end{equation}
    with $u_m \in \operatorname{Ext}(\{u \in \U : \norm{u}_\U \leq 1\})$ and $a_m \in \R$.
\end{proof}

\begin{remark}The setup of Theorem \ref{thm:representer_theorem_integral_RKBS} follows Definition \ref{def:vector_integral_RKBS}, where both $X$ and $\Omega$ are assumed to be locally-compact Hausdorff, and which ensures that the generalized Riesz representation of \eqref{eq:duality_C0_and_RadonMeasureSpace} applies both in $\M(\Omega;\U)$ and in $\M(X;\U^\di)$. However, $\B^\di$ does not play a role in the proof, so the representer theorem also holds without restricting $X$ to be locally compact or imposing that $\phi$ vanishes at infinity in the $x$ variable. Effectively, the only requirements are that the integral in \eqref{eq:vv_integral_RKBS_f} should exist, is finite for any $\mu \in \M(\Omega;\U)$, and that $\Omega$ is locally-compact Hausdorff so that the Riesz theorem applies.
\end{remark}

The representation of $f^\dagger$ closely resembles the standard vv-RKHS representation
\begin{equation}
    \sum_{n=1}^N K(\cdot, x_n) u_n,
\end{equation}
where $K \colon X \times X \to \L(\U)$ is the vv-RKHS kernel (Definition \ref{def:RKHS_kernel}), $\U$ is a Hilbert space, and $u_n \in \U$.  
To make the analogy precise, let $W = \{w_1, \ldots, w_{Nd}\}$, so $|W| \leq Nd$.  
For each unique $\widetilde{w}_k \in W$, set
\begin{equation}
    I_{\widetilde{w}_k} = \set{ m \in \{1, \ldots, Nd\} \given w_m = \widetilde{w}_k }, 
    \qquad 
    \widetilde{u}_k = \sum_{m \in I_{\widetilde{w}_k}} a_m u_m.
\end{equation}
Then the representation from Theorem~\ref{thm:representer_theorem_integral_RKBS} can be rewritten as
\begin{equation}
    f^\dagger = A_{\Omega \to X}\mu^\dagger 
      = \sum_{k=1}^{|W|} \phi(\cdot, \widetilde{w}_k) \widetilde{u}_k
      = \sum_{k=1}^{|W|} K_\U(\cdot, \widetilde{w}_k) \widetilde{u}_k,
\end{equation}
where the last equality uses the kernel definition in Definition \ref{def:vector_integral_RKBS}.  
This mirrors the vv-RKHS formula, with the difference that we only guarantee $|W| \leq Nd$, not $|W| \leq N$.

To understand this limitation, substitute the above $f^\dagger$ into the optimization problem, assuming the optimal $\widetilde{w}_k$ are already fixed:
\begin{equation}
    \min_{\widetilde{u}_k} \; \frac{1}{N}\sum_{n=1}^N 
    L\!\left( \sum_{k=1}^{|W|} \phi(x_n, \widetilde{w}_k)\widetilde{u}_k, \, y_n \right) 
    + \lambda \sum_{k=1}^{|W|} \norm{\widetilde{u}_k}_\U.
\end{equation}

Here, the role of the norm on $\U$ is decisive. If $\widetilde{u}_k \in \R^d$ and $\norm{\cdot}_\U$ behaves like an $\ell_1$ norm, no groupwise sparsity is expected, so the number of active terms $|W|$ need not reduce to $N$. In this case, the vectors $u_i \in \R^d$ correspond to standard basis vectors since they are extremal points of the $\ell_1$ unit ball. Consequently, we obtain at most $N d$ nonzero coefficients, one per component across the $N d$ terms.

In contrast, if $\norm{\cdot}_\U$ behaves more like an $\ell_2$ norm, the formulation resembles a group-lasso penalty, which is known to promote entire vectors to vanish \cite{bach2010structured}. This effect may decrease the number of active terms, potentially yielding $|W| \leq N$. Another important difference compared to the $\ell_1$ case is that the $u_i \in \R^d$ are no longer the standard basis vectors but can have full entries as they lie on the $\ell_2$ unit ball. If groupwise sparsity were possible, we would again obtain only $N d$ nonzero coefficients, analogous to the $\ell_1$ case. However, because the representer theorem in our setting does not permit such groupwise sparsity, the number of coefficients increases to $N d^2$ when using the $\ell_2$ norm.

A detailed investigation into groupwise sparsity is left for future work.

\section{\texorpdfstring{$\R^d$}{Rd}-valued neural networks, hypernetworks, and DeepONets} \label{sec:vv-NNs_HypNets_DeepONets}
This section constructs the spaces for $\R^d$-valued neural networks, hypernetworks, and DeepONets, and derives the corresponding representer theorems by applying Theorem \ref{thm:representer_theorem_integral_RKBS} to each setting.

\subsection{\texorpdfstring{$\R^d$}{Rd}-valued neural network}
For $\R^d$-valued neural networks, the following representer theorem is immediate. 
\begin{corollary}[Representer theorem $\R^d$-valued neural network]
    Assume we are given data $\{(x_n, u_n)\}_{n=1}^N \subseteq X \times \R^d$ .  
    Let $L \colon \R^d \times \R^d \to \R$ be such that for every $u \in \R^d$, the map $L(\cdot, u)$ is convex, coercive, and lower semicontinuous.   
    Let $\B$ be a neural vv-RKBS of functions from compact $X \subseteq \R^{d_x}$ to $\R^{d}$ with weight space $\Omega \subseteq \R^{d_x} \times \R$ and 
    \begin{equation}
        \phi \in C_0(X \times \Omega),\quad \phi(x,(\omega,b)) \coloneqq \sigma(\braket{\omega}{x}_{\ell_2} + b) \beta((\omega,b)),
    \end{equation}
    for a measurable activation $\sigma \colon \mathbb R \to \mathbb R$ and a measurable positive $\beta \colon \Omega \to \mathbb R$.
    
    Then there exists a solution to the supervised learning problem
    \begin{equation}
        \min_{f \in \B} \; \frac{1}{N} \sum_{n=1}^N L(f(x_n), u_n) + \lambda \norm{f}_\B, 
        \qquad \lambda \geq 0,
    \end{equation}
    that admits a neural network representation:
    \begin{equation}
        f^\dagger(x) = A_{\Omega \to X} \mu^\dagger(x)
          = \sum_{m=1}^{Nd} \sigma(\braket{\omega_m}{x}_{\ell_2} + b_m) u_m = U \sigma(Wx + B),
    \end{equation}
    with $U \in \R^{d_u \times Nd}$ containing the $u_m$ as columns, $W \in \R^{Nd \times d_x}$ the $\omega_m$ as rows, and $B \in \R^{Nd}$ the $b_m$ values. 
\end{corollary}
\begin{proof}
    By Theorem \ref{thm:representer_theorem_integral_RKBS}, we get:
    \begin{equation}
          f^\dagger = A_{\Omega \to X} \mu^\dagger 
          = \sum_{m=1}^{Nd} a_m  \sigma(\braket{\omega_m}{\cdot}_{\ell_2} + b_m) \beta((\omega_m, b_m)) u_m,
    \end{equation}
    Redefining $u_m$ to $a_m \beta((\omega_m, b_m)) u_m$ yields the result.
\end{proof}

\subsection{DeepONets and hypernetworks} \label{sec:deeponets_and_hypernets}
As noted in the introduction, neural operators and conditional implicit neural representations (INRs) are closely related. Many neural operator methods, including DeepONet, FNO, and CNO, can be viewed as special cases of INRs. In this section, we investigate the vv-RKBS structure of neural operators via this INR perspective, particularly focusing on DeepONet and hypernetworks. 

We begin by introducing two function spaces that cover both hypernetworks and DeepONets. Subsequently, we formulate an optimization problem for each function space and prove a joint representer theorem showing that the two problems admit the same sparse solution. The section concludes with a discussion of the function space formulations.

\subsubsection{Function space covering the DeepONet and hypernetwork}

The DeepONet framework is designed to learn mappings from inputs $z \in Z$ to functions $u \colon X \to \V$ for some Banach space $\V$. It achieves this by representing the output as a linear combination of learned basis functions:
\begin{equation}
f(z)(x) = \sum_{n=1}^{n_b} a_n(z) \zeta_n(x),
\label{eq:deeponet_formula}
\end{equation}
where both the coefficients $a_n \colon Z \to \R$ and the basis functions $\zeta_n \colon X \to \V$ 
 are parameterized by neural networks. The basis functions can be obtained through standard training methods such as gradient descent \cite{lu2021learning} or constructed using techniques like proper orthogonal decomposition (POD) \cite{lu2022comprehensive}.

The DeepONet learns a linear subspace of functions. Hypernetworks provide a more nonlinear parameterization by mapping an input to the weights of another base network. Equivalently, this can be seen as mapping an input to an entire neural network.

To construct a function space covering both DeepONets and hypernetworks, we note that DeepONets define linear subspaces, whereas hypernetworks yield nonlinear parameterizations. This suggests that DeepONets may arise as elements of a hypernetwork function space. We confirm this intuition and introduce two hypernetwork function spaces: one interpreting the hypernetwork as a mapping to weights and the other as a mapping to entire networks. 

\begin{definition}[Hypernetworks in weight form]
    Let $Z$, $X$, $\Omega$, and $\Theta$ be locally-compact Hausdorff and $(\V, \V^\di)$ a continuous dual pair of Banach spaces. Define
    \begin{subequations}
    \begin{align}
        \B_h &= \set*{f: Z\to \M(\Theta;\V) \given f(z) = \int_\Omega \phi(z,w)d\mu(w), \mu\in \M(\Omega;\M(\Theta;\V))} \\
        \norm{f}_{\B_h} &= \inf\set[\bigg]{ |\mu|_{\M(\Theta;\V)}(\Omega) \given \mu\in \M(\Omega;\M(\Theta;\V))\text{ satisfies }f(z) = \int_\Omega \phi(z,w)d\mu(w)} \\
        \B^\di_h &= \overline{\B^\di_{h,\text{pre}}}= \overline{\set*{g: \Omega\to \M(X;\V^\di) \given g(w) = \int_Z \phi(z,w)d\pi(z), \pi\in \M(Z;\M(X;\V^\di))}} \\
        \norm{g}_{\B^\di_h} &= \sup_{w \in \Omega} |g(w)
        |_{\V^\di}(X) \\
        \braket{g}{f}_{\B_h} &= \braket{\mu}{g}_{C_0(\Omega; \M(X;\V^\di))} \stackrel{g\in\B^\di_{h, \text{pre}}}=\int_{Z \times \Omega} \phi(z,w) \,d\!\braket{\pi}{\mu}_\M\!(z,w), 
    \end{align}    
    \end{subequations}
    where $\braket{\pi}{\mu}_\M$ is the Hahn-Kolmogorov extension of the scalar-valued measure satisfying 
    \begin{equation}
        \braket{\pi}{\mu}_\M(E \times F) = \braket{\pi(E)}{\mu(F)}_{\M(X; \V^\di),\M(\Theta; \V)}
    \end{equation}
    with
    \begin{equation}
        \braket{\rho}{\nu}_{\M(X; \V^\di),\M(\Theta; \V)} \coloneqq \int_{X \times \Theta} d \braket{\rho}{\nu}_\V(x, \theta)
    \end{equation}
    and $\braket{\rho}{\nu}_\V$ the Hahn-Kolmogorov extension of the scalar measure satisfying $\braket{\rho}{\nu}_\V(E\times F)= \braket{\rho(E)}{\nu(F)}_\V$ under the duality $(\V, \V^\di)$. The pair $(\B_h, \B_h^\di)$ is called an integral vv-RKBS pair of hypernetworks \textit{in weight space form}. 
    \label{def:hypernetwork_RKBS_weight_form}
\end{definition}
\begin{definition}[Hypernetworks in function form]
    Let $Z$, $X$, $\Omega$, and $\Theta$ be locally-compact Hausdorff and $(\V, \V^\di)$ a continuous dual pair of Banach spaces.
    Additionally, let $(\B, \B^\di)$ be an integral vv-RKBS pair with $\B \subset \{ u\colon X \to \V\}$, $\B^\di \subset \{ v\colon \Theta \to \V^\di\}$, and kernel $\psi \colon X \times \Theta \to \R$. Define
    \begin{subequations}
    \begin{align}
        \B_h &= \set*{f: Z\to \B \given f(z) = \int_\Omega \phi(z,w)d\mu(w), \mu\in \M(\Omega;\B)} \\
        \norm{f}_{\B_h} &= \inf\set[\bigg]{ |\mu|_{\B}(\Omega) \given \mu\in \M(\Omega;\B)\text{ satisfies }f(z) = \int_\Omega \phi(z,w)d\mu(w)} \\
        \B^\di_h &= \overline{\B^\di_{h,\text{pre}}} = \overline{\set*{g: \Omega\to \B^\di \given g(w) = \int_Z \phi(z,w)d\pi(z), \pi\in \M(Z;\B^\di)}} \\
        \norm{g}_{\B^\di_h} &= \sup_{w \in \Omega} \norm{g(w)}_{\B^\di} \\
        \braket{g}{f}_{\B_h} &= \braket{\mu}{g}_{C_0(\Omega; \B^\di)} \stackrel{g\in\B^\di_{h,\text{pre}}} = \int_{Z \times \Omega} \phi(z,w) \,d\!\braket{\pi}{\mu}_\B \!(z,w)
    \end{align}    
    \end{subequations}
    where $\braket{\pi}{\mu}_\B$ is the Hahn-Kolmogorov extension of the scalar measure satisfying $\braket{\pi}{\mu}_\B(E \times F) = \braket{\pi(E)}{\mu(F)}_\B$, the latter pairing being as in Definition \ref{def:adjoint_pair_vector_RKBS}. The pair $(\B_h, \B_h^\di)$ is called an integral vv-RKBS pair of hypernetworks \textit{in function space form}.
    \label{def:hypernetwork_RKBS_func_form}
\end{definition}

\begin{remark}
In the weight-space viewpoint, it seems natural to let $\M(\Theta;\V)$ and $\M(X;\V^\di)$ have a pairing that reflects how the output weights enter the base neural network's structure:
\begin{equation}
    \braket{\rho}{\mu}_{\M(X,\V^\di),\,\M(\Theta,\V)} 
    = \int_{X\times\Theta}\psi(x,\theta)\,d\braket{\rho}{\mu}_\V
    = \braket{g}{f}_{\B},
\end{equation}
where $\braket{\rho}{\mu}_\V(A\times B)=\braket{\rho(A)}{\mu(B)}_\V$. The last equality follows from Definition \ref{def:vector_integral_RKBS}, where $(\B, \B^\di)$ is the integral vv-RKBS pair in the hypernetwork's function-space viewpoint. However, if the nullspace of $A_{\Theta\to X}$ is nontrivial, then $\braket{\rho}{\mu}_{\M(X,\V^\di),\,\M(\Theta,\V)}=\braket{g}{f}_\B=\braket{g}{0}_\B=0$ may hold for nonzero $\rho,\mu$. Hence, this pairing is dual only when the nullspace of $A_{\Theta\to X}$ is trivial. So the duality assumption in Definition~\ref{def:vector_integral_RKBS} holds only in that case. In contrast, the function-space viewpoint avoids this issue, since it pairs $f$ and $g$ directly rather than their representing measures.
\end{remark}

\begin{remark}
    The function-space viewpoint of the integral vv-RKBS pair of hypernetworks naturally contains DeepONets as a special case. Let $(\B, \B^\di)$ be a neural vv-RKBS, and take $\{\zeta_n\}_{n=1}^{n_b} \subset \B$ as basis functions, which are neural networks by assumption. Define the $\B$-valued measure  
    \begin{equation}
        \mu = \sum_{n=1}^{n_b} \left(\sum_{k=1}^{h_{nn}} a_{nk} \delta_{w_{nk}} \right) \zeta_n
    \end{equation} 
    Then  
    \begin{equation}
        f(z)(x) = (A_{\Omega \to Z}\mu)(z)(x) = \sum_{n=1}^{n_b} \left(\sum_{k=1}^{h_{nn}} a_{nk} \phi(z, w_{nk}) \right) \zeta_n(x) \eqqcolon \sum_{n=1}^{n_b} a_n(z) \zeta_n(x)
    \end{equation}
    which recovers the DeepONet formulation in \eqref{eq:deeponet_formula}.
\end{remark}
While the weight and function space viewpoints are similar, they differ in key aspects. The weight-space viewpoint first maps to measures and then integrates them to obtain a function. The function-space viewpoint does the opposite and first maps to functions (integrating out measures) and then integrates a measure of these functions. A more significant difference lies in the norms. The weight-space viewpoint uses the total variation of $\M(\Theta; \U)$ within the total variation of $\M(\Omega; \M(\Theta; \U))$, while the function-space viewpoint uses $\norm{\cdot}_\B$ inside the total variation of $\M(\Omega; \B)$. These differences affect the resulting optimization problems. 

\subsubsection{Shared representer theorem}
Since the two function spaces describe the same structure, we expect the corresponding optimization problems to have related solutions and share a representer theorem. To define these optimization problems, let $A_{\Omega \to Z}$ be as in Equations \eqref{eq:vv_integral_RKBS_f}--\eqref{eq:vv_integral_RKBS_g}, with $\phi$ as integrand. Since this operation is valid for any measure on $\Omega$, the same notation applies whether we consider $\M(\Omega; \M(\Theta; \V))$ or $\M(\Omega; \B)$. Similarly, independent of the target space of $\nu$, we define  
\begin{equation}\label{eq:defAThetaX}
    A_{\Theta \to X} \nu \coloneqq \int_\Theta \psi(\cdot, \theta) d\nu(\theta)
\end{equation}
with $\psi \colon X \times \Theta \to \R$ as in Definition \ref{def:hypernetwork_RKBS_func_form}.

We also define a measurement operator $M \colon \B \to \R^d$, where $\B$ is the output RKBS in the hypernetwork function-space formulation. In the weight-space formulation, we apply $M$ to $f = A_{\Theta \to X}(u(z)) \in \B$ with $u(z) \in \M(\Theta; \V)$. The components of $M$ are taken as  
\begin{equation}
    M_{n} f = \braket{v_n^\di}{f(x_n)}_\V.
    \label{eq:measurement_operator_hypernetwork_RKBS}
\end{equation}
If $\V = \R^d$, one may consider
\begin{equation}
    M_{nj} f = \braket{e_j}{f(x_n)}_{\ell_2},
\end{equation}
which produces measurements for each component and is simply a reindexed version of \eqref{eq:measurement_operator_hypernetwork_RKBS}.  

With these definitions, the optimization problems can be written directly in terms of measures rather than functions $f \in \B_h$. By the first step in the proof of Theorem \ref{thm:representer_theorem_integral_RKBS}, this suffices since minimizing over measures yields the same minimizer as minimizing over $f$.

In the weight-space viewpoint, using $|\mu|_\M$ as short-hand notation for $|\mu|_{\M(\Theta;\V)}$, the problem is
\begin{equation}
    \min_{\mu \in \M(\Omega; \M(\Theta;\V))} \; \frac{1}{N} \sum_{n=1}^N L(M(A_{\Theta \to X}( (A_{\Omega \to Z}\mu)(z_n))), y_n) + \lambda |\mu|_{\M}(\Omega)
    \label{eq:opt_prob_weight_viewpoint_hypnet}
\end{equation}
while in the function-space viewpoint, it is
\begin{equation}
    \min_{\mu \in \M(\Omega; \B)} \; \frac{1}{N} \sum_{n=1}^N L(M( (A_{\Omega \to Z}\mu)(z_n))), y_n) + \lambda |\mu|_{\B}(\Omega)
    \label{eq:opt_prob_function_viewpoint_hypnet}
\end{equation}
The key differences are that $A_{\Theta \to X}$ appears explicitly in the weight-space formulation but is absorbed into $\mu$ in the function-space formulation, and the total variation regularization uses different inner norms. Despite these differences, we can show that a solution to the optimization problem corresponding to the weight space viewpoint (see \eqref{eq:opt_prob_weight_viewpoint_hypnet}) provides a solution to the optimization problem of the function space viewpoint (see \eqref{eq:opt_prob_function_viewpoint_hypnet}).
\begin{theorem}
    Assume the optimization problem corresponding to the weight space viewpoint (see \eqref{eq:opt_prob_weight_viewpoint_hypnet}) has a minimizer $\mu_\M^\dagger \in \M(\Omega; \M(\Theta; \V))$. Then the measure $\mu_\B^\dagger \in \M(\Omega;\B)$ defined via $\mu_\B^\dagger(E) \coloneqq (A_{\Theta \to X} (\mu_\M^\dagger(E)))$ is a solution to the optimization problem of the function space viewpoint (see \eqref{eq:opt_prob_function_viewpoint_hypnet}).
\end{theorem}
\begin{proof}
    Pick any $\varepsilon > 0$ and $\mu_\B \in \M(\Omega; \B)$. By Lemma \ref{lemma:C0_domain_decomposition}, there exists a decomposition of $Z$ into finitely-many pairwise-disjoint sets $A_1, \ldots, A_\ell \subseteq Z$ and a decomposition of $\Omega$ into finitely-many pairwise-disjoint sets $B_1, \ldots, B_I \subseteq \Omega$ such that 
    \begin{equation}
        |\phi(z,w) - \phi(z, \widetilde{w})| \leq \varepsilon \quad \text{for all } z \in Z,\text{all } w, \widetilde{w} \in B_i \text{ and all }i\in \set{1,\hdots, I}.
    \end{equation}
    For the same $\varepsilon$, by the definition of the total variation, there exists a partition $\widetilde{B}_1, \ldots \widetilde{B}_{J} \subseteq \Omega$ such that
    \begin{equation}
        |\mu_\B|_\B(\Omega) \leq  \varepsilon + \sum_{j=1}^J \norm{\mu_\B(\widetilde{B}_j)}_\B 
    \end{equation}
    Define $\overline{B}_{ji} \coloneqq \widetilde{B}_j \cap B_i$. Using the finite additivity of vector measures, we obtain
    \begin{equation}
        |\mu_\B|_\B(\Omega)  \leq \varepsilon + \sum_{j=1}^J \norm{\sum_{i=1}^I \mu_\B(\overline{B}_{ji})}_\B \leq \varepsilon + \sum_{j=1}^J \sum_{i=1}^I \norm{ \mu_\B(\overline{B}_{ji})}_\B = \varepsilon + \sum_{j=1}^J \sum_{i=1}^I \inf_{A_{\Theta \to X}\nu = \mu_\B(\overline{B}_{ji})} |\nu|_\V(\Theta)
    \end{equation}
    For an arbitrary $\delta > 0$, let $\nu_{ji} \in \M(\Theta;\V)$ be chosen such that
    \begin{equation}
        |\nu_{ji}|_\V(\Theta) \leq \left(\inf_{A_{\Theta \to X}\nu = \mu_\B(\overline{B}_{ji})} |\nu|_\V(\Theta)\right) + \frac{\delta}{IJ} = \norm{\mu_\B(\overline{B}_{ji})}_\B + \frac{\delta}{IJ}.
    \end{equation}
    Fix $\delta = \varepsilon = \frac{1}{2k}$ and define
    \begin{equation}
        \mu_{\M}^{(k)} \coloneqq \sum_{j=1}^J \sum_{i=1}^I \delta_{w_{ji}} \nu_{ji}
    \end{equation}
    for arbitrary $w_{ji} \in \overline{B}_{ji}$. Then
    \begin{equation}
    \begin{split}
        |\mu_\B|_\B(\Omega) &\leq \frac{1}{2k} + | \mu_{\M}^{(k)} |_{\M}(\Omega) \\&= \frac{1}{2k} + \left(\sum_{j=1}^J \sum_{i=1}^I |\nu_{ji}|_\V(\Theta)\right)  \\ &\leq \frac{1}{2k} + \left(\sum_{j=1}^J \sum_{i=1}^I \norm{\mu_\B(\overline{B}_{ji})}_\B\right) + \frac{1}{2k} \\ &\leq \frac{1}{k} + |\mu_\B|_\B(\Omega),
    \end{split}
    \end{equation}
    which shows that $| \mu_{\M}^{(k)} |_{\M}(\Omega) \to |\mu_\B|_\B(\Omega)$ as $k \to \infty$. 

    Define $\mu_\B^{(k)}$ via the bounded linear operator $T$
    \begin{equation}
        \mu_\B^{(k)}(B) \coloneqq T(\mu_\M^{(k)})(B) \coloneqq A_{\Theta \to X}(\mu_{\M}^{(k)}(B)) = \sum_{j=1}^J \sum_{i=1}^I \delta_{w_{ji}}(B) A_{\Theta \to X}(\nu_{ji}) = \sum_{j=1}^J \sum_{i=1}^I \delta_{w_{ji}}(B) \mu_\B(\overline{B}_{ji})
    \end{equation}
    The operator is bounded because 
    \begin{equation}
        \left|T\mu_\M\right|_\B (\Omega) = \sup_{\set{\hat{B}_\bullet}} \sum_{h=1}^{|\set{\hat{B}_\bullet}|}\norm{A_{\Omega \to X} (\mu_\M(\hat{B}_h))}_\B \leq \sup_{\set{\hat{B}_\bullet}} \sum_{h=1}^{|\set{\hat{B}_\bullet}|} \norm{\mu_\M(\hat{B}_h)}_{\M(\Theta; \V)} = \left|\mu_\M\right|_\M(\Omega)
        \label{ineq:B_tot_var_against_M_tot_var}
    \end{equation}
    with the suprema going over all the partitions $\set{\hat{B}_\bullet}$ of $\Omega$ and $\mu_\M \in \M(\Omega; \M(\Theta; \V))$ is arbitrary.
    
    We can estimate
    \begin{equation}
    \begin{split}
        \norm{\int_\Omega \phi(z, w) d \mu_\B - \int_\Omega \phi(z,w) d \mu_\B^{(k)}}_\B & \leq \sum_{j=1}^J \sum_{i=1}^I \norm{\int_{\overline{B}_{ji}} \phi(z,w) d \mu_\B - \int_{\overline{B}_{ji}} \phi(z,w) d \mu_\B^{(k)}}_\B \\
        & = \sum_{j=1}^J \sum_{i=1}^I \norm{\int_{\overline{B}_{ji}} \phi(z,w) d \mu_\B - \int_{\overline{B}_{ji}} \phi(z,w_{ji}) d \mu_\B}_\B \\
        & \leq \sum_{j=1}^J \sum_{i=1}^I \left(\sup_{w, \widetilde{w} \in \overline{B}_{ji}} |\phi(z, w) - \phi(z, \widetilde{w})| \right) |\mu_\B|_\B(\overline{B}_{ji}) \\&\leq \frac{1}{2k} |\mu_\B|_\B(\Omega)
    \end{split}
    \end{equation}
    This shows that $A_{\Omega \to Z} \mu_\B^{(k)} \to A_{\Omega \to Z}\mu_\B$ uniformly in $z$. Since $M$ is a bounded linear operator on $\B$, we also have
    \begin{equation}
        M((A_{\Omega \to Z} \mu_\B^{(k)})(z)) \to M((A_{\Omega \to Z} \mu_\B)(z)) \quad \text{uniformly in } z.
    \end{equation}
    By continuity of $L$ in the first argument and convergence of the regularization terms, it follows that
    \begin{equation}
        \frac{1}{N}\sum_{n=1}^N L(M((A_{\Omega \to Z} \mu_\B)(z_n)), y_n) + \lambda |\mu_\B|_\B(\Omega) = \lim_{k \to \infty}  \frac{1}{N}\sum_{n=1}^N L(M((A_{\Omega \to Z} \left(T\mu_\M^{(k)}\right))(z_n)), y_n) + \lambda |\mu_\M^{(k)}|_\M(\Omega).
        \label{eq:limiting_equality}
    \end{equation}
    Moreover, by Theorem \ref{thm:simple_approximation_of_C0_functions}, there exists a sequence of simple functions
    \begin{equation}
        \phi_k = \sum_{i,j} a^{(k)}_{ij} \mathbbm{1}_{A_i^{(k)} \times B_j^{(k)}} \to \phi
    \end{equation}
    uniformly, with pairwise disjoint $A_i^{(k)}$'s and $B_j^{(k)}$'s covering $Z$ and $\Omega$, respectively. Then
    \begin{equation}
    \begin{split}
        A_{\Omega \to Z}(T \mu_\M)(z) 
        & = \int_\Omega \phi(z,w) d \left(T \mu_\M\right)(w) \\ 
        &= \lim_{k \to \infty} \int_\Omega \sum_{i,j} a^{(k)}_{ij} \mathbbm{1}_{A_i^{(k)} \times B_j^{(k)}}(z,w) d \left(T \mu_\M\right)(w) \\
        & = \lim_{k \to \infty} \sum_{i,j} a^{(k)}_{ij} \mathbbm{1}_{A_i^{(k)}}(z) A_{\Theta \to X}(\mu_\M(B_j^{(k)})) \\ 
        &= \lim_{k \to \infty}A_{\Theta \to X}\left(  \sum_{i,j} a^{(k)}_{ij} \mathbbm{1}_{A_i^{(k)}}(z)\mu_\M(B_j^{(k)})\right) \\
        & = \lim_{k \to \infty}A_{\Theta \to X}\left( \int_\Omega \sum_{i,j} a^{(k)}_{ij} \mathbbm{1}_{A_i^{(k)} \times B_j^{(k)}}(z,w) d\mu_\M(w)\right) \\
        & = A_{\Theta \to X}\left( \int_\Omega \phi(z,w) d\mu_\M(w)\right) \\
        &= A_{\Theta \to X}((A_{\Omega \to Z} \mu_\M)(z))
    \end{split}
    \end{equation}
    where the second equality follows from $\phi_k \to \phi$ uniformly and the second-to-last equality uses the continuity of $A_{\Theta \to X}$. 
    
    Combining this with \eqref{eq:limiting_equality} yields    
    \begin{equation}
    \begin{split}
        &\frac{1}{N}\left(\sum_{n=1}^N L(M((A_{\Omega \to Z} \mu_\B)(z_n)), y_n) \right) + \lambda |\mu_\B|_\B(\Omega) \\
        &\qquad = \lim_{k \to \infty}  \frac{1}{N}\left(\sum_{n=1}^N L(M((A_{\Omega \to Z} \left(T\mu_\M^{(k)}\right))(z_n)), y_n) \right) + \lambda |\mu_\M^{(k)}|_\M(\Omega) \\
        &\qquad = \lim_{k \to \infty}  \frac{1}{N}\left(\sum_{n=1}^N L(M(A_{\Theta \to X} \left(\left(A_{\Omega \to Z}\mu_\M^{(k)}\right)(z_n)\right)), y_n) \right) + \lambda |\mu_\M^{(k)}|_\M(\Omega) \\
        &\qquad \geq \frac{1}{N}\left(\sum_{n=1}^N L(M(A_{\Theta \to X} \left(\left(A_{\Omega \to Z}\mu_\M^\dagger\right)(z_n)\right)), y_n) \right) + \lambda |\mu_\M^\dagger|_\M(\Omega) \\
        &\qquad = \frac{1}{N}\left(\sum_{n=1}^N L(M\left(\left(A_{\Omega \to Z} \left(T\mu_\M^\dagger)\right)\right)(z_n)\right), y_n) \right) + \lambda |\mu_\M^\dagger|_\M(\Omega) \\
        &\qquad \geq \frac{1}{N}\left(\sum_{n=1}^N L(M((A_{\Omega \to Z}\mu_\B^\dagger)(z_n)), y_n) \right) + \lambda |\mu_\B^\dagger|_\B(\Omega)
    \end{split}
    \end{equation}
    where the first inequality follows from the optimality of $\mu_\M^\dagger$, and the second from the identity $\mu_\B^\dagger = T\mu_\M^\dagger$ together with inequality \eqref{ineq:B_tot_var_against_M_tot_var}. Hence, the above chain of inequalities establishes the optimality of $\mu_\B^\dagger$.
\end{proof}
Building on the previous theorem, we can now state a joint representer theorem for both formulations. In particular, the theorem shows that the optimization problems corresponding to the weight space and function space viewpoints admit a coinciding sparse solution.
\begin{theorem}[Representer theorem hypernetwork]\label{thm:representer_hypernetwork}
    Consider either optimization problem \eqref{eq:opt_prob_weight_viewpoint_hypnet} or \eqref{eq:opt_prob_function_viewpoint_hypnet}, with the measurement operator $M \colon \B \to \R^d$ defined in \eqref{eq:measurement_operator_hypernetwork_RKBS} with $\{(v_j^\di, x_j)\}_{j=1}^d \subseteq \V^\di \times X$. Assume we are given data $\{(z_n, y_n)\}_{n=1}^N \subseteq Z \times \R^d$. Let $L \colon \R^d \times \R^d \to \R$ be such that for every $y \in \R^d$, the map $L(\cdot, y)$ is convex, coercive, and lower semicontinuous. Finally, let $\V$ have a predual $\V^\di$, and let $\Omega$ and $\Theta$ be locally-compact Hausdorff spaces. Then there exist solutions that have the form
    \begin{equation}
        \mu_\M^\dagger = \sum_{m=1}^{Nd} \delta_{w_m} \delta_{\theta_m} v_m, \quad \mu_\B^\dagger = \sum_{m=1}^{Nd} \delta_{w_m} A_{\Theta \to X}\left(\delta_{\theta_m}v_m\right)
    \end{equation}
    with $v_m \in \V$ and scalar coefficients absorbed into $v_m$. Both measures yield the same function
    \begin{equation}
        f^\dagger(z) = \sum_{m=1}^{Nd} \phi(z, w_m) \psi(x, \theta_m) v_m
    \end{equation}    
\end{theorem}
\begin{proof}
Define $\widetilde{M} \colon \M(\Theta; \V) \to \R^d$ by $\widetilde{M} \coloneqq M\circ A_{\Theta \to X}$. With this definition, optimization problem \eqref{eq:opt_prob_weight_viewpoint_hypnet} turns into
\begin{equation}
    \min_{\mu \in \M(\Omega; \M(\Theta;\V))} \; \frac{1}{N} \sum_{n=1}^N L(\widetilde{M}( (A_{\Omega \to Z}\mu)(z_n))), y_n) + \lambda |\mu|_{\M(\Theta;\V)}(\Omega)
\end{equation}
Since $(\V^{\di})^* = \V$ and $\Theta$ is a locally-compact Hausdorff space, we have $C_0(\Theta; \V^{\di})^* \cong \M(\Theta; \V)$ by \eqref{eq:duality_C0_and_RadonMeasureSpace}. Moreover, by construction of $M$, 
\begin{equation}
    (\widetilde{M}\mu)_j = \braket{v_j^\di}{(A_{\Theta \to X}\mu(x_j)}_\V = \braket{\mu}{\psi(x_j, \cdot) v_j^\di}_{C_0(\Theta; \V^{\di})}
\end{equation}
where the final equality follows from the definition of $\braket{\cdot}{\cdot}_{C_0(\Theta; {\V^\di})}$ as the continuous extension of \eqref{eq:braket_integral_defining_equation} by uniform approximation with simple functions.

This shows that the measurement operator is of the form required for Theorem \ref{thm:representer_theorem_integral_RKBS}. The desired representation of the solutions, therefore, follows directly from it.
\end{proof}
\begin{remark}
    Note that the structure of the representer theorem resembles the DeepONet in \eqref{eq:deeponet_formula}. In particular, assume $Z \subseteq \R^{d_z}$ and $X \subseteq \R^{d_x}$ compact, $\Omega \subseteq \R^{d_z} \times \R$ and $\Theta \subseteq \R^{d_x} \times \R$ not necessarily compact. As in the neural vv-RKBS setting introduced in Definition \ref{def:vector_integral_RKBS}, we define
    \begin{equation}
        \phi \in C_0(Z \times \Omega),\quad \phi(z,(\omega,b)) \coloneqq \sigma_\phi(\braket{\omega}{z}_{\ell_2} + b) \beta_\phi((\omega,b)),
    \end{equation}
    and
    \begin{equation}
    \psi \in C_0(X \times \Theta),\quad \psi(x,(\theta,c)) \coloneqq \sigma_\psi(\braket{\theta}{x}_{\ell_2} + c) \beta_\psi((\theta,c)),
    \end{equation}
    for measurable activation functions $\sigma_\phi, \sigma_\psi \colon \mathbb R \to \mathbb R$ and measurable positive functions $\beta_\phi, \beta_\psi \colon \Omega \to \mathbb R$.
    
    In this setting, the hypernetwork provided by the representer theorem takes the form
    \begin{equation}
        \begin{split}
        f^\dagger(z)(x) & = \sum_{m=1}^{Nd} \sigma_\phi(\braket{\omega_m}{z}_{\ell_2} + b_m) \beta_\phi((\omega_m,b_m)) \sigma_\psi(\braket{\theta_m}{x}_{\ell_2} + c_m) \beta_\psi((\theta_m,c_m)) v_m \\
        & = \sum_{m=1}^{Nd} \sigma_\phi(\braket{\omega_m}{z}_{\ell_2} + b_m) \sigma_\psi(\braket{\theta_m}{x}_{\ell_2} + c_m) v_m
        \end{split}
    \end{equation}
    where the $\beta$ terms are absorbed in the $v_m$'s. 
    
    Here, the $\sigma_\phi(\braket{\omega_m}{z}_{\ell_2} + b_m)$ parts determine the coefficients and the $\sigma_\psi(\braket{\theta_m}{x}_{\ell_2} + c_m) v_m$ parts serve as the basis functions in the DeepONet formulation.
\end{remark}

\subsubsection{Discussion on function space formulation}
While we have adopted Definitions \ref{def:hypernetwork_RKBS_weight_form} and \ref{def:hypernetwork_RKBS_func_form} as the function spaces for both DeepONets and hypernetworks, we could also consider alternative formulations of the same spaces or consider completely different function spaces.

For an example of the former, assume $\V = \R$ for simplicity. In the situation of Definitions \ref{def:hypernetwork_RKBS_weight_form} and \ref{def:hypernetwork_RKBS_func_form} with an integral RKBS pair $(\B,\B^\di)$ with kernel $\psi:X\times \Theta \to \R$, consider an integral RKBS setup with measures $\xi \in \M(\Omega \times \Theta)$ representing functions $h:Z \times X \mapsto \R$ by
\begin{equation}\label{eq:product-repr}h(z,x)=\int_{\Omega \times \Theta} \phi(z,w)\psi(x,\theta) d\xi(w,\theta).\end{equation}
By disintegration (see Theorem 10.4.8 in Bogachev \cite{bogachev2007measure2} for an adequate version for signed measures) through the projection onto the first component in $\Omega \times \Theta$, there is a bijection between $\xi \in \M(\Omega \times \Theta)$ and $\mu \in \M(\Omega; \M(\Theta))$. This makes both representations coincide, that is $[f(z)](x)=[A_{\Theta \to X}( (A_{\Omega \to Z}\mu)(z))](x)=h(z,x)$ with $A_{\Theta \to X}$ defined as in \eqref{eq:vv_integral_RKBS_f} and $A_{\Omega \to Z}$ as defined in \eqref{eq:defAThetaX}. However, the constructions of Definitions \ref{def:hypernetwork_RKBS_weight_form} and \ref{def:hypernetwork_RKBS_func_form} lead to a clear distinction between the data points $z_i \in Z$ and observation points $x_j \in X$ for the measurement operator $M$ in the formulation of the optimization problems and the resulting Theorem \ref{thm:representer_hypernetwork}. We argue that this advantage makes them preferable to working directly with representations in the product space of parameters $\Omega \times \Theta$.

Such a correspondence also has parallels with the RKHS conditional mean embedding as introduced in a measure-theoretic framework in Definition 3.1 of Park and Muandet \cite{park2020measure}. Specifically, consider the underlying topological space $Z \times X$, a Borel probability measure $\xi \in \P(Z \times X)$, the projections $\Pi_Z:Z \times X \to Z$ and $\Pi_X:Z \times X \to X$, and $\H_Z,\H_X$ two RKHS with kernels $K_Z:Z \times Z \to \R$ and $K_X:X \times X \to \R$. Then, on the one hand we have the unconditional mean embedding
\begin{equation}\label{eq:mean-embedding}\mathbb{E}_{\xi}\big[K_Z(\Pi_Z, \cdot) K_X(\Pi_X, \cdot)\big],\end{equation}
which is analogous to the representation \eqref{eq:product-repr} and produces an element of the RKHS over $Z \times X$ with kernel $((z_1,x_1),(z_2,x_2)) \mapsto K_Z(z_1,z_2)K_X(x_1,x_2)$. Note that this kernel is positive semi-definite and its associated RKHS is linearly isometric to the tensor product $\H_Z \otimes \H_X$, see Theorem 5.11 of Paulsen and Raghupathi \cite{paulsen2016introduction}. On the other hand, conditioning with respect to the random variable $\Pi_Z$ gives rise to the conditional mean embedding
\begin{equation}\label{eq:CME}\mathbb{E}_{\xi}\big[K_X(\Pi_X, \cdot) \,\big|\, \Pi_Z\big].\end{equation}
This conditional expectation is a random variable $Z \to \H_X$, similarly to the function-space view of Definition \ref{def:hypernetwork_RKBS_func_form} in which we consider maps $Z \to \B$. However, unlike for \eqref{eq:mean-embedding}, in \eqref{eq:CME} there is no kernel involving the variable $z \in Z$, since this quantity doesn't measure a relation with other points in $Z$ and it also does not take into account the effect of parameters $w \in \Omega$.

While the previous discussion focused on reformulating Definitions \ref{def:hypernetwork_RKBS_weight_form} and \ref{def:hypernetwork_RKBS_func_form} in an equivalent way and relating this alternative to mean embeddings, one could also consider entirely different function spaces. For instance, assume $(\B, \B^\di)$ is an integral vv-RKBS pair with $\B \subset \{ u\colon X \to \V\}$, $\B^\di \subset \{ v\colon \Theta \to \V^\di\}$
and kernel $\psi \colon X \times \Theta \to \R$. For a parametrized family of Borel maps $F_z: \Omega \to \Theta$ indexed by $z \in Z$, one can define a pushforward-induced space as
\begin{equation}
    \B_F := \left\{ r \colon Z \to \B \,\middle\vert\, \big[r(z)\big](\cdot) = \int_\Theta \psi(\cdot, \theta) d \big[(F_z)_\#\nu\big](\theta)= \int_\Omega \psi(\cdot, F_z(w)) d \nu(w) \text{ for some }\nu \in \M(\Omega; \V)\right\},
\end{equation}
and endow it with the norm
\begin{equation}
    \|r\|_{\B_F} := \inf \left\{ |\nu|_{\V}(\Omega) \,\middle\vert\, \nu \in \M(\Omega; \V) \text{ satisfies }\big[r(z)\big](\cdot) = \int_\Theta \psi(\cdot, \theta) d \big[(F_z)_\#\nu\big](\theta) \text{ for all } z \in Z\right\}.
\end{equation}
It would be tempting to come up with a formalization of this type for the concept of hypernetworks, with the family $F_z$ determining the weights appearing in the $\B$ representation. In fact, the set $\B_F$ with the norm $\|\cdot\|_{\B_F}$ is a vector-valued reproducing kernel Banach space. To see that point evaluation is bounded, first, we notice that 
\begin{equation}\label{eq:PFdecreaseTV}|(F_z)_\#\nu|_{\V}(\Theta) \leq |\nu|_{\V}(\Omega)\quad\text{ for all }\nu \in \M(\Omega; \V)\text{ and all }z \in Z.\end{equation}
This follows directly from the definition of the total variations of $\nu$ and $(F_z)_\#\nu$, noticing that any partition of $\Theta= \bigcup_{i=1}^\infty \Theta_n$ induces a partition of $\Omega$ as $\Omega = \bigcup_{i=1}^\infty F_z^{-1}(\Theta_n)$.
Using \eqref{eq:PFdecreaseTV} and that $\B$ is part of an integral vv-RKBS pair, we have
\begin{equation}
    \|r(z)\|_{\B} = \inf \left\{ |\mu|_{\V}(\Theta) \,\middle\vert\, \mu \in \M(\Theta; \V) \text{ satisfies }\big[r(z)\big](\cdot) = \int_\Theta \psi(\cdot, \theta) d \mu(\theta)\right\} \leq \|r\|_{\B_F}.
\end{equation}
However, while Theorem \ref{thm:every_vv_RKBS_corresponds_to_adjoint_vv_RKBS_pair} guarantees the existence of an adjoint space $\B^\di_F$, there is no clear choice of $\B^\di_F$ that makes $(\B_F,\B^\di_F)$ an integral vv-RKBS pair, even if the maps $F_z$ arise from another integral RKBS. We attribute this to the fact that this construction, due to the composition with the functions $F_z$, is closer to a two-layer architecture than a single-layer one. Given these drawbacks, we argue that a more suitable generalization to deep networks would be a combination of the single-layer hypernetwork formalism of Definitions~\ref{def:hypernetwork_RKBS_weight_form} and \ref{def:hypernetwork_RKBS_func_form} with a principled approach dealing with deep architectures, such as the reproducing kernel chains proposed in Heeringa et al.\ \cite{heeringa2025deep}.

\section{Conclusion}
Mathematical analysis of neural networks helps reveal their underlying properties. In particular, studying the function spaces associated with neural networks provides insight into, for instance, how optimization within these spaces yields neural architectures, which algorithms are suited for solving the optimization problem, and what types of generalization guarantees can be given. Many of these neural network function spaces are instances of reproducing kernel Banach spaces (RKBS). While scalar-valued neural networks are well understood within this framework, the theory for $\R^d$-valued neural networks and neural operators is less developed.

To address this gap, we introduce a kernel-based framework for vector-valued RKBS (vv-RKBS) that extends existing constructions. Most known RKBS models for neural networks do not explicitly employ a kernel and instead only rely on the assumption that the point evaluation functionals are bounded. Among the methods that do incorporate a kernel, it is typically assumed to be given a priori.

We propose a new definition of an adjoint pair of vv-RKBS that explicitly incorporates a kernel. Within this framework, we show that every Banach space of functions with bounded point evaluations corresponds to such an adjoint pair, and hence admits an associated kernel. The adjoint pair consists of a dual pair $(\B, \B^\di)$ of Banach spaces of functions, both of which have bounded point evaluations. By combining this duality pairing with the kernel, we obtain a reproducing property and thereby complete the construction of the adjoint pair of vv-RKBS.

In this setting, the kernel is defined as
\begin{equation}
    K \colon X \times \Omega \to \twin(\U, \U^\di)
\end{equation}
where $(\U, \U^\di)$ is a dual pair. Here, $X$ and $\U$ are the domain and codomain of functions in $\B$, while $\Omega$ and $\U^\di$ are the domain and codomain of functions in $\B^\di$. The space $\twin(\mathcal U, \mathcal U^\di)$ generalizes $\mathcal L(\mathcal U)$ as used in the vv-RKHS construction. This kernel definition generalizes existing constructions by allowing asymmetric kernel domains, Banach space–valued outputs that may be infinite-dimensional, and not imposing structural assumptions such as reflexivity or separability on any of the involved spaces. The absence of reflexivity is important, since the neural network function spaces of interest are non-reflexive.

Within the adjoint pair of vv-RKBS framework, we show that many well-known properties of vv-RKHS naturally extend to our setting. In particular, every suitable kernel corresponds to an adjoint pair of vv-RKBS given sufficient structure on the dual pair $(\U, \U^\di)$, and every adjoint pair of vv-RKBS is equivalent to a scalar adjoint pair of vv-RKBS. 

To gain insight into neural networks, we further specialize the adjoint pair construction to integral and neural vv-RKBS, which are the function spaces most relevant to neural networks. We analyze their associated kernel structure and establish a general representer theorem. We then apply the integral and neural vv-RKBS framework to both $\R^d$-valued neural networks and neural operators. For the latter, our focus is on DeepONets and hypernetworks, which have the implicit neural representation (INR) structure underlying many neural operator approaches. We show that the hypernetwork space, which includes DeepONets, admits two complementary interpretations: a weight-space view and a function-space view. Both viewpoints lead to specific optimization problems that share a joint representer theorem and thus have a coinciding sparse solution.

Altogether, this work provides a unifying RKBS viewpoint on $\R^d$-valued neural networks and neural operators. It highlights how Banach-space structures can capture cases beyond Hilbert theory while retaining kernel-based tools, thereby opening the door to deeper mathematical understanding of modern neural network and operator learning methods.

\section{Future work}
Our work considers a general setting and mainly focuses on shallow neural networks rather than deep networks and operators. However, many state-of-the-art architectures are inherently deep. In Heeringa et al.\ \cite{heeringa2025deep}, chaining kernels yields a function space for deep neural networks. Since our work introduces vv-RKBS kernels, it would be natural to explore similar ideas in our setting, potentially leading to a notion of vv-RKBS kernel chaining.

Another direction for future work is related to representer theorems. In the vv-RKHS setting, we study functions $f \colon X \to \U$ taking values in a Hilbert space $\U$. Given data $\{(x_n, v_n)\}_{n=1}^N \subseteq X \times \U$ and a kernel
\begin{equation}
    K \colon X \times X \to \L(\U)
\end{equation}
the representer theorem ensures that the solution to the supervised optimization problem takes the form
\begin{equation}
    f^\dagger(x) = \sum_{n=1}^N K(\cdot, x_n) u_n, \quad u_n \in \U
\end{equation}
The vv-RKBS setting yields a weaker result. For integral and neural vv-RKBS pairs, we have a dual pair of Banach spaces $(\B, \B^\di)$ and $(\U, \U^\di)$, where the space $\B$ consists of functions $f \colon X \to \U$, and $\B^\di$ consists of functions $g \colon \Omega \to \U^\di$. The corresponding kernel is
\begin{equation}
    K \colon X \times \Omega \to \twin(\U, \U^\di)
\end{equation}
In this setting, instead of direct data $\{(x_n, v_n)\}_{n=1}^N \subseteq X \times \U$, we consider data $\{(x_n, y_n)\}_{n=1}^N \subseteq X \times \R^d$. Here, the $y_n$ are related to the $v_n$'s through a measurement operator $M \colon \U \to \R^d$ as $y_n = Mv_n$. Hence, the outputs are always mapped to finite-dimensional measurements for fitting.

For optimization problems over $f \in \B$ with such data $\{(x_n, y_n)\}_{n=1}^N \subseteq X \times \R^d$, the representer theorem then only guarantees representations of the form  
\begin{equation}
    f^\dagger(x) = \sum_{m=1}^{Nd} K_\U(\cdot, w_m) u_m
\end{equation}
Hence, the number of terms is not necessarily bounded by $N$. As argued in this paper, we postulate that this limitation arises because the regularizers used so far do not promote groupwise sparsity, unlike, for instance, the structured sparsity regularizers studied by Bach et al.\ \cite{bach2010structured}. Achieving such improvements would likely require a more explicit use of the vv-RKBS structure rather than general convexity arguments in topological vector spaces as in Bredies and Carioni \cite{bredies2020sparsity}, which in turn must rely on the use of finite-dimensional measurements.

Moreover, our representer theorems are tailored to integral and neural vv-RKBSs. Similar to the work of Wang et al.\ \cite{wang2024sparse}, which investigates when scalar RKBSs admit a representer theorem, it would be interesting to study under what conditions our (adjoint pair of) vv-RKBSs admit one. In particular, an important open question is what additional structure is needed to ensure representations with at most $N$ terms in the expansion. We expect that the groupwise sparsity mentioned above could play a central role.

\textit{\paragraph{Declaration of generative AI use.} In the final stages of the preparation of this work, the authors used AI tools to improve the readability of the manuscript. After using these tools, the authors reviewed and edited the content as needed and take full responsibility for the content of the article.}

\textbf{Funding.} \textit{Sven Dummer acknowledges financial support from EU EFRO OPoost (OOST-00103).}
\newpage

\appendix

\section{Notation}\label{sec:notation}
\begin{table}[h!t]
    \centering
    \begin{tabular}{|C{4cm}|m{11cm}|}
        \hline
         \textbf{\textit{\underline{Notation}}} & \textbf{\textit{\underline{Meaning}}} \\
         \hline \hline
         $|\cdot|_{B}$ & Given a Banach space $B$ with norm $\norm{\cdot}_B$, $|\cdot|_{B}$ denotes the total variation measure of a $B$-valued measure. \\
         \hline 
         $\Omega$ & Weight space of adjoint vv-RKBS pair of functions. Also the weights of the hypernetwork. \\
         \hline 
         $X$ & The input space of adjoint vv-RKBS pair of functions. Also used as the input of the base network in the hypernetwork setting. \\
         \hline 
         $Z$ & The input space of the hypernetwork. \\
         \hline 
         $\Theta$ & The weight space of the base network in the hypernetwork setting. \\
         \hline 
         $\U$ & The output Banach space of a vv-RKBS. \\
         \hline 
         $\U^\di$ & The output Banach space of the adjoint vv-RKBS in an adjoint vv-RKBS pair. \\
         \hline 
         $\B$ & A vv-RKBS. \\
         \hline 
         $\B^\di$ & The adjoint vv-RKBS in an adjoint vv-RKBS pair. \\
         \hline
         $C_0(X; \U)$ & Assuming $X$ is a locally-compact topological space, $f \in C_0(X; \U)$ if it is continuous and for any positive number $\varepsilon >0$ there exists a compact subset $K \subseteq X$ such that $\norm{f(x)} \leq \varepsilon$ for all $x \in X \setminus K$.  \\
         \hline 
         $\M(\Omega)$ & The set of regular signed Radon measures with finite total variation over $\Omega$ equipped with the Borel $\sigma$-algebra. \\
         \hline 
         $\M(\Omega; \V)$ & The set of regular countably additive $\V$-valued vector measures with finite total variation over $\Omega$ equipped with the Borel $\sigma$-algebra. By definition, $\mu \in \M(\Omega; \V)$ is regular if and only if $|\mu|_\V$ is regular.\\
         \hline 
         $\phi \colon X \times \Omega \to \R$ & the integrand for defining the integral / neural vv-RKBS (pairs). \\
         \hline 
         $A_{\Omega \to X}$ & $A_{\Omega \to X} \mu = \int_\Omega \phi(\cdot,w) d \mu(w)$ with $\mu \in \M(\Omega; \U)$. Hence, it means integrating out the $\Omega$ part to get a function of $x \in X$. The Banach space $\U$ is arbitrary. \\
         \hline 
         $\L(\U; \V)$ / $\L(\U)$ & The spaces of bounded linear operators from $\U$ to $\V$ and from $\U$ to $\U$, respectively. \\
         \hline 
         $f$ and $g$ & functions from the vv-RKBS and the adjoint vv-RKBS, respectively. \\
         \hline 
         $\braket{\cdot}{\cdot}_{B^p} = \braket{\cdot}{\cdot}_{B^d,B^p}$ & the dual pairing for a dual pair of Banach spaces $(B^p, B^d)$. Such pairings appear both for explicitly prescribed pairs $(B,B^\di)$ used in the RKBS definitions as well as for the canonical pair $(B,B^\ast)$ with the continuous dual $B^\ast$. In case $B$ is a Hilbert space, the pairing can always be considered as the inner product. The abbreviated notation is used when the dual space $B^d$ is clear from the context, with the subscript denoting the space of the primal element in the second argument. \\
         \hline
         $\mathcal{N}(A)$ & denotes the nullspace of a linear operator $A$. \\
         \hline
         $\twin(\U, \U^\di)$ & the space of twin operators for a dual pair of Banach spaces $(\U, \U^\di)$. \\
         \hline
         $T_\U$ and $T_{\U^\di}$ & the $T_\U$ and $T_{\U^\di}$ operators used in the definition of the space of twin operators. \\
         \hline
         $K \colon X \times \Omega \to \twin(\U, \U^\di)$ & the kernel corresponding to an adjoint pair of vv-RKBS. \\
         \hline 
         $K_\U \colon X \times \Omega \to \L(\U)$ and $K_{\U^\di} \colon X \times \Omega \to \L(\U^\di)$ & the linear operators used in the definition of the space of twin operators corresponding to the twin operator $K(x, w)$. \\
         \hline
         $M:\U \to \R^d$ & Measurement operator used in supervised learning problems.\\
         \hline
    \end{tabular}
    \caption{Meaning of frequently used mathematical objects.}
    \label{tab:notations}
\end{table}
\newpage
\section{Additional proofs and theorems}\label{sec:additionalresults}

\begin{theorem}\label{thm:existence_braket_measure}
    Let $\Sigma_X$ be a $\sigma$-algebra on $X$ and $\Sigma_\Omega$ a $\sigma$-algebra on $\Omega$. Let $(\U, \U^\di)$ be a dual pair of Banach spaces with a continuous dual pairing; that is, there exists a $C > 0$ such that
    \begin{equation}
        |\braket{u^\di}{u}_\U| \leq C \norm{u^\di}_{\U^\di} \norm{u}_\U.
    \end{equation}
    Then for $\mu \in \M(\Omega; \U)$ and $\rho \in \M(X; \U^\di)$, there exists a unique measure $\braket{\rho}{\mu}_\U \in \M(X \times \Omega)$ defined over the $\sigma$-algebra $\sigma(\Sigma_X \times \Sigma_\Omega)$ generated by $\Sigma_X \times \Sigma_\Omega$ such that 
    \begin{equation}
        \braket{\rho}{\mu}_\U(A \times B) = \braket{\rho(A)}{\mu(B)}_\U, \quad \text{for all } A \in \Sigma_X, B\in \Sigma_\Omega.
    \end{equation}
    Moreover, $\braket{\rho}{\mu}_\U$ is of bounded variation, with 
    \begin{equation}
        |\braket{\rho}{\mu}_\U|(X \times \Omega) \leq C |\rho|_{\U^\di}(X) |\mu|_\U(\Omega).
    \end{equation}
\end{theorem}
\begin{proof}
    Let $\mathcal{R}$ denote the algebra generated by sets of the form $A \times B$ with $A \in \Sigma_X$ and $B\in \Sigma_\Omega$. Elements of $\mathcal{R}$ are finite disjoint unions of such rectangles, i.e., sets of the form $\bigcup_{i=1}^n A_i \times B_i$ with $A_i \times B_i \in \Sigma_X \times \Sigma_\Omega$ pairwise disjoint and $n < \infty$. We define a set function on $\mathcal{R}$ by
    \begin{equation}
        \braket{\rho}{\mu}_\U\left(\bigcup_{i=1}^n A_i \times B_i\right) = \sum_{i=1}^n \braket{\rho}{\mu}_\U(A_i \times B_i) = \sum_{i=1}^n \braket{\rho(A_i)}{\mu(B_i)}_\U
    \end{equation}
    which is well-defined independently of the chosen pairwise disjoint decomposition.

    Before applying the Hahn-Kolmogorov extension theorem, we show that $\braket{\rho}{\mu}_\U$ is countably additive on $\mathcal{R}$. First, consider a rectangle $A \times B$ expressed as a countable union of pairwise disjoint rectangles $A_i \times B_i$. Using the partial unions $A_n \coloneqq \bigcup_{i=1}^n A_i$ and $B_n \coloneqq \bigcup_{i=1}^n B_i$, we can decompose $A\times B$ as  
    \begin{equation}
        A \times B = (A_n \times B_n) \cup (A_n \times (B \setminus B_n)) \cup ((A \setminus A_n) \times B_n) \cup ((A \setminus A_n) \times (B \setminus B_n)).
        \label{eq:decomp_of_A_times_B}
    \end{equation}
    This is a decomposition into disjoint sets. Then consider
    \begin{equation}
        \left|\left(\sum_{i=1}^n \braket{\rho}{\mu}_\U(A_i \times B_i) \right) - \braket{\rho}{\mu}_\U(A \times B)  \right|
    \end{equation}
    and note that using \eqref{eq:decomp_of_A_times_B} it can be rewritten to
    \begin{equation}
    \begin{split}
    & =  \left|\braket{\rho}{\mu}_\U (A_n \times (B \setminus B_n)) + \braket{\rho}{\mu}_\U((A \setminus A_n) \times B_n) + \braket{\rho}{\mu}_\U((A \setminus A_n) \times (B \setminus B_n)) \right| \\
        &  = \left|\braket{\rho\left(A_n\right)}{\mu\left(B \setminus B_n\right)}_\U + \braket{\rho\left(A \setminus A_n\right)}{\mu\left(B_n\right)}_\U + \braket{\rho\left(A \setminus A_n\right)}{\mu\left(B \setminus B_n\right)}_\U\right| \\
        & \leq C \left(\norm{\rho\left(A_n\right)}_{\U^\di} \norm{\mu\left(B \setminus B_n\right)}_{\U} + \norm{\rho\left(A_n\right)}_{\U^\di} \norm{\mu\left(B \setminus B_n\right)}_{\U} + \norm{\rho\left(A \setminus A_n\right)}_{\U^\di} \norm{\mu\left(B \setminus B_n\right)}_{\U} \right) 
    \end{split}\label{ineq:proof_count_addit}
    \end{equation}
    By countable additivity of $\rho$ and $\mu$, we have $\rho(A_n) \to \rho(A)$ and $\mu(B_n) \to \mu(B)$, which shows that each bounding term in \eqref{ineq:proof_count_addit} goes to zero. Hence, $\braket{\rho}{\mu}_\U$ is countably additive on rectangles $A \times B$.

    To extend countable additivity to all of $\mathcal{R}$, let $C = \bigcup_{n=1}^\infty D_n \in \mathcal{R}$  be a countable union of disjoint sets $D_n \in \mathcal{R}$. As any set in $\mathcal{R}$ can be written as a finite disjoint union of rectangles, the sets $C$ and $D_n$ can be written as $C = \bigcup_{j=1}^N C_j$ and $D_n = \bigcup_{i=1}^{M_n} D_{n,i}$ with both $\{C_j\}_{j=1}^N$ and $\{D_{n, i}\}_{i=1}^{M_n}$ finite sets of disjoint rectangles. Given that the $D_n$ are pairwise disjoint, define the pairwise disjoint rectangles $D_{n, i, j} \coloneqq D_{n,i} \cap C_j$. Then
    \begin{equation}
    \begin{split}
        \braket{\rho}{\mu}_\U(C) & = \sum_{j=1}^N \braket{\rho}{\mu}_\U(C_j) = \sum_{j=1}^N  \sum_{n=1}^\infty \sum_{i=1}^{M_n}\braket{\rho}{\mu}_\U(D_{n, i, j}) = \sum_{n=1}^\infty \sum_{i=1}^{M_n} \sum_{j=1}^N  \braket{\rho}{\mu}_\U(D_{n, i, j}) \\
        & = \sum_{n=1}^\infty \sum_{i=1}^{M_n} \braket{\rho}{\mu}_\U(D_{n,i}) = \sum_{n=1}^\infty \braket{\rho}{\mu}_\U(D_n),
    \end{split}
    \end{equation}
    which proves countable additivity on $\mathcal{R}$.  

    Now we construct the extension. Let
    \begin{equation}
        \nu_\mathcal{R}^{+}  = 0.5 \left(\left|\braket{\rho}{\mu}_\U\right| + \braket{\rho}{\mu}_\U \right), \qquad
        \nu_\mathcal{R}^{-}  = 0.5 \left(\left|\braket{\rho}{\mu}_\U\right| - \braket{\rho}{\mu}_\U \right), 
    \end{equation}
    with $\left|\braket{\rho}{\mu}_\U\right|$ the total variation on $\mathcal{R}$. These are well-defined and finite by the following total variation estimate, 
    \begin{equation}
        \begin{aligned}
            |\braket{\rho}{\mu}_\U|(X \times \Omega) 
            & = \sup_{\set{A_\bullet \times B_\bullet}} \sum_i |\braket{\rho}{\mu}_\U(A_i \times B_i)| \\&= \sup_{\set{A_\bullet \times B_\bullet}} \sum_i |\braket{\rho(A_i)}{\mu(B_i)}_\U| \\
            & \leq C \sup_{\set{A_\bullet \times B_\bullet}} \sum_i \norm{\rho(A_i)}_{\U^\di}\norm{\mu(B_i)}_\U \\
            &\leq C \sup_{\set{A_\bullet \times B_\bullet}} \sum_i |\rho|_{\U^\di}(A_i)|\mu|_\U(B_i) \\
            &= C \sup_{\set{A_\bullet \times B_\bullet}} \sum_i \left(|\rho|_{\U^\di}
            \otimes|\mu|_\U\right)(A_i\times B_i) \\
            & = C \left(|\rho|_{\U^\di}\otimes|\mu|_{\U}\right)(X\times\Omega) \\
            & = C |\rho|_{\U^\di}(X) |\mu|_{\U}(\Omega)
        \end{aligned}
        \label{ineq:tv_ineq_extension}
    \end{equation}
    where the suprema run over all the partitions $\set{A_\bullet \times B_\bullet}$ of $X\times\Omega$.
    Moreover, since $\braket{\rho}{\mu}_\U$ is countably additive and of bounded variation, its total variation $|\braket{\rho}{\mu}_\U|$ is countably additive on $\mathcal{R}$. All together, $\nu_\mathcal{R}^+$ and $\nu_\mathcal{R}^-$ are nonnegative, finite, countably additive set functions on $\mathcal{R}$, and by the Hahn-Kolmogorov extension theorem (see Theorem 1.7.8 in Tao \cite{tao2011introduction}), these extend to countable additive measures $\nu^+$ and $\nu^{-}$ on $\sigma(\Sigma_X \times \Sigma_\Omega)$. Defining $\nu = \nu^+ - \nu^-$, we obtain a finite signed measure such that
    \begin{equation}
        \nu(A \times B) = \braket{\rho(A)}{\mu(B)}_\U, \quad \text{for all } A\in \Sigma_X, B \in \Sigma_\Omega.
    \end{equation}

    Uniqueness of the extension $\nu$ follows from Dynkin's lemma. If $\widetilde{\nu}$ is another extension, define
    \begin{equation}
        \mathcal{D} \coloneqq \{ S \in \sigma(\Sigma_X \times \Sigma_\Omega) \mid \widetilde{\nu}(S) = \nu(S) \}.
    \end{equation}
    Note that $X \times \Omega \in \mathcal{D}$. Since $\nu$ is finite, for any $A \subseteq B$ with $A, B \in \mathcal{D}$, we have
    \begin{equation}
        \nu(B \setminus A) = \nu(B) - \nu(A) = \widetilde{\nu}(B) - \widetilde{\nu}(A) = \widetilde{\nu}(B \setminus A)
    \end{equation}
    which shows that $B \setminus A \in \mathcal{D}$. Next, consider an increasing sequence $A_1 \subseteq A_2 \subseteq A_3 \subseteq \ldots$ with $A_n \in \mathcal{D}$. Define
    \begin{equation}
        \widetilde{A}_n \coloneqq A_n \setminus \bigcup_{i=1}^{n-1} A_i = A_n \setminus A_{n-1}, \qquad \widetilde{A}_1 = A_1.
    \end{equation}
    Then $\widetilde{A}_1 = A_1 \in \mathcal{D}$ and by
    \begin{equation}
        \nu(\widetilde{A}_n) = \nu(A_n) - \nu(A_{n-1}) = \widetilde{\nu}(A_n) - \widetilde{\nu}(A_{n-1}) = \widetilde{\nu}(\widetilde{A}_n) 
    \end{equation}
    we have $\widetilde{A}_n \in \mathcal{D}$ for $n \geq 2$. By countable additivity, we have
    \begin{equation}
        \nu\left(\bigcup_{n=1}^\infty A_n \right) = \nu\left(\bigcup_{n=1}^\infty \widetilde{A}_n \right) = \sum_{n=1}^\infty \nu(\widetilde{A_n}) = \sum_{n=1}^\infty \widetilde{\nu}(\widetilde{A_n}) =  \widetilde{\nu}\left(\bigcup_{n=1}^\infty \widetilde{A}_n \right) =  \widetilde{\nu}\left(\bigcup_{n=1}^\infty A_n \right)  
    \end{equation}
    These properties verify that $\mathcal{D}$ is a Dynkin system containing the $\pi$-system $\Sigma_X \times \Sigma_\Omega$. Hence, by Dynkin's Lemma 
    \begin{equation}
        \sigma(\Sigma_X \times \Sigma_\Omega) \subseteq \mathcal{D} \subseteq \sigma(\Sigma_X \times \Sigma_\Omega) \implies \mathcal{D} = \sigma(\Sigma_X \times \Sigma_\Omega)
    \end{equation}
    It follows that $\widetilde{\nu} = \nu$, and hence $\nu$ is the unique extension.

    Finally, we show that $\nu$ is of bounded variation:
    \begin{equation}
    \begin{split}
        |\nu|(X \times \Omega) & = |\nu^+ - \nu^-|(X \times \Omega) \leq |\nu^+|(X \times \Omega) + | \nu^-|(X \times \Omega) = \nu^+(X \times \Omega) + \nu^-(X \times \Omega) \\
        & = 0.5\left(|\braket{\rho}{\mu}_\U|(X \times \Omega) + \braket{\rho}{\mu}_\U(X \times \Omega)\right) + 0.5\left(|\braket{\rho}{\mu}_\U|(X \times \Omega) - \braket{\rho}{\mu}_\U(X \times \Omega)\right) \\
        & = |\braket{\rho}{\mu}_\U|(X \times \Omega) \leq C |\rho|_{\U^\di}(X) |\mu|_{\U}(\Omega)
    \end{split}
    \end{equation}
    where the final inequality follows from \eqref{ineq:tv_ineq_extension}. 
\end{proof}

\begin{lemma}
    Let $X$ and $\Omega$ be topological spaces and $\phi \in C_0(X \times \Omega)$. Then for any $\delta > 0$ and $\varepsilon > 0$, there exist:
    \begin{itemize}
        \item pairwise disjoint Borel sets $A_1, \dots, A_n \subseteq X$ with $\bigcup_{i=1}^n A_i = X$,
        \item pairwise disjoint Borel sets $B_1, \dots, B_m \subseteq \Omega$ with $\bigcup_{j=1}^m B_j = \Omega$,
        \item a Borel set $D_\delta \subseteq X \times \Omega$, 
    \end{itemize}
    such that either:
    \begin{itemize}
        \item $A_i \times B_j \subseteq D_\delta$ and $|\phi(x,w) - \phi(\widetilde{x}, \widetilde{w})| \leq \epsilon$ for all $(x,w), (\widetilde{x}, \widetilde{w}) \in A_i \times B_j$, or
        \item $A_i \times B_j \subseteq \left( X \times \Omega \right) \setminus D_\delta$ and $|\phi(x,w)| \leq \delta$ for all $(x,w) \in A_i \times B_j$.
    \end{itemize}
    \label{lemma:C0_domain_decomposition}
\end{lemma}

\begin{proof}
    Since $\phi \in C_0(X \times \Omega)$, for every $\delta > 0$ there exists a compact set $\widetilde{D}_\delta \subseteq X \times \Omega$ such that $|\phi(x,w)| \leq \delta$ outside $\widetilde{D}_\delta$. By continuity, for each $(x,w) \in \widetilde{D}_\delta$, there exists an open neighborhood $O_{x,w}$ with $|\phi(x, w) - \phi(\widetilde{x}, \widetilde{w})| \leq \varepsilon/2$ for all $(\widetilde{x}, \widetilde{w}) \in O_{x,w}$. Since open rectangles form a basis for the product topology, we may assume $O_{x,w} = \widetilde{A}_{x,w} \times \widetilde{B}_{x,w}$ for open sets $\widetilde{A}_{x,w} \subseteq X$ and $\widetilde{B}_{x,w} \subseteq \Omega$.

    The collection $\{\widetilde{A}_{x,w} \times \widetilde{B}_{x,w}\}_{(x,w) \in \widetilde{D}_\delta}$ is an open cover of the compact set $\widetilde{D}_\delta$, so it has a finite subcover $\{\widetilde{A}_i \times \widetilde{B}_i\}_{i=1}^{n_c} \coloneqq \{\widetilde{A}_{x_i,w_i} \times \widetilde{B}_{x_i,w_i}\}_{i=1}^{n_c}$, with $(x_i, w_i) \in \widetilde{A}_i \times \widetilde{B}_i$, and
    \begin{equation}
    |\phi(x_i, w_i) - \phi(x, w)| \leq \varepsilon/2 \quad \text{for all } (x,w) \in \widetilde{A}_i \times \widetilde{B}_i.
    \end{equation}
    
    Define 
    \begin{equation}
        D_\delta \coloneqq \bigcup_{i=1}^{n_c} \widetilde{A}_i \times \widetilde{B}_i \supseteq \widetilde{D}_\delta
    \end{equation}
    Moreover, define the finite collections
    \begin{equation}
    \mathcal{S}_A := \left\{ \bigcap_{i=1}^{n_c} E_i \ \middle|\ E_i \in \{\widetilde{A}_i, \widetilde{A}_i^c\} \right\} \setminus \{\emptyset\}, \quad
    \mathcal{S}_B := \left\{ \bigcap_{i=1}^{n_c} F_i \ \middle|\ F_i \in \{\widetilde{B}_i, \widetilde{B}_i^c\} \right\} \setminus \{\emptyset\}.
    \end{equation}
    Each $A \in \mathcal{S}_A$ is either fully contained in or disjoint from each $\widetilde{A}_i$, and similarly for $B \in \mathcal{S}_B$ and each $\widetilde{B}_i$. For any $A \in \mathcal{S}_A$ and $B \in \mathcal{S}_B$, define
    \begin{equation}
    I_{A,B} := \{ i \in \{1, \dots, n_c\} \mid A \subseteq \widetilde{A}_i,\, B \subseteq \widetilde{B}_i \}.
    \end{equation}
    \begin{itemize}
    \item If $I_{A,B} \neq \emptyset$, then $A \times B \subseteq \widetilde{A}_i \times \widetilde{B}_i \subseteq D_\delta$ for some $i$, and for all $(x,w), (\widetilde{x}, \widetilde{w}) \in A \times B$,
    \begin{equation}
    |\phi(x,w) - \phi(\widetilde{x}, \widetilde{w})| \leq |\phi(x,w) - \phi(x_i, w_i)| + |\phi(x_i, w_i) - \phi(\widetilde{x}, \widetilde{w})| \leq \varepsilon.
    \end{equation}
    \item If \( I_{A,B} = \emptyset \), then \( A \times B \) is disjoint from all sets \( \widetilde{A}_i \times \widetilde{B}_i \). Since  
    \begin{equation}
    \bigcup_i \widetilde{A}_i \times \widetilde{B}_i = D_\delta ,
    \end{equation}  
    it follows that  
    \begin{equation}
    A \times B \subseteq (X \times \Omega) \setminus D_\delta \subseteq (X \times \Omega) \setminus \widetilde{D}_\delta ,
    \end{equation}  
    where the last inclusion holds because \( \widetilde{D}_\delta \subseteq D_\delta \). Consequently, \( |\phi(x,w)| \le \delta \) on \( A \times B \).
    \end{itemize}
    Let $\mathcal{S}_A = (A_1, \dots, A_n)$ and $\mathcal{S}_B = (B_1, \dots, B_m)$ be ordered versions of the previously defined collections. $A_i$ are pairwise disjoint with $\bigcup_{i=1}^n A_i = X$, and the sets $B_j$ are pairwise disjoint with $\bigcup_{j=1}^m B_j = \Omega$. Together with $D_\delta$, they satisfy the conditions stated in the theorem.
\end{proof}

\begin{theorem}\label{thm:simple_approximation_of_C0_functions}
    Let $X$ and $\Omega$ be topological spaces and $\phi \in C_0(X \times \Omega)$. Then for any $\varepsilon > 0$, there exist:
    \begin{itemize}
        \item pairwise disjoint Borel sets $A_1, \dots, A_n \subseteq X$ with $\bigcup_{i=1}^n A_i = X$,
        \item pairwise disjoint Borel sets $B_1, \dots, B_m \subseteq \Omega$ with $\bigcup_{j=1}^m B_j = \Omega$,
    \end{itemize}
    and a simple function of the form
    \begin{equation}
        \phi_{\varepsilon}(x,w) = \sum_{i=1}^n \sum_{j=1}^m a_{ij} \mathbbm{1}_{A_i \times B_j}(x,w) 
    \end{equation}
    such that 
    \begin{equation}
        \sup_{(x,w) \in X \times \Omega} \left| \phi_{\varepsilon}(x,w) - \phi(x,w) \right| \leq \varepsilon
    \end{equation}
\end{theorem}

\begin{proof}
    Define the sets $\{A_1, \ldots, A_n\}$, $\{B_1, \ldots, B_m\}$, and $D_\varepsilon$ as in Lemma \ref{lemma:C0_domain_decomposition} with $\delta = \varepsilon$. For each $i \in \{1, \ldots, n\}$ and $j \in \{1, \ldots, m\}$, define
    \begin{equation}
    a_{ij} := \begin{cases}
    \phi(\widetilde{x},\widetilde{w}) & \text{for some } (\widetilde{x},\widetilde{w}) \in A_i \times B_j \text{ if } A_i \times B_j \subseteq D_\varepsilon, \\
    0 & \text{otherwise.}
    \end{cases}
    \end{equation}
    Define the simple function
    \begin{equation}
    \phi_{\varepsilon}(x,w) := \sum_{i=1}^n \sum_{j=1}^m a_{ij} \mathbbm{1}_{A_i \times B_j}(x,w).
    \end{equation}
    Note that the sets $A_i$ are pairwise disjoint with $\bigcup_{i=1}^n A_i = X$, and similarly the sets $B_j$ are pairwise disjoint with $\bigcup_{j=1}^m B_j = \Omega$. It follows that
    \begin{equation}
    X \times \Omega = \bigcup_{i=1}^n \bigcup_{j=1}^m A_i \times B_j.
    \end{equation}
    As the sets $A_i \times B_j$ are disjoint, each $(x,w) \in X \times \Omega$ belongs to a unique such set, and thus $\phi_{\varepsilon}(x,w) = a_{ij}$ for the corresponding pair $(i,j)$.
    
    If $(x,w) \in A_i \times B_j \subseteq D_\varepsilon$, then $a_{ij} = \phi(\widetilde{x}, \widetilde{w})$ for some $(\widetilde{x}, \widetilde{w}) \in A_i \times B_j \subseteq D_\varepsilon$, and
    \begin{equation}
    |\phi_{\varepsilon}(x,w) - \phi(x,w)| = |a_{ij} - \phi(x,w)| = |\phi(\widetilde{x},\widetilde{w}) - \phi(x,w)| \leq \varepsilon,
    \end{equation}
    Otherwise, $(x,w) \in (X \times \Omega) \setminus D_\varepsilon$, and since $a_{ij} = 0$, we have
    \begin{equation}
    |\phi_{\varepsilon}(x,w) - \phi(x,w)| = |a_{ij} - \phi(x,w)| = |\phi(x,w)| \leq \varepsilon.
    \end{equation}
    In either case,
    \begin{equation}
    \sup_{(x,w) \in X \times \Omega} |\phi_{\varepsilon}(x,w) - \phi(x,w)| \leq \varepsilon,
    \end{equation}
    completing the approximation.   
\end{proof}

\newpage
\bibliographystyle{plain}
\bibliography{references}

\end{document}